\documentclass[oneside,english]{amsart}
\usepackage[T3,T1]{fontenc}
\usepackage[latin9]{inputenc}
\usepackage{color}
\usepackage{babel}
\usepackage{mathrsfs}
\usepackage{mathtools}
\usepackage{amsbsy}
\usepackage{amstext}
\usepackage{amsthm}
\usepackage{amssymb}
\usepackage{cancel}
\usepackage{mathdots}
\usepackage[pdfusetitle,
 bookmarks=true,bookmarksnumbered=true,bookmarksopen=true,bookmarksopenlevel=3,
 breaklinks=false,pdfborder={0 0 1},backref=false,colorlinks=true]
 {hyperref}
\hypersetup{
 pdfborderstyle=,pdfborderstyle={},pdfborderstyle={},pdfborderstyle={},pdfborderstyle={},pdfborderstyle={},pdfborderstyle={},pdfborderstyle={},pdfborderstyle={},pdfborderstyle={},pdfborderstyle={},pdfborderstyle={},pdfborderstyle={},pdfborderstyle={},pdfborderstyle={},pdfborderstyle={},pdfborderstyle={},pdfborderstyle={},pdfborderstyle={},pdfborderstyle={},pdfpagelayout=OneColumn,pdfnewwindow=true,pdfstartview=XYZ,plainpages=false,linkcolor=blue,urlcolor=blue,citecolor=red,anchorcolor=blue,linkcolor=blue,urlcolor=blue,citecolor=red,anchorcolor=blue}

\makeatletter
\numberwithin{equation}{section}
\numberwithin{figure}{section}
\theoremstyle{plain}
\newtheorem{thm}{\protect\theoremname}
\theoremstyle{plain}
\newtheorem{prop}{\protect\propositionname}
\theoremstyle{definition}
\newtheorem{defn}{\protect\definitionname}
\theoremstyle{remark}
\newtheorem{rem}{\protect\remarkname}
\theoremstyle{definition}
\newtheorem{example}{\protect\examplename}
\theoremstyle{definition}
\newtheorem{problem}{\protect\problemname}
\theoremstyle{plain}
\newtheorem{cor}{\protect\corollaryname}
\theoremstyle{plain}
\newtheorem{lem}{\protect\lemmaname}
\theoremstyle{remark}
\newtheorem{claim}{\protect\claimname}

\usepackage{babel}
\usepackage{enumerate}
\usepackage{verbatim}
\usepackage{mathrsfs}
\usepackage{cite}
\usepackage{bbm}

\mathtoolsset{showonlyrefs}

\usepackage{physics}
\usepackage{tikz}
\usepackage{yhmath}
\usepackage{siunitx}
\usepackage{array}
\usepackage{multirow}
\usepackage{gensymb}
\usepackage{tabularx}
\usepackage{extarrows}
\usepackage{booktabs}
\usetikzlibrary{fadings}
\usetikzlibrary{patterns}
\usetikzlibrary{shadows.blur}
\usetikzlibrary{shapes}
\usetikzlibrary{arrows.meta}

\tikzset{
    ShortArrow/.tip={Triangle[length=2pt, width=3pt]},
    >=ShortArrow
}

\providecommand{\keywords}[1]{\textbf{Index terms---} #1}

\DeclareSymbolFont{tipa}{T3}{cmr}{m}{n}
\DeclareMathAccent{\inbreve}{\mathalpha}{tipa}{16}

\newcommand{\subsetsim}{\mathrel{%
  \ooalign{\raise0.2ex\hbox{$\subset$}\cr\hidewidth\raise-0.8ex\hbox{\scalebox{0.9}{$\sim$}}\hidewidth\cr}}}
\newcommand{\supsetsim}{\mathrel{%
  \ooalign{\raise0.2ex\hbox{$\supset$}\cr\hidewidth\raise-0.8ex\hbox{\scalebox{0.9}{$\sim$}}\hidewidth\cr}}}

\newcommand{\subsetapprox}{\mathrel{%
  \ooalign{\raise0.4ex\hbox{$\subset$}\cr\hidewidth\raise-0.8ex\hbox{\scalebox{0.9}{$\approx$}}\hidewidth\cr}}}
\allowdisplaybreaks[1]
\newcommand{\bone}{\mathbbm{1}}

\def\area{\mathrm{area}}
\def\M{\mathcal{M}}
\def\X{\mathcal{X}}
\def\P{\mathscr{P}}

\def\W{\mathsf{W}_{2}}

\def\L{\mathsf{L}} 

\def\1{\mathbf{1}}

\def\d{{\text {\rm d}}}

\def\Bern{\mathrm{Bern}}

\theoremstyle{definition}
\theoremstyle{definition}
\newtheorem{condition}{Condition}

\providecommand{\corollaryname}{Corollary}
\providecommand{\lemmaname}{Lemma}
\providecommand{\propositionname}{Proposition}
\providecommand{\remarkname}{Remark}
\providecommand{\theoremname}{Theorem}

\providecommand{\definitionname}{Definition}

\providecommand{\claimname}{Claim}

\providecommand{\examplename}{Example}

\providecommand{\problemname}{Problem}

\makeatother

\providecommand{\claimname}{Claim}
\providecommand{\corollaryname}{Corollary}
\providecommand{\definitionname}{Definition}
\providecommand{\examplename}{Example}
\providecommand{\lemmaname}{Lemma}
\providecommand{\problemname}{Problem}
\providecommand{\propositionname}{Proposition}
\providecommand{\remarkname}{Remark}
\providecommand{\theoremname}{Theorem}

\begin{document}
\title[Large Deviations Principle for Isoperimetry]{Large Deviations Principle for Isoperimetry and Its Equivalence to
Nonlinear Log-Sobolev Inequalities }
\author{Lei Yu}
\thanks{L. Yu is with the School of Statistics and Data Science, LPMC, KLMDASR,
and LEBPS, Nankai University, Tianjin 300071, China (e-mail: leiyu@nankai.edu.cn). }
\begin{abstract}
The isoperimetric problem is a classic topic in geometric measure
theory, yet critical questions regarding the characterization of optimal
solutions---even asymptotically optimal ones---remain largely unresolved.
In this paper, we investigate the large deviations asymptotics (governing
sequences of exponentially small sets) for the isoperimetric problem
on the product Riemannian manifold $\M^{n}$ endowed with the product
probability measure $\nu^{\otimes n}$, where $(\M,\nu)$ is a weighted
Riemannian manifold with nonnegative Bakry--Émery--Ricci curvature.
We establish an exact characterization of the large deviations asymptotics
of the isoperimetric profile, which reveals a precise equivalence
between this asymptotic isoperimetric inequality and an established
strengthening of log-Sobolev inequality---namely, the nonlinear log-Sobolev
inequality. It is observed that conditional typical sets (i.e., $\mathbf{f}$-coordinate
half-spaces for some vector $\mathbf{f}$ of functions), a fundamental
concept from information theory, form an asymptotically optimal solution
to the isoperimetric problem. This class of subsets further yields
an upper bound on the isoperimetric profile in the central limit regime
(concerning sets of constant weighted volume). Although our results
are stated for product spaces, they imply certain isoperimetric inequalities
for non-product spaces, e.g., they can be used to recover the weaker
equivalence established by Ledoux and Bobkov for arbitrary non-product
spaces or used to establish quantitative relations among the optimal
constants in isoperimetric, concentration and transport inequalities
for product or non-product spaces. Our results provide a rigorous
justification from the perspective of nonlinear log-Sobolev inequalities
for why isoperimetric minimizers behave fundamentally differently
across spaces with distinct geometric structures. Our proof idea is
a new framework which integrates tools from information theory, optimal
transport, and geometric measure theory. 
\end{abstract}


\maketitle

\tableofcontents{}

\section{Introduction}

In a metric-measure space, the isoperimetric problem seeks the subset
of fixed measure that minimizes surface area---where surface area
is rigorously defined via the boundary measure. This classic problem
remains largely unresolved in both its exact and asymptotic formulations,
with the explicit characterization of the isoperimetric profile (the
minimal surface area corresponding to each fixed measure) and its
asymptotic behavior posing a widely open question, save for a handful
of specific cases: e.g., the standard Gaussian measure, log-concave
measures on the real line, and Euclidean spheres endowed with constant
density. In this paper, we study the isoperimetric problem across
various asymptotic regimes, with particular emphasis on the large
deviations regime that considers sequences of exponentially small
sets. We formalize our core analytical setup as follows.

Let $(\M,g)$ be a $k$-dimensional smooth complete oriented connected
Riemannian manifold with induced geodesic distance $d$. We assume
$k\ge1$ and $(\M,g)=(\mathbb{R},|\cdot|)$ if $k=1$. Let $\nu=e^{-V}\mathrm{vol}$
with $V\in C^{2}(\M)$ be a reference Borel probability measure on
$\M$ that is absolutely continuous with respect to the Riemannian
volume $\mathrm{vol}$. We consider the $n$-fold product space $\M^{n}$,
equipped with the product measure $\nu_{n}:=\nu^{\otimes n}$ and
the product distance $d_{n}(\mathbf{x},\mathbf{y})=\sqrt{\sum_{i=1}^{n}d(x_{i},y_{i})^{2}}$,
where $\mathbf{x}=(x_{1},x_{2},...,x_{n})$.

For a set $A\subseteq\M^{n}$, denote its $r$-enlargement as 
\begin{equation}
A^{r}:=\bigcup_{\mathbf{x}\in A}\{\mathbf{y}\in\M^{n}:d_{n}(\mathbf{x},\mathbf{y})\leq r\}.\label{eq:Gamma-1}
\end{equation}
For a Borel set $A\subseteq\M^{n}$, its boundary measure (i.e., perimeter)
is defined by 
\begin{equation}
\nu_{n}^{+}(A):=\liminf_{r\downarrow0}\frac{\nu_{n}(A^{r})-\nu_{n}(A)}{r}.\label{eq:perimeter}
\end{equation}
In the isoperimetric problem, the objective is to minimize $\nu_{n}^{+}(A)$
over all sets $A$ of a given measure, i.e., to determine the isoperimetric
profile: For $a\in[0,1]$, 
\begin{equation}
I_{n}(a):=\inf_{A:\nu_{n}(A)=a}\nu_{n}^{+}(A),\label{eq:Gamma}
\end{equation}
where $A$ is a Borel set. Note that there are several other ways
in the literature to define the perimeter of a set, but they coincide
when the set admits smooth boundary, and thus, lead to the same isoperimetric
profile in our setup; see details in \cite[pp. 32-33]{bayle2003proprietes,ambrosio2017perimeter}.

Exactly characterizing the isoperimetric profile is in general extremely
difficult, and remains widely open except for certain specific cases.
For example, when $\nu$ is the standard Gaussian measure, the isoperimetric
profile $I_{n}(a)=I_{\mathrm{G}}(a):=\phi(\Phi^{-1}(a))$ for all
$n\ge1$ with half-spaces as isoperimetric minimizers \cite{sudakov1978extremal,borell1975brunn}.
Here, $\phi$ and $\Phi$ are respectively the pdf and cdf of the
standard Gaussian. It is widely known that $I_{\mathrm{G}}(a)\sim a\sqrt{2\ln\frac{1}{a}}$
as $a\to0$. That is, the asymptotic isoperimetric ratio for the Gaussian
measure is 
\begin{equation}
\lim_{a\to0}\frac{I_{\mathrm{G}}(a)}{a\sqrt{\ln\frac{1}{a}}}=\sqrt{2}.\label{eq:Gaussian}
\end{equation}
Another example is log-concave measures on the real line (i.e., $n=1$),
for which the isoperimetric profile is attained by intervals or half-lines
\cite{bobkov1996extremal}. Furthermore, in the Euclidean sphere of
a given dimension equipped with constant density, the isoperimetric
profile is known as well, which is attained by spherical caps.

As noted at the outset of this paper, our primary aim is to study
the asymptotic behavior of the isoperimetric profile, in particular
its large deviations asymptotics. We focus on this behavior in the
regime where the probability measure of the set decays exponentially
with increasing dimension, and to this end we identify the following
limits: for $\alpha>0$, 
\begin{align}
\underline{\Lambda}(\alpha) & :=\liminf_{n\to\infty}\frac{I_{n}(e^{-n\alpha})}{e^{-n\alpha}\sqrt{n}},\label{eq:-82}
\end{align}
and $\overline{\Lambda}(\alpha)$ which is defined similarly but with
$\liminf_{n\to\infty}$ replaced by $\limsup_{n\to\infty}$. We denote
them by $\Lambda(\alpha)$ if they are equal.

\subsection{Main Results}

We next provide a detailed description of our main results.

\subsubsection{Large Deviations Principle }

A natural convexity condition on a manifold with density is the curvature-dimension
condition $\mathrm{CD}(0,\infty)$.

\begin{condition}[Nonnegative Bakry--Émery--Ricci Curvature]\label{cond:convexity}
We assume that 
\[
\mathrm{Ric}_{V}:=\mathrm{Ric}_{g}+\mathrm{Hess}_{g}V\geq0\;\;\text{ as 2-tensor fields},
\]
where $\mathrm{Ric}_{g}$ is the Ricci curvature and $\mathrm{Hess}_{g}V$
is the Hessian of $V$. \end{condition}

Since $I_{n}(a)/a$ is non-increasing in $a$ under Condition \ref{cond:convexity}
(see e.g. \cite[Proposition 3.1]{milman2010isoperimetric}), $\Lambda$
is non-decreasing. 
\begin{prop}
Under Condition \ref{cond:convexity}, $\Lambda$ is non-decreasing. 
\end{prop}

Let $\P_{2}^{\mathrm{ac}}(\M)$ be the set of probability measures
which are absolutely continuous with respect to $\nu$ and have finite
second-order moments. We also need the concept of nonlinear log-Sobolev
inequality and a generalized version of half-spaces. 
\begin{defn}
\label{def:nonlinearLS} For a convex function $\theta:[0,\infty]\to[0,\infty]$,
we say a probability measure $\nu$ admits the $\theta$-nonlinear
log-Sobolev inequality if 
\begin{equation}
\theta(D(\mu\|\nu))\leq I(\mu\|\nu),\;\forall\mu\in\P_{2}^{\mathrm{ac}}(\M),\label{eq:-74}
\end{equation}
where for $\mu=\rho\nu$, 
\[
D(\mu\|\nu):=\int\rho\ln\rho\d\nu\;\textrm{ and }\;I(\mu\|\nu):=\int\frac{|\nabla\rho|^{2}}{\rho}\d\nu
\]
are respectively the (relative) entropy and Fisher (relative) information. 
\end{defn}

\begin{rem}
The inequality in \eqref{eq:-74} holds for all $\mu\in\P_{2}^{\mathrm{ac}}(\M)$
if and only if it holds for all $\mu\in\P^{\mathrm{ac}}(\M)$ (the
set of probability measures which are absolutely continuous with respect
to $\nu$) \cite{otto2000generalization}. 
\end{rem}

The nonlinear log-Sobolev inequality, also known as the entropy-energy
inequality or the entropy isoperimetric inequality, can be seen as
a strengthening of the standard (linear) log-Sobolev inequality since
it provides more information about the tradeoff between entropy and
Fisher information. This concept was investigated in several important
works in the literature, e.g., Costa and Cover \cite{costa1984similarity},
Bakry \cite{bakry2004functional}, F. Y. Wang \cite{wang2008generalized},
Bakry, Gentil, and Ledoux \cite{bakry2013analysis}, as well as Polyanskiy
and Samorodnitsky \cite{polyanskiy2019improved}.

For a convex $\theta$, a $\theta$-nonlinear log-Sobolev inequality
holds only if $\theta(0)=0$ and $\theta$ is non-decreasing. In particular,
when $\theta(t)=Kt$ for some constant $K>0$, then the $\theta$-nonlinear
log-Sobolev inequality reduces to the standard log-Sobolev inequality.
Define 
\[
\Theta(\alpha):=\inf_{\mu\in\P_{2}^{\mathrm{ac}}(\M):D(\mu\|\nu)\ge\alpha}I(\mu\|\nu).
\]
Then, given the probability measure $\nu$, the optimal convex function
$\theta$ in the $\theta$-nonlinear log-Sobolev inequality in \eqref{eq:-74}
is the lower convex envelope of $\Theta$, given by 
\begin{align}
\breve{\Theta}(\alpha) & :=\inf_{\substack{\lambda\in[0,1],\alpha_{1},\alpha_{2}\ge0:\\
\lambda\alpha_{1}+(1-\lambda)\alpha_{2}\ge\alpha
}
}\lambda\Theta(\alpha_{1})+(1-\lambda)\Theta(\alpha_{2})\nonumber \\
 & =\inf_{\substack{\lambda\in[0,1],\mu_{1},\mu_{2}\in\P_{2}^{\mathrm{ac}}(\M):\\
\lambda D(\mu_{1}\|\nu)+(1-\lambda)D(\mu_{2}\|\nu)\ge\alpha
}
}\lambda I(\mu_{1}\|\nu)+(1-\lambda)I(\mu_{2}\|\nu).\label{eq:Theta_LCE}
\end{align}
It is obvious that $\breve{\Theta}(0)=0$ and $\breve{\Theta}$ is
non-decreasing and convex.

The second concept we need is a generalized version of half-spaces
which incorporates several important concepts as special cases, e.g.,
the concept of typical sets from information theory. 
\begin{defn}
Given $n\in\mathbb{N}$ and $\mathbf{f}=(f_{1},f_{2},...,f_{n})$
with each $f_{i}:\M\to\mathbb{R}$, the $\mathbf{f}$-coordinate half-space
is defined as 
\begin{equation}
H_{t}(\mathbf{f}):=\Big\{\mathbf{x}\in\M^{n}:\sum_{i=1}^{n}f_{i}(x_{i})\ge t\Big\}\label{eq:halfspace}
\end{equation}
for a real number $t$. In particular, when $f_{i}=\ln\frac{\d\mu_{i}}{\d\nu}$
for some probability measures $\mu_{i},i\in[n]$ and $t=\sum_{i=1}^{n}D(\mu_{i}\|\nu)-n\epsilon$
for some $\epsilon>0$, the $\mathbf{f}$-coordinate half-space reduces
to the conditional (relative entropy) typical set of $(\mu_{i})_{i\in[n]}$
w.r.t. $\nu$ which is given by 
\[
A_{n,\epsilon}((\mu_{i})_{i\in[n]}\|\nu):=\Big\{\mathbf{x}\in\M^{n}:\sum_{i=1}^{n}\ln\frac{\d\mu_{i}}{\d\nu}(x_{i})\ge\sum_{i=1}^{n}D(\mu_{i}\|\nu)-n\epsilon\Big\}.
\]
\end{defn}

\begin{rem}
The concept of $\mathbf{f}$-coordinate half-space reduces to the
Orlicz ball if $-f_{i}$ are Young functions (and $-t>0$). However,
we do not pose any restriction on functions $f_{i}$ here (except
for measurability), and thus our concept is more general and incorporates
standard half-spaces as special cases. 
\end{rem}

\begin{rem}
Here, $(\mu_{i})_{i\in[n]}$ can be seen as a conditional probability
measure $\mu_{X|W}$ with $\mu_{X|W=i}=\mu_{i}$, $\frac{1}{n}\sum_{i=1}^{n}\ln\frac{\d\mu_{i}}{\d\nu}(x_{i})$
is known as the conditional information density of $\mu_{X|W}$ w.r.t.
$\nu$ given $\mu_{W}\sim\mathrm{Unif}[n]$, and its expectation $D(\mu_{X|W}\|\nu|\mu_{W}):=\frac{1}{n}\sum_{i=1}^{n}D(\mu_{i}\|\nu)$
is known as the conditional relative entropy. That is why it is called
the conditional typical set. In addition, the standard definition
of the conditional relative entropy typical set is $\big\{\mathbf{x}\in\M^{n}:\big|\sum_{i=1}^{n}\ln\frac{\d\mu_{i}}{\d\nu}(x_{i})-\sum_{i=1}^{n}D(\mu_{i}\|\nu)\big|\le n\epsilon\big\}$
\cite[(11.206)]{Cov06}. Our definition here is slightly different
from the standard one, and has benefit that it incorporates the isoperimetric
minimizers for Gaussian and Lebesgue measures as special cases, which
is explained below. 
\end{rem}

The concept of $\mathbf{f}$-coordinate half-space generalizes the
standard half-spaces $\{\mathbf{x}\in\mathbb{R}^{n}:\sum_{i=1}^{n}w_{i}x_{i}\le t\}$
in Euclidean spaces, and the former reduces to the latter when $f_{i}(x_{i})=-w_{i}x_{i}$
and $t$ is rechosen properly, or essentially equivalently, set $\nu=\mathcal{N}(0,1),\mu_{i}=\mathcal{N}(w_{i},1)$
in the conditional typical set. It also generalizes the Euclidean
balls $\left\{ \mathbf{x}\in\mathbb{R}^{n}:\sum_{i=1}^{n}x_{i}^{2}\le t\right\} $
with $t>0$, and the former reduces to the latter when $f_{i}(x_{i})=-x_{i}^{2}$
and $t$ is rechosen properly, or essentially equivalently, set $\nu=\mathrm{Lebesgue},\mu=\mathcal{N}(0,c)$
with $c>0$ in the conditional typical set (although as an infinite
measure, the Lebesgue measure is not considered in this paper). In
both these two cases, $\mu_{i}$ are the extremizers in the optimal
nonlinear log-Sobolev inequality for $\nu$ \cite{bakry2004functional,polyanskiy2019improved}.
Furthermore, it is well known that standard half-spaces and Euclidean
balls are respectively isoperimetric minimizers for Gaussian measures
and Lebesgue measures. This observation indicates that the isoperimetric
minimizers perhaps can be identified by the extremizers in the optimal
nonlinear log-Sobolev inequality.

We now formalize the above observation and establish the large deviations
principle for the isoperimetric problem, a result of particular interest
as it reveals a precise equivalence between the asymptotic isoperimetric
inequality and the nonlinear log-Sobolev inequality. We state our
results by splitting into cases according to whether the following
condition holds.

\begin{condition}[Exponential Integrability]\label{cond:integrability}
For some (thus all) $x_{0}\in\M$, 
\[
\mathbb{E}_{\nu}[e^{\lambda d^{2}(X,x_{0})}]<\infty,\;\exists\lambda\in(0,\infty).
\]
\end{condition}

This condition can be partitioned into two cases: 
\begin{equation}
\mathbb{E}_{\nu}[e^{\lambda d^{2}(X,x_{0})}]<\infty,\;\forall\lambda\in(0,\infty),\label{eq:subGaussian}
\end{equation}
and 
\begin{equation}
\mathbb{E}_{\nu}[e^{\lambda_{1}d^{2}(X,x_{0})}]<\mathbb{E}_{\nu}[e^{\lambda_{2}d^{2}(X,x_{0})}]=\infty,\;\exists0<\lambda_{1}<\lambda_{2}<\infty.\label{eq:Gaussian-2}
\end{equation}

Under Condition \ref{cond:convexity}, Condition \ref{cond:integrability}
is in fact equivalent to that $\nu$ admits a standard log-Sobolev
inequality \cite{wang1997logarithmic}. We need the following assumption
when the condition in \eqref{eq:Gaussian-2} holds.

\begin{condition}[Ricci Curvature Bounded Below]\label{cond:Ricci}
$\mathrm{Ric}\ge K$ for some $K\in\mathbb{R}$. \end{condition}
\begin{thm}[Large Deviations Principle for Isoperimetry]
\label{thm:equivalence} Suppose the weighted Riemannian manifold
$(\M,d,\nu)$ satisfies Condition \ref{cond:convexity}. 
\begin{enumerate}
\item Suppose Condition \ref{cond:integrability} holds. Suppose Condition
\ref{cond:Ricci} (or Conditions \ref{cond:covering}, \ref{cond:volume},
or \ref{cond:CD} given in Section \ref{subsec:Variant-Isoperimetry-for})
holds when the condition in \eqref{eq:Gaussian-2} holds. Then, for
any $\alpha\ge0$, 
\begin{equation}
{\displaystyle \Lambda(\alpha)=\sqrt{\breve{\Theta}(\alpha)}}.\label{eq:asympisoper2}
\end{equation}
Moreover, a sequence of conditional typical sets $A_{n,\epsilon}((\mu_{1}^{*},...,\mu_{1}^{*},\mu_{2}^{*},...,\mu_{2}^{*})\|\nu)$
with $\mu_{1}^{*},\mu_{2}^{*}$ respectively appearing $\lfloor n\lambda^{*}\rfloor$
and $n-\lfloor n\lambda^{*}\rfloor$ times asymptotically attains
the limit $\Lambda(\alpha)$ as $n\to\infty$ first and $\epsilon\to0$
then, where $(\lambda^{*},\mu_{1}^{*},\mu_{2}^{*})$ is the optimal
solution in \eqref{eq:Theta_LCE}. 
\item If Condition \ref{cond:integrability} does not hold, then, for any
$\alpha\ge0$, 
\[
{\displaystyle \Lambda(\alpha)=\sqrt{\breve{\Theta}(\alpha)}=0}.
\]
\end{enumerate}
\end{thm}

Although the theorem above is valid for spaces that are general enough,
it does not cover isoperimetric problems for log-concave probability
measures defined on closed convex subsets of Euclidean spaces. To
cover them, we extend the theorem above to an $\mathrm{RCD}(0,\infty)$-subspace
of the Riemannian manifold $\M$. Here, an $\text{RCD}(0,\infty)$-space
(a space satisfying Riemannian Ricci Curvature Dimension bound $\text{RCD}(0,\infty)$)
is a generalization of Riemannian manifolds satisfying Condition \ref{cond:convexity}
to possibly non-smooth setting, e.g., compact subsets of Euclidean
spaces equipped with log-concave probability measures. From the perspective
of the geometry of optimal transport, an $\text{RCD}(0,\infty)$-space
is a Polish metric probability measure space such that infinitesimal
Hilbertianity is satisfied (i.e., the Cheeger energy is a quadratic
form), and meanwhile, the relative entropy functional $D(\cdot\|\nu)$
is not merely $K$-geodesically convex on the Wasserstein space $(\mathscr{P}_{2}(\mathcal{X}),\W)$,
but that this convexity is robust: it extends to hold along all interpolated
measures induced by weighted optimal geodesic plans (or test plans).
See details for the definition of the $\text{RCD}(0,\infty)$-space,
or equivalently, the Riemannian Energy measure space satisfying $\text{BE}(K,\infty)$
in Section \ref{subsec:Riemannian-Energy-Measure} or \cite{AmbrosioGigliSavare14,Ambrosio2015BakryEmery}.
\begin{thm}
\label{thm:approximate}Let $\nu$ be an absolutely continuous probability
measure with density $f$ (w.r.t. Riemannian volume) on a finite-dimensional
smooth complete oriented connected Riemannian manifold $\M$ such
that $\nu$ together with its support forms a $\mathrm{RCD}(0,\infty)$-subspace
of $\M$. Suppose that there exists a sequence of probability measures
$\nu_{m}\in\P_{2}^{\mathrm{ac}}(\M),m\in\mathbb{N}$ with densities
$f_{m}$ (w.r.t. Riemannian volume) satisfying Condition \ref{cond:convexity}
and $f_{m}\ge(1-1/m)f$, as well as Condition \ref{cond:Ricci} if
the condition in \eqref{eq:Gaussian-2} holds. Then, for any $\alpha\ge0$,
${\displaystyle \Lambda(\alpha)=\sqrt{\breve{\Theta}(\alpha)}}.$ 
\end{thm}

\begin{rem}
The convergence of $\{\nu_{m}\}$ here is called ``convergence from
above'' by E. Milman \cite{milman2009role}. In addition, another
notion, ``convergence from within'', was also introduced by him.
Theorem \ref{thm:approximate} still holds if convergence of $\{\nu_{m}\}$
from above is replaced by convergence from within. 
\end{rem}

The equivalence $\Lambda=\sqrt{\breve{\Theta}}$ in the two theorems
above can be alternatively formulated as follows: For any convex $\theta:[0,\infty)\to[0,\infty)$,
\[
{\displaystyle \Lambda(\alpha)\ge\sqrt{\theta(\alpha)}},\;\forall\alpha\ge0\quad\Longleftrightarrow\quad I(\mu\|\nu)\ge\theta(D(\mu\|\nu)),\;\forall\mu\in\P_{2}^{\mathrm{ac}}(\M).
\]
It is also equivalent to that for any $\alpha>0$, 
\begin{equation}
{\displaystyle \lim_{n\to\infty}\frac{I_{n}(a_{n})}{I_{\mathrm{G}}(a_{n})}=\lim_{n\to\infty}\frac{I_{n}(a_{n})}{a_{n}\sqrt{2\ln\frac{1}{a_{n}}}}=\sqrt{\frac{\breve{\Theta}(\alpha)}{2\alpha}}},\label{eq:asympisoper2-2}
\end{equation}
where $a_{n}=e^{-n\alpha}$ and $I_{\mathrm{G}}$ is the isoperimetric
profile for the standard Gaussian measure. That is, as the weighted
volume vanishes exponentially, the asymptotic behavior of the isoperimetric
profile for any product probability measure resembles the one for
the standard Gaussian measure, but with different factors that are
determined by the exponential convergence rate of the weighted volume.
One prototypical case occurs when $\Theta(\alpha)$ grows linearly
in $\alpha$---for instance, $\Theta(\alpha)=2\alpha$ for Gaussian
measures---implying that $I_{n}(a)$ behaves as $C\cdot a\sqrt{\ln\frac{1}{a}}$,
with standard half-spaces likely serving as approximate extremizers
for small weighted volumes. A second prototypical case is when $\Theta(\alpha)$
grows exponentially in $\alpha$, as illustrated in the subsequent
example, which implies that $I_{n}(a)$ behaves as $C\cdot\sqrt{n}\,a^{1-1/n}$,
with geodesic balls acting as approximate extremizers for small weighted
volumes. Theorem \ref{thm:equivalence} provides a rigorous justification
from the perspective of nonlinear log-Sobolev inequalities for why
isoperimetric minimizers of small weighted volumes behave fundamentally
differently across spaces with distinct geometric structures. 
\begin{example}[Log-Concave Probability Measures]
\label{exa:Hypercube-1} All smooth log-concave probability densities
on Euclidean spaces satisfy Conditions \ref{cond:convexity} and \ref{cond:Ricci}.
Moreover, by standard regularization techniques (e.g., the Moreau-Yosida
regularization and convolution with Gaussian kernels), all log-concave
probability densities defined on closed convex subsets of Euclidean
spaces can be approximated by smooth log-concave probability densities
defined on the whole Euclidean spaces in the sense as in Theorem \ref{thm:approximate}.
Moreover, closed convex subsets equipped with log-concave probability
densities satisfy $\mathrm{RCD}(0,\infty)$. Thus, $\Lambda=\sqrt{\breve{\Theta}}$
holds for all (smooth or non-smooth) log-concave probability measures
on Euclidean spaces (supported on the whole spaces or convex subsets).
The standard Gaussian measure is one of the most important examples,
which admits $\breve{\Theta}(\alpha)=\Theta(\alpha)=2\alpha$ and
thus, $\Lambda(\alpha)=\sqrt{2\alpha}$, consistent with the formula
in \eqref{eq:Gaussian}. 
\end{example}

\begin{example}[Bounded Manifolds admit Exponential Growth of $\Theta$]
\label{exa:exponential} Let $D(t)=D(\mu_{t}\|\nu)$, $I(t)=I(\mu_{t}\|\nu)$,
and $N(t)=\exp(-2D(\mu_{t}\|\nu))$ respectively be the relative entropy,
Fisher information, and entropy power along the (generalized) heat
flow $\partial_{t}\rho=L\rho$ with $L=\Delta-\nabla V\cdot\nabla$
and $\mu_{t}=\rho_{t}\nu$. It is easy to see that $\Theta(\alpha)$
grows exponentially as $\alpha\to\infty$, as long as $N(t)$ is concave
in the time $t$ and $\Theta(\alpha_{0})>0$ for some $\alpha_{0}>0$.
This is because, in this case, $D(t)$ continuously decays to zero
as $t\to\infty$ and $N'(t)=2N(t)I(t)$ is nonincreasing, which implies
that for any $\mu_{0}$ satisfying $D(0)\ge\alpha_{0}$, we always
have $N(0)I(0)\ge N(t_{0})I(t_{0})\ge\Theta(\alpha_{0})e^{-2\alpha_{0}}$
with $t_{0}$ being the solution to $D(t_{0})=\alpha_{0}$. Therefore,
$I(0)\ge\Theta(\alpha_{0})e^{2(D(0)-\alpha_{0})}$ for any $\mu_{0}$
satisfying $D(0)\ge\alpha_{0}$, i.e., exponential growth of $\Theta$
(when $\Theta(\alpha_{0})>0$): $\Theta(\alpha)\ge\Theta(\alpha_{0})e^{2(\alpha-\alpha_{0})}$
for all $\alpha\ge\alpha_{0}$. The concavity of the entropy power
holds under the $\mathrm{CD}(0,m)$-condition with $m\in[n,\infty)$
for smooth Riemannian manifolds, as shown by S. Li and X. Li \cite[Theorem 2.2]{li2024on},
and also under the $\mathrm{RCD}(0,m)$-condition with $m\in[n,\infty)$
and $\int_{\mathcal{X}}\left[\frac{|\Delta\rho|^{2}}{\rho}+\frac{|\nabla\Delta\rho|^{2}}{\rho}\right]\d\nu<\infty$
for Polish metric measure spaces with $\rho$ denoting the heat flow,
as shown by X. Li and E. Zhang \cite[Theorem 3.1 and Remark 3.3]{li2025w}.
The condition that $\Theta(\alpha_{0})>0$ for some $\alpha_{0}>0$
is guaranteed as long as a standard log-Sobolev inequality (or a Poincaré
inequality) is admitted. These two conditions, concavity of $N$ and
positivity of $\Theta(\alpha_{0})$, are satisfied by the (normalized)
unweighted measure on a bounded manifold with nonnegative Ricci curvature,
which hence admits exponential growth of $\Theta$ and also exponential
growth of $\Lambda$ (since Conditions \ref{cond:convexity} and \ref{cond:Ricci}
are satisfied in this setting and Theorem \ref{thm:equivalence} is
valid). See other sufficient conditions to ensure exponential growth
of $\Theta$ in \cite[Theorem 6.8.1]{bakry2013analysis}\cite[Theorem 2.5]{li2024on}. 
\end{example}

We next consider a simple but nontrivial case---the uniform measure
on hypercube. 
\begin{example}[Hypercube]
\label{exa:Hypercube} Consider the hypercube $[0,1]^{n}$ (equivalently,
$(0,1)^{n}$) with Euclidean metric and uniform measure. The standard
log-Sobolev inequality and the Poincaré inequality with constants
$K_{LS}=K_{P}=\pi^{2}$ ($K_{LS},K_{P}$ defined later in \eqref{eq:LS}
and \eqref{eq:poincare}) respectively hold in this setting \cite{ghang2014sharp},
which implies $\lim_{\alpha\to0}\frac{\breve{\Theta}(\alpha)}{2\alpha}=\lim_{\alpha\to0}\frac{\Theta(\alpha)}{2\alpha}=\pi^{2}$.
Moreover, the $\mathrm{RCD}(0,m)$-condition holds obviously for $m>n$
and the condition $\int_{[0,1]}\left[\frac{|\Delta\rho|^{2}}{\rho}+\frac{|\nabla\Delta\rho|^{2}}{\rho}\right]\d\nu<\infty$
is satisfied for any $t>0$ for heat flow $\rho$ with the Neumann
boundary condition \cite[Section 4.2]{strauss2007partial} (since
$[0,1]$ is compact, $\rho\in C^{\infty}([0,1])$, and $\rho_{t}>0$
for any $t>0$). Thus, as mentioned in Example \ref{exa:exponential},
we have exponential growth of $\Theta$, i.e., $\Theta(\alpha)\ge2\pi^{2}\alpha_{0}\cdot e^{2(\alpha-\alpha_{0})}\ge\frac{\pi^{2}}{e}\cdot e^{2\alpha}$
(with $\alpha_{0}$ chosen as $1/2$), which further implies exponential
growth of $\Lambda$, i.e., $\Lambda(\alpha)\ge\sqrt{\frac{\pi^{2}}{e}}\cdot e^{\alpha}$.
This lower bound is sharp in the exponent by verifying balls centered
at the vertices of the hypercube, since for these balls with volume
$e^{-n\alpha}$, their surface areas are $(1+o(1))\sqrt{\frac{\pi en}{2}}\cdot e^{\alpha(1-n)}$
as $n\to\infty$, giving $\Lambda(\alpha)\le\sqrt{\frac{\pi e}{2}}\cdot e^{\alpha}$.
In fact, it was already known \cite{ritore2015isoperimetric,glaudo2024isoperimetric}
that given each $n$, for all small volumes, balls centered at the
vertices are isoperimetric minimizers. However, it is unclear whether
``small volume'' here is large enough to include ``exponentially
small volume''. We conjecture the following precise relation holds:
$\Lambda(\alpha)\sim\sqrt{\frac{\pi e}{2}}\cdot e^{\alpha}$, or equivalently,
$\breve{\Theta}(\alpha)\sim\frac{\pi e}{2}\cdot e^{2\alpha}$ as $\alpha\to\infty$,
suggesting that half-normal distributions truncated on $(0,1)$ with
sufficiently small variances approximately saturate nonlinear log-Sobolev
inequalities for large $\alpha$. The factor $\frac{\pi^{2}}{e}$
obtained above is about $0.85\cdot\ensuremath{\frac{\pi e}{2}}$. 
\end{example}

Theorem \ref{thm:equivalence} reveals a precise equivalence between
the asymptotic isoperimetric inequality and the nonlinear log-Sobolev
inequality. It is termed the \textit{large deviations principle for
isoperimetry} for two reasons. First, both this theorem and large
deviations theory focus on sets of exponentially small measures. Second,
isoperimetric minimizers have been proven to be expressed in terms
of $\mathbf{f}$-coordinate half-spaces (or conditional typical sets)
which concern the remote tails of a sum. These remote tails themselves
are the core object of investigation in large deviations theory.

Showing equivalence among inequalities from various fields is an interesting
topic, since it will unveil the intrinsic relationships among diverse
branches of mathematics. However, most of equivalences in the literature
on this topic are known only in a weak sense. In particular, it was
proved by Ledoux \cite{ledoux1994semigroup,ledoux1994simple} on Gaussian
spaces or compact Riemannian manifolds and by Bobkov \cite{bobkov1999isoperimetric}
on Euclidean spaces with log-concave probability measures that for
a fixed $n$, a Gaussian-type isoperimetric inequality 
\[
I_{n}(a)\ge\sqrt{K}I_{\mathrm{G}}(a),\forall a
\]
holds for some $K>0$, if and only if the standard log-Sobolev inequality
holds with some constant $K'>0$. Theorem \ref{thm:equivalence} provides
an equivalence in a stronger sense, which showing a precise equivalence
between the isoperimetric inequality and (nonlinear) log-Sobolev inequality.
As demonstrated in the case of bounded manifolds, Ledoux and Bobkov's
weak equivalence exhibits significant looseness for small $a$ (that
is, in the context of small sets).

Although our equivalence is proven for product spaces, it in fact
implies Ledoux and Bobkov's weak equivalence for non-product spaces.
In particular, by the facts that $I_{1}\ge I_{n}$ and $\M$ is an
arbitrary (non-product) space, if the standard log-Sobolev inequality
holds with some constant $K'>0$, then our equivalence implies that
$\liminf_{a\to0}\frac{I_{1}(a)}{I_{\mathrm{G}}(a)}\ge\liminf_{n\to\infty}\frac{I_{n}(a_{n})}{I_{\mathrm{G}}(a_{n})}=\sqrt{K'},$
where $a_{n}=e^{-n\alpha}$ for any $\alpha>0$. (More details on
the inequality here is given in the following subsubsection.) That
is, $I_{1}(a)\ge\sqrt{K}I_{\mathrm{G}}(a),\forall a$ for some $K>0$,
yielding a Gaussian-type isoperimetric inequality for the space $\M$.
Conversely, if a Gaussian-type isoperimetric inequality for $\M$
holds with constant $K>0$, then by the well known fact that $I_{1}(a)\ge\sqrt{K}I_{\mathrm{G}}(a),\forall a\Longleftrightarrow I_{n}(a)\ge\sqrt{K}I_{\mathrm{G}}(a),\forall a,n$,
our equivalence implies the standard log-Sobolev inequality with the
same constant $K$.

Our equivalence will be applied in the following to compare the optimal
constants in various inequalities.

\subsubsection{Companion of Various Constants}

While Theorem \ref{thm:equivalence} is primarily tailored to subsets
of weighted volumes decaying exponentially in $n$, it admits application
to bounding the isoperimetric profile for weighted volumes decaying
arbitrarily slowly or even constant weighted volumes. In other words,
Theorem \ref{thm:equivalence} can be used to quantitatively compare
the log-Sobolev constant, the isoperimetric constant, and other constants
for various inequalities. To this end, we define the following constants. 
\begin{itemize}
\item $K_{CD}$ denotes the optimal constant $K$ in the Bakry--Émery--Ricci
curvature-dimension condition 
\begin{equation}
\mathrm{Ric}_{g}+\mathrm{Hess}_{g}V\geq Kg.\label{eq:RicHess-1}
\end{equation}
\item $K_{LS}$ denotes the optimal constant $K$ in the log-Sobolev inequality
\begin{equation}
I(\mu\|\nu)\ge2K\cdot D(\mu\|\nu),\;\forall\mu,\label{eq:LS}
\end{equation}
or equivalently, 
\[
K_{LS}=\frac{\breve{\Theta}'(0)}{2}=\inf_{\alpha>0}\frac{\breve{\Theta}(\alpha)}{2\alpha}=\inf_{\alpha>0}\frac{\Theta(\alpha)}{2\alpha},
\]
where $\breve{\Theta}'(0)$ is the right derivative of $\breve{\Theta}$
at $0$. 
\item $K_{LS}^{+}$ denotes 
\begin{equation}
K_{LS}^{+}:=\sup_{\alpha>0}\frac{\breve{\Theta}(\alpha)}{2\alpha}=\lim_{\alpha\to\infty}\frac{\breve{\Theta}(\alpha)}{2\alpha}.\label{eq:-11}
\end{equation}
\item $K_{HC}$ denotes the optimal constant $K$ in the hypercontractivity
inequality for the semigroup $T_{t}=e^{-tL}$ with $L=\Delta-\nabla V\cdot\nabla$,
\[
\|T_{t}f\|_{q}\le\|f\|_{p},\;\forall f,
\]
for all $t\ge\frac{1}{2K}\ln\frac{q-1}{p-1}$, where the norms are
taken w.r.t. $\nu$. 
\item $K_{T}$ denotes the optimal constant $K$ in the transport inequality
\[
\forall\mu\in\P_{2}^{\mathrm{ac}}(\M),\quad\W(\mu,\nu)\leq\sqrt{\frac{2D(\mu\|\nu)}{K}},
\]
where 
\[
\W(\mu,\nu):=\left(\inf_{\pi\in\Pi(\mu,\nu)}\int_{\M\times\M}d^{2}(x,y)\d\pi(x,y)\right)^{1/2}
\]
is the Wasserstein distance of order $2$ and $\Pi(\mu,\nu)$ is the
set of couplings of $\mu,\nu$, i.e., the set of joint distributions
with marginals equal to $\mu,\nu$. 
\item $K_{T}^{+}$ denotes the optimal constant $K$ in the transport inequality
that there is some constant $D_{0}\ge0$ such that for all 
\[
\W(\mu,\nu)\leq\sqrt{\frac{2D(\mu\|\nu)}{K}},\quad\forall\mu\in\P_{2}^{\mathrm{ac}}(\M)\textrm{ s.t. }D(\mu\|\nu)\ge D_{0}.
\]
\item $K_{C}$ denotes the optimal constant $K$ in the dimension-free Gaussian
concentration inequality that for all $n\ge1$ and $r\ge0$, $\nu_{n}((A^{r})^{c})\le e^{-\frac{Kr^{2}}{2}}$
for all $A$ such that $\nu_{n}(A)\ge1/2$. 
\item $K_{C}^{+}$ denotes the optimal constant $K$ in the asymptotic Gaussian
concentration inequality that there is some constant $r_{0}\ge0$
such that for all $n\ge1$ and $r\ge\sqrt{n}r_{0}$, $\nu_{n}((A^{r})^{c})\le e^{-\frac{Kr^{2}}{2}}$
for all $A$ such that $\nu_{n}(A)\ge1/2$. 
\item $K_{IS}$ denotes the optimal constant $K$ in the large deviations
asymptotics of the isoperimetric inequalities 
\[
\liminf_{n\to\infty}\frac{I_{n}(a_{n})}{a_{n}\sqrt{2\ln\frac{1}{a_{n}}}}\ge\sqrt{K},\;\forall\alpha>0,
\]
where $a_{n}=e^{-n\alpha}$. 
\item $K_{IS}^{+}$ denotes the optimal constant $K$ in the following asymptotic
isoperimetric inequality 
\[
\liminf_{a\to0}\frac{I_{n}(a)}{a\sqrt{2\ln\frac{1}{a}}}\ge\sqrt{K},\;\forall n\ge1.
\]
\item $K_{IS}^{-}$ denotes the constant in the dimension-free (also infinite-dimensional)
isoperimetric inequality 
\begin{equation}
\liminf_{a\to0}\frac{I_{\inf}(a)}{a\sqrt{2\ln\frac{1}{a}}}=\sqrt{K_{IS}^{-}},\label{eq:-23}
\end{equation}
where 
\[
I_{\inf}(a):=\lim_{n\to\infty}I_{n}(a)=\inf_{n\ge1}I_{n}(a).
\]
\end{itemize}
From the definitions, the constants $K_{IS}$, $K_{IS}^{+}$, and
$K_{IS}^{-}$ are the asymptotic isoperimetric ratios $\frac{I_{n}(a)}{a\sqrt{2\ln\frac{1}{a}}}$
respectively for exponentially, superexponentially, and subexponentially
small sequences of volumes. We now compare all optimal constants defined
above. 
\begin{thm}
\label{thm:comparison} Suppose the weighted Riemannian manifold $(\M,d,\nu)$
satisfies Condition \ref{cond:convexity} (i.e., $K_{CD}\ge0$). Then,
it holds that 
\begin{align}
K_{CD}\overset{\mathrm{(a)}}{\le}\left(\frac{1+K_{CD}/K_{C}}{2}\right)^{2}K_{C}\overset{\mathrm{(b)}}{\le} & K_{IS}^{-}\overset{\mathrm{(c)}}{\le}K_{IS}\overset{\mathrm{(d)}}{\le}K_{LS}\overset{\mathrm{(e)}}{=}K_{HC}\label{eq:-15}\\
 & \overset{\mathrm{(f)}}{\le}K_{T}\overset{\mathrm{(g)}}{=}K_{C}\overset{\mathrm{(h)}}{\le}K_{C}^{+}\overset{\mathrm{(i)}}{=}K_{T}^{+},\nonumber 
\end{align}
and $K_{LS}\overset{\mathrm{(j)}}{\le}K_{LS}^{+}.$ If additionally,
Condition \ref{cond:Ricci} holds when the condition in \eqref{eq:Gaussian-2}
holds, then the equality in (d) holds and moreover, $K_{LS}^{+}\overset{\mathrm{(k)}}{\le}K_{IS}^{+}.$
In addition, if $K_{CD}$ is equal to any one of $K_{IS}^{-},K_{IS},K_{LS},K_{HC},K_{T},K_{C}$,
then all the inequalities in (a)--(g) of \eqref{eq:-15} become equalities. 
\end{thm}

Inequality (a) is trivial, by noting that $0\le K_{CD}/K_{C}\le1$.
Inequality (c) is due to the fact that $n\mapsto I_{n}(a)$ is nonincreasing.
Equality (d) is implied by Theorem \ref{thm:equivalence}. Equality
(e) is a classic result of Gross \cite{gross1975logarithmic}. The
inequality $K_{LS}\le K_{T}$ (or equivalently, Inequality (f)) was
proven by Otto and Villani \cite{otto2000generalization}. Equality
(g) is also well known; see e.g. \cite[Remark 22.23]{villani2008optimal}.
Inequalities (h) and (j) follow by definition. Equality (i) follows
similarly to (g). Inequality (k) follows by Theorem \ref{thm:equivalence},
with a detailed proof given in Appendix C. So, it only remains to
prove Inequality (b). This is done in Appendix C. In addition, a weaker
version of (b), $\left(\frac{1+K_{CD}/K_{C}}{2}\right)^{2}K_{C}\le K_{LS}$,
was already known; see Theorem 22.21 of \cite{villani2008optimal}.
However, we do not know if this is equivalent to the inequality (b)
since we do not know if $K_{IS}^{-}=K_{IS}$ holds. 
\begin{example}[Hypercube]
We continue considering the hypercube described in Example \ref{exa:Hypercube}.
Statements in Example \ref{exa:Hypercube} imply that for this case,
\[
K_{CD}=0,\;K_{IS}=K_{LS}=K_{HC}=K_{T}=K_{C}=\pi^{2},\;K_{LS}^{+}=K_{C}^{+}=K_{T}^{+}=K_{IS}^{+}=\infty,
\]
where the density $1+\epsilon\cos\pi x$ asymptotically attains the
constants $K_{LS}$ and $K_{T}$ as $\epsilon\to0$. Moreover, it
holds that $\sqrt{2\pi}\le K_{IS}^{-}\le\pi^{2}$; the lower bound
will be explained in Example \ref{exa:We-continue-considering}. However,
the exact value of $K_{IS}^{-}$ is still unknown. 
\end{example}

An open problem on this topic is the following one. 
\begin{problem}
\label{prob:Assume-that-}Under Condition \ref{cond:convexity}, is
it true that $K_{IS}^{-}=K_{LS}$? 
\end{problem}

\subsubsection{Central Limit Regime }

We now consider the asymptotics of the isoperimetric profile for fixed
weighted volume as $n\to\infty$. Let $I_{\infty}$ be the isoperimetric
profile defined for the infinite dimensional product space $\M^{\infty}$
equipped with the product topology and the product measure $\nu_{\infty}:=\nu^{\otimes\infty}$,
but the enlargement of a set is defined under $d_{\infty}(\mathbf{x},\mathbf{y})=\sqrt{\sum_{i=1}^{\infty}d(x_{i},y_{i})^{2}}$
(although it might diverse to infinity for two arbitrary points).
Then, it is obvious that $I_{\inf}(a)\ge I_{\infty}(a)$. In the following,
we show that this is actually an equality. That is, the isoperimetric
profile $I_{\infty}$ on the infinite dimensional space is the pointwise
limit of the isoperimetric profiles on finite dimensional spaces.
The proof is provided in Appendix D. 
\begin{thm}
\label{thm:I_inf}Assume that $I_{\inf}$ is continuous (e.g., Condition
\ref{cond:convexity} holds). Then, $I_{\inf}(a)=I_{\infty}(a)$ for
all $a\in[0,1]$. 
\end{thm}

A well-known lower bound on $I_{\inf}$ is that 
\begin{equation}
I_{\inf}(a)\ge K_{0}\cdot I_{\mathrm{G}}(a),\;\forall a\in[0,1],\label{eq:-24}
\end{equation}
where $I_{\mathrm{G}}$ is the Gaussian isoperimetric profile and
$K_{0}:=\inf_{a\in(0,1)}\frac{I_{\nu}(a)}{I_{\mathrm{G}}(a)}$ \cite{bobkov1997isoperimetric,bakry1996levy}.
An upper bound can be obtained by considering the $\mathbf{f}$-coordinate
half-space given in \eqref{eq:halfspace}, as shown in the following
theorem. The proof is based on the coarea formula and the central
limit theorem, which is given in Appendix E. 
\begin{thm}[Upper Bound for Isoperimetry in Central Limit Regime]
\label{thm:clt} Assume that $I_{\inf}$ is continuous (e.g., Condition
\ref{cond:convexity} holds). Then, for any $a\in[0,1]$, 
\begin{equation}
I_{\inf}(a)\le\sqrt{K_{P}}\cdot I_{\mathrm{G}}(a),\label{eq:-27}
\end{equation}
where $K_{P}$ is the spectral gap (the reciprocal of the Poincaré
constant) for $\nu$, i.e., the optimal constant in 
\begin{equation}
K_{P}\cdot\mathrm{Var}_{\nu}(f)\leq\mathbb{E}_{\nu}\left[\|\nabla f\|^{2}\right]\label{eq:poincare}
\end{equation}
for all $f$ in the Sobolev space $W^{1,2}(\M)$. 
\end{thm}

As a related result, the Poincaré constant and the Cheeger constant
are known to be closely related; see e.g., \cite{klartag2025isoperimetric}.
The Cheeger constant characterizes the isoperimetric behavior exclusively
for large subsets, whereas the inequality in \eqref{eq:-27} captures
isoperimetry for all subsets without restriction.

The upper bound $\sqrt{K_{P}}I_{\mathrm{G}}(a)$ in \eqref{eq:-27}
can be used to bound $K_{IS}^{-}$, namely, obtaining that $K_{IS}^{-}\le K_{P}$.
However, this bound is worse than the upper bound $K_{LS}$, since
in general $K_{LS}\le K_{P}$. Note that by perturbation analysis,
one can see that 
\[
K_{P}=\liminf_{\alpha\to0}\frac{\Theta(\alpha)}{2\alpha}.
\]
Thus, $K_{LS}=K_{P}$ if $\Theta(\alpha)\ge2K_{P}\alpha$ (e.g., if
$\Theta$ is convex), and otherwise, $K_{LS}<K_{P}$ (e.g., for exponential
measures \cite[Example 21.19]{villani2008optimal}). For the latter
case, the upper bound $\sqrt{K_{P}}I_{\mathrm{G}}(a)$ is not tight
for sufficiently small $a$. Combining the best known lower bound
in \eqref{eq:-24}, the upper bound in \eqref{eq:-27}, and the trivial
upper bound $I_{1}$ yields 
\[
K_{0}I_{\mathrm{G}}(a)\le I_{\inf}(a)\le I_{1}(a)\wedge\sqrt{K_{P}}I_{\mathrm{G}}(a).
\]

\begin{example}[Hypercube]
\label{exa:We-continue-considering}We continue considering the hypercube
described in Example \ref{exa:Hypercube}. Applying the inequalities
above to the hypercube $(0,1)^{n}$ yields $\sqrt{2\pi}I_{\mathrm{G}}(a)\le I_{\inf}(a)\le1\wedge\pi I_{\mathrm{G}}(a)$.
As shown in \cite{glaudo2024isoperimetric}, the lower bound is not
tight. 
\end{example}

A challenging problem is to exactly determine the function $I_{\inf}$
(or equivalently, $I_{\infty}$). This problem is a commonly recognized
open one. 
\begin{problem}
Under Condition \ref{cond:convexity}, what is the explicit expression
of $I_{\inf}$? 
\end{problem}

\subsection{Related Works}
\begin{itemize}
\item \textbf{Isoperimetry} 
\end{itemize}
In the one-dimensional setting, Bobkov \cite{bobkov1996extremal}
established that half-lines are extremal sets in the isoperimetric
inequality for log-concave probability measures. Extending to higher-dimensional
structures, Bobkov and Houdré \cite{bobkov1997some} further explored
the isoperimetric constant and log-Sobolev inequalities in product
spaces, laying a foundation for multi-dimensional isoperimetric analysis.
The optimality of standard half-spaces--- a core topic in isoperimetry---
has been systematically studied by multiple scholars: Sudakov and
Tsirel'son \cite{sudakov1978extremal} initiated early investigations,
Bobkov \cite{bobkov1996extremal} deepened this topic by addressing
the isoperimetric problem for log-concave product measures on $\mathbb{R}^{n}$
equipped with the uniform distance, and explicitly provided necessary
and sufficient conditions for standard half-spaces to be extremal,
and Barthe \cite{barthe2001extremal} investigated the extremal properties
of central half-spaces (with measure $1/2$) for product probability
measures on Riemannian manifolds equipped with the geodesic distance.
It should be noting that the isoperimetric problem under the geodesic
distance and the one under the uniform distance are significantly
different. When equipped with the uniform distance, the asymptotics
of the isoperimetric profile has been characterized by Bobkov in \cite{bobkov1997isoperimetric3}.
However, for the geodesic distance as in our setup, the asymptotics
is still widely unknown except for special cases, e.g., the Gaussian
setting. Roberto \cite{roberto2010isoperimetry} offers a critical
survey for a comprehensive overview of isoperimetry in product spaces.
See also other related works or survey in \cite{barthe2000some,barthe2002log,ros2005isoperimetric,milman2009role}.
Furthermore, the isoperimery for small sets was already investigated
in the literature; see, e.g., \cite{Leonardi2022isoperimetric} for
the continuous setting and \cite{raghavendra2010graph,ODonnell14analysisof}
for the discrete setting. In the discrete setting, the isoperimery
for small sets is known as the small set expansion and has important
applications in complexity theory \cite{raghavendra2010graph}. It
should be noted that exact characterization of the isoperimetric profile
remains widely open, and previously so does its asymptotics. 
\begin{itemize}
\item \textbf{Concentration of Measure} 
\end{itemize}
The concentration of measure in a probability metric space describes
a key phenomenon: a slight enlargement of any measurable set with
non-negligible probability will always result in a set with large
probability. In functional analysis terms, this is equivalent to the
fact that the value of any Lipschitz function is concentrated around
its median. The modern study of this phenomenon began in the early
1970s, when V. Milman pioneered its application to the asymptotic
geometry of Banach spaces. Subsequent decades saw in-depth exploration
by scholars including Gromov, Maurey, Pisier, Schechtman, Talagrand,
and Ledoux, solidifying its role in analysis and probability. A critical
advancement came from Talagrand \cite{talagrand1995concentration},
who studied concentration of measure in product spaces equipped with
product probability measures and derived a suite of sharp concentration
inequalities. In information theory, this phenomenon is known as the
blowing-up lemma \cite{Ahls76,Marton86}, which Gács, Ahlswede, and
Körner leveraged to prove strong converses for two classic coding
problems.

Concentration of measure has natural connections to isoperimetric
inequalities. In fact, it can be seen as a variant version of isoperimetric
problem concerning on sets with probability measure $1/2$ in which
the boundary of a set is set to the $r$-neighbourhood of the set
with $r>0$. A landmark methodological contribution in concentration
of measure was made by Marton \cite{Marton86}, who was the first
to introduce information-theoretic techniques (especially transport-entropy
inequalities) to the study of concentration of measure, yielding an
elegant and concise proof for a dimension-free bound on concentration
of measure. Talagrand extended this idea by developing a new transport-entropy
inequality, adapting it to Gaussian measures and the Euclidean metric
\cite{talagrand1996transportation}; this argument has since become
a standard reference, featured in prominent textbooks \cite{ledoux2001concentration,RagSason,villani2003topics}.
Gozlan and Léonard further refined this approach by replacing the
``linear'' transport-entropy inequality in Marton's argument with
a ``nonlinear'' version, obtaining a dimension-free sharp bound
on concentration of measure---exponentially tight, as the exponent
of the bound is asymptotically attained \cite{gozlan2007large}. Prior
to Gozlan and Léonard's work, Alon et al. \cite{alon1998asymptotic}
are the first to characterize the sharp convergence exponent for concentration
of measure, but limited to finite spaces. Gozlan \cite{gozlan2009characterization}
also used Marton's framework to prove the equivalence between the
Gaussian bound and Talagrand's transport-entropy inequality. Additionally,
Dembo \cite{dembo1997information} proposed a new class of transport-entropy
inequalities, which he used them to recover and generalize several
of Talagrand's key results in \cite{talagrand1995concentration}.
Furthermore, E. Milman \cite{milman2010isoperimetric} established
an equivalence between isoperimetry and concentration of measure.
There are many existing works investigating isoperimetric inequalities
or concentration of measure in the discrete setting, e.g., on various
graphs, e.g., \cite{margulis1974veroyatnostniye,Ahls76,bollobas1986combinatorics,diskin2024isoperimetry}.

Ahlswede and Zhang \cite{ahlswede1999asymptotical} and Yu \cite{yu2024exact}
contributed to the study of isoperimetry with thick boundaries for
exponentially small sets. This isoperimetric variant refines the concentration
of measure phenomenon by relaxing the set size requirement to allow
arbitrarily exponentially small measures, yet it behaves fundamentally
differently from---and is far more intricate than---concentration
of measure. This is because, in contrast to the latter, the isoperimetric
variant admits no dimension-independent bounds. Ahlswede and Zhang
\cite{ahlswede1999asymptotical} focused on finite spaces for which
a key tool called the inherently typical subset lemma was developed
by them. Yu \cite{yu2024exact} generalized their result to Polish
metric spaces (complete, separable metric spaces), expanding the problem's
applicability, which plays a key role in our proof in the present
paper. 
\begin{itemize}
\item \textbf{Log-Sobolev Inequalities} 
\end{itemize}
Ledoux \cite{ledoux1994simple,ledoux2004spectral} proposed a semigroup
approach to derive isoperimetric inequalities from hypercontractivity
(which is equivalent to log-Sobolev inequalities). The semigroup approach
can be used to establish isoperimetric inequalities for the whole
range of volumes, but the constant obtained via this method is in
general suboptimal. In contrast, the constant in our work is asymptotically
optimal. Otto and Villani \cite{otto2000generalization} established
that log-Sobolev inequalities imply transport-entropy inequalities,
bridging log-Sobolev theory with concentration of measure. Building
on this, Ledoux \cite{ledoux2011concentration} (alongside Milman
\cite{milman2010isoperimetric}) used a semigroup approach to derive
isoperimetric inequalities from concentration inequalities (equivalently,
transport-entropy inequalities). As a nonlinear log-Sobolev inequality
for the standard heat semigroups (with Lebesgue measures as stationary
measures), the entropic isoperimetric inequality was investigated
in \cite{costa1984similarity,dembo1991information} which is in fact
equivalent to the standard log-Sobolev inequality for Ornstein--Uhlenbeck
semigroups (with Gaussian measures as stationary measures). Nonlinear
log-Sobolev inequalities for general heat (diffusion) semigroups were
investigated by Bakry \cite{bakry2004functional} (called the entropy-energy
inequality therein), F. Y. Wang \cite{wang2008generalized}, Bakry,
Gentil, and Ledoux \cite{bakry2013analysis}, as well as Polyanskiy
and Samorodnitsky \cite{polyanskiy2019improved}. Specifically, F.
Y. Wang \cite{wang2008generalized} showed that nonlinear log-Sobolev
inequalities imply nonlinear transport-entropy inequalities, extending
Otto and Villani's linear version to the nonlinear setting. Bakry
\cite{bakry2004functional}, Bakry, Gentil, and Ledoux \cite{bakry2013analysis},
and Polyanskiy and Samorodnitsky \cite{polyanskiy2019improved} applied
nonlinear log-Sobolev inequalities to strengthen hypercontractive
inequalities. Though nonlinear Sobolev inequalities have been studied
in the literature not directly to address asymptotic isoperimetry,
an exact equivalence between these two concepts does exist---one
we establish in the present work. Furthermore, equivalence between
isoperimetry and Sobolev-type inequalities was investigated in \cite{BobkovHoudre1997}.

\subsection{Proof Idea and Paper Organization}

\begin{figure}
\centering

\tikzset{every picture/.style={line width=0.75pt}} 
\begin{tikzpicture}[x=0.75pt,y=0.75pt,yscale=-1,xscale=1] 
\begin{scope}[scale=0.8]
\draw   (49,386.67) .. controls (49,375.62) and (71.18,366.67) .. (98.55,366.67) .. controls (125.92,366.67) and (148.1,375.62) .. (148.1,386.67) .. controls (148.1,397.71) and (125.92,406.67) .. (98.55,406.67) .. controls (71.18,406.67) and (49,397.71) .. (49,386.67) -- cycle ; 
\draw   (338.55,386.67) .. controls (338.55,375.62) and (367.42,366.67) .. (403.05,366.67) .. controls (438.67,366.67) and (467.55,375.62) .. (467.55,386.67) .. controls (467.55,397.71) and (438.67,406.67) .. (403.05,406.67) .. controls (367.42,406.67) and (338.55,397.71) .. (338.55,386.67) -- cycle ; 
\draw   (194.5,386.67) .. controls (194.5,375.62) and (216.68,366.67) .. (244.05,366.67) .. controls (271.42,366.67) and (293.6,375.62) .. (293.6,386.67) .. controls (293.6,397.71) and (271.42,406.67) .. (244.05,406.67) .. controls (216.68,406.67) and (194.5,397.71) .. (194.5,386.67) -- cycle ; 
\draw   (510.5,386.67) .. controls (510.5,375.62) and (532.68,366.67) .. (560.05,366.67) .. controls (587.42,366.67) and (609.6,375.62) .. (609.6,386.67) .. controls (609.6,397.71) and (587.42,406.67) .. (560.05,406.67) .. controls (532.68,406.67) and (510.5,397.71) .. (510.5,386.67) -- cycle ; 
\draw    (148.1,386.67) -- (192.5,386.67) ;
\draw [shift={(194.5,386.67)}, rotate = 180] [color={rgb, 255:red, 0; green, 0; blue, 0 }  ][line width=0.75]    (10.93,-3.29) .. controls (6.95,-1.4) and (3.31,-0.3) .. (0,0) .. controls (3.31,0.3) and (6.95,1.4) .. (10.93,3.29)   ; 
\draw    (467.55,386.42) -- (508.5,386.89) ;
\draw [shift={(510.5,386.92)}, rotate = 180.67] [color={rgb, 255:red, 0; green, 0; blue, 0 }  ][line width=0.75]    (10.93,-3.29) .. controls (6.95,-1.4) and (3.31,-0.3) .. (0,0) .. controls (3.31,0.3) and (6.95,1.4) .. (10.93,3.29)   ; 
\draw    (293.6,386.67) -- (336.55,386.67) ;
\draw [shift={(338.55,386.67)}, rotate = 180] [color={rgb, 255:red, 0; green, 0; blue, 0 }  ][line width=0.75]    (10.93,-3.29) .. controls (6.95,-1.4) and (3.31,-0.3) .. (0,0) .. controls (3.31,0.3) and (6.95,1.4) .. (10.93,3.29)   ; 
\draw    (560.05,407.17) .. controls (524.48,459.34) and (150.57,462.57) .. (99.3,407.5) ;
\draw [shift={(98.55,406.67)}, rotate = 49.22] [color={rgb, 255:red, 0; green, 0; blue, 0 }  ][line width=0.75]    (10.93,-3.29) .. controls (6.95,-1.4) and (3.31,-0.3) .. (0,0) .. controls (3.31,0.3) and (6.95,1.4) .. (10.93,3.29)   ;
\draw (54.5,375.5) node [anchor=north west][inner sep=0.75pt]   [align=left] {Log-Sobolev }; 
\draw (357.5,366.5) node [anchor=north west][inner sep=0.75pt]   [align=left] {\begin{minipage}[lt]{58.84pt}\setlength\topsep{0pt} \begin{center} Variant\\Isoperimetry \end{center}
\end{minipage}}; 
\draw (209.5,375.5) node [anchor=north west][inner sep=0.75pt]   [align=left] {Transport}; 
\draw (517,375.5) node [anchor=north west][inner sep=0.75pt]   [align=left] {Isoperimetry}; 
\draw (153,362.5) node [anchor=north west][inner sep=0.75pt]   [align=left] {\begin{minipage}[lt]{16.89pt}\setlength\topsep{0pt} \begin{center} OT \end{center}
\end{minipage}}; 
\draw (304,361.5) node [anchor=north west][inner sep=0.75pt]   [align=left] {\begin{minipage}[lt]{11.79pt}\setlength\topsep{0pt} \begin{center} IT \end{center}
\end{minipage}}; 
\draw (468,359) node [anchor=north west][inner sep=0.75pt]   [align=left] {\begin{minipage}[lt]{25.39pt}\setlength\topsep{0pt} \begin{center} GMT \end{center}
\end{minipage}}; 
\draw (261,451.5) node [anchor=north west][inner sep=0.75pt]   [align=left] {\begin{minipage}[lt]{80pt}\setlength\topsep{0pt} \begin{center} IT + OT or GMT \end{center}
\end{minipage}};
\end{scope}
\end{tikzpicture}\caption{Proof strategy of Theorem \ref{thm:equivalence}, where OT, IT, and
GMT respectively refer to optimal transport, information theory, and
geometric measure theory. }\label{fig:Proof-strategy-of}
\end{figure}
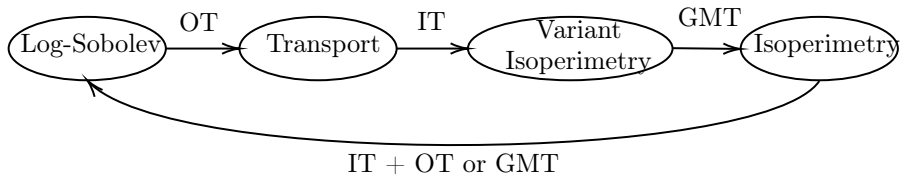

Ledoux's semigroup method, as developed in \cite{ledoux1994simple,ledoux2004spectral},
seems insufficient for establishing Theorem \ref{thm:equivalence}.
This limitation originates from the core strategy of the method: it
transforms the isoperimetric problem into a ``soft boundary'' variant
known as the noise stability problem. However, obtaining a precise
solution to the former demands an exact (or sufficiently precise,
not merely asymptotic) solution to the latter---an outcome that remains
elusive in general. While asymptotically optimal solutions for the
noise stability problem are well-established \cite{ODonnell14analysisof,kirshner2021moment,yu2024graphs,yu2021strong_article},
an exact resolution has yet to be achieved. Accordingly, a novel approach
is required to advance the proof of Theorem \ref{thm:equivalence}.

The heuristic reasoning behind our proof idea is as follows. Given
an arbitrary set in a high-dimensional space, due to lack of structures,
its (weighted) surface area does not admit any simple formulas and
thus, is difficult to analyze. To resolve this difficulty, a possible
way is to tailor the set to a new one that is ``regular'' in a certain
sense, so that the surface area of the resulting set admits a simple
formula. However, this is difficult to realize in a direct way, since
the surface area is very sensitive to the shape of the set, and thus
any subtle change of the set could lead to drastic change in surface
area.

Our proof idea of Statement 1 of Theorem \ref{thm:equivalence} is
as follows, which is new and totally different from the semigroup
approach. Firstly, using a generalized Heintze-Karcher theorem \cite{morgan2005manifolds}
from geometric measure theory, we convert the standard isoperimetric
problem into a variant version in which the ``perimeter'' is measured
by the size of a thick shell of the set. Secondly, using information-theoretic
tools---the inherently typical subset lemma \cite{ahlswede1997identification}
and its consequence \cite{yu2024exact}---as well as covering arguments,
we tailor the set into a ``regular'' one the measure of whose thick
shell admits a simple formula, which is expressed in terms of nonlinear
transport-entropy inequality. The feasibility of this step is due
to the fact that the size of thick shell is robust to any slight change
of the set. Lastly, by establishing a nonlinear transport inequality
analogous to \cite{wang2008generalized} using gradient flow (heat
flow) from optimal transport \cite{villani2008optimal}, we simplify
the resulting formula and express it in terms of nonlinear log-Sobolev
inequality, leading to the equivalence between the standard isoperimetric
problem and the nonlinear log-Sobolev inequality.

Our specific strategy is illustrated in Fig. \ref{fig:Proof-strategy-of},
which is described in the order opposite to the one given above. Specifically,
using tools from optimal transport, in Section \ref{sec:From-Log-Sobolev-to}
we first deduce a specific form of nonlinear transport inequality
from the nonlinear log-Sobolev inequality. By tools from information
theory as well as covering and heat flow arguments, we show in Section
\ref{sec:From-Talagrand-} that the nonlinear transport inequality
in turn implies a variant isoperimetric inequality---the isoperimetric
inequality with thick boundaries. Lastly, by tools from geometric
measure theory, we show in Section \ref{sec:From-Variant-Isoperimetry}
that the variant isoperimetric inequality implies the original isoperimetric
inequality. Putting everything together, in Section \ref{sec:Combining-All-Above}
we deduce that the nonlinear log-Sobolev inequality implies the isoperimetric
inequality. As for the other direction, we provide two proofs in Section
\ref{sec:From-Isoperimetry-to}: one is based on information theory
and optimal transport, and the other is based on geometric measure
theory.

We prove Statement 2 of Theorem \ref{thm:equivalence}, Theorem \ref{thm:approximate},
Theorem \ref{thm:comparison}, Theorem \ref{thm:I_inf}, and Theorem
\ref{thm:clt} respectively in Appendices A--E.

\section{From Log-Sobolev to Transport}\label{sec:From-Log-Sobolev-to}

In this section, we connect log-Sobolev inequality and transport-entropy
inequality. Let $(\mathcal{X},d,\nu)$ be a Polish (complete and separable)
metric probability measure space, equipped with the $\nu$-completion
of the Borel $\sigma$-algebra.

We now need some metric notions; refer to \cite{ambrosio2014calculus,gigli2013log}
for details. A curve $\gamma:[0,\infty)\to\X$ is said to be absolutely
continuous if there exists $g\in L^{1}(0,\infty)$ such that 
\begin{equation}
d(\gamma_{s},\gamma_{t})\leq\int_{s}^{t}g(r)\,\d r,\quad\forall0\le s<t<\infty.\label{eq:-2}
\end{equation}
In this case, the metric speed of $\gamma$ is well-defined for a.e.
$t$ by 
\[
|\dot{\gamma}_{t}|:=\lim_{h\to0}\frac{d(\gamma_{t+h},\gamma_{t})}{|h|}.
\]
It turns out that $|\dot{\gamma}|\in L^{1}(0,1)$ and this is the
minimal $L^{1}$ function for which \eqref{eq:-2} holds.

Given a function $f:X\to\mathbb{R}$, the local Lipschitz constant
$|\nabla f|:\X\to[0,\infty]$ is defined by 
\[
|\nabla f|(x):=\limsup_{y\to x}\frac{|f(x)-f(y)|}{d(x,y)}.
\]
Define the Cheeger energy functional $\operatorname{Ch}:L^{2}(\X,\nu)\to[0,\infty]$
by 
\begin{equation}
\operatorname{Ch}(f):=\inf\liminf_{n\to\infty}\frac{1}{2}\int|\nabla f_{n}|^{2}\,\d\nu,\label{eq:Ch}
\end{equation}
where the infimum is taken among all sequences $(f_{n})$ of Lipschitz
functions converging to $f$ in $L^{2}(\X,\nu)$. Then the Sobolev
space $W^{1,2}(\mathcal{X},d,\nu)$ is defined as the set of $f\in L^{2}(\X,\nu)$
such that $\operatorname{Ch}(f)<\infty$ endowed with the norm 
\begin{equation}
\|f\|_{W^{1,2}}:=\|f\|_{L^{2}}^{2}+2\operatorname{Ch}(f).\label{eq:-32}
\end{equation}

A function $G\in L^{2}(\X,\nu)$ is said to be a relaxed gradient
of $f\in L^{2}(\X,\nu)$ if there exist Lipschitz functions $f_{n}\in L^{2}(\X,\nu)$
such that: 
\begin{itemize}
\item $f_{n}\to f$ in $L^{2}(\X,\nu)$ and $|\nabla f_{n}|$ weakly converge
to $\tilde{G}$ in $L^{2}(\X,\nu)$; 
\item $\tilde{G}\leq G$ $\nu$-a.e. in $\X$. 
\end{itemize}
A function $G$ is said to be the minimal relaxed gradient of $f$
if its $L^{2}(\X,\nu)$ norm is minimal among relaxed gradients. We
shall denote by $|\nabla f|_{*}$ the minimal relaxed gradient. In
terms of the minimal relaxed gradient, the Cheeger energy can be equivalently
expressed as 
\[
\operatorname{Ch}(f)=\frac{1}{2}\int|\nabla f|_{*}^{2}\,\d\nu.
\]

For a probability measure $\mu$ defined on $\mathcal{X}$ with density
$f$ w.r.t. $\nu$, the Fisher information is defined as 
\[
I(\mu\|\nu):=\int_{\{f>0\}}\frac{|\nabla f|_{*}^{2}}{f}\d\nu,
\]
which reduces to the definition in when the space is a smooth Riemannian
manifold. Recall the definition of the relative entropy: 
\begin{align*}
D(\mu\|\nu) & :=\int f\ln f\d\nu.
\end{align*}

We first the following equivalent formulations of the nonlinear log-Sobolev
inequality, which is a nonlinear analogue of the linear version given
in \cite[Proposition 1]{gigli2013log}. The proof is almost same as
the linear version, and omitted here. 
\begin{prop}[Equivalent formulations of the nonlinear log Sobolev inequality \cite{gigli2013log}]
Let $(\X,d,\nu)$ be a Polish metric measure space and let $\theta$
be a continuous function. Then the following are equivalent: 
\begin{itemize}
\item For any Lipschitz function $f:\X\to[0,\infty)$ with $\int f\,\d\nu=1$,
it holds 
\[
\theta(\int f\ln f\,\d\nu)\leq\int_{\{f>0\}}\frac{|\nabla f|^{2}}{f}\,\d\nu.
\]
\item For any non-negative $f\in W^{1,2}(\X,d,\nu)$ with $\int f\,\d\nu=1$,
it holds 
\[
\theta(\int f\ln f\,\d\nu)\leq\int_{\{f>0\}}\frac{|\nabla f|_{*}^{2}}{f}\,\d\nu.
\]
\end{itemize}
\end{prop}

Given $f\in W^{1,2}(\X,d,\nu)$, we write $\partial^{-}\mathrm{Ch}(f)\subset L^{2}(\X,\nu)$
for the subdifferential at $f$ of the restriction to $L^{2}(\X,\nu)$
of Cheeger's energy, namely, $\xi\in\partial^{-}\mathrm{Ch}(f)$ if
and only if 
\[
\mathrm{Ch}(g)\ge\mathrm{Ch}(f)+\int\xi(g-f)\,\mathrm{d}\nu\quad\forall g\in L^{2}(\X,\nu).
\]
Notice that the functional $\operatorname{Ch}:L^{2}(\X,\nu)\to[0,\infty]$
is convex, lower semicontinuous and with dense domain, so that the
standard theory of gradient flows on Hilbert space applies. This means
that \cite[Theorem 2.13]{AmbrosioGigliSavare14} for any $f\in L^{2}(\X,\nu)$
there exists a unique locally absolutely continuous curve $(f_{t})\subset L^{2}(\X,\nu)$
on $[0,\infty)$ such that $f_{0}=f$ and it holds 
\begin{equation}
\frac{\mathrm{d}}{\mathrm{d}t}f_{t}\in-\partial^{-}\operatorname{Ch}(f_{t}),\quad\text{a.e. }t>0,\label{eq:gradCh}
\end{equation}
which is the gradient flow of the Cheeger energy. The operator $P_{t}$
given by $f_{0}\mapsto f_{t}$ is known as the heat semigroup in $L^{2}(\X,\nu)$.

In \cite{ambrosio2014calculus} properties of this gradient flow have
been studied extensively, especially in connection with concepts coming
from optimal transport theory like Wasserstein distance and relative
entropy. In particular, the main results of \cite{ambrosio2014calculus}
are critical in the identification of the gradient flows for the Cheeger
energy in $L^{2}(\X,\nu)$ and for the entropy functional in $(\mathscr{P}_{2}(\X),\W)$.

It is well known that Sobolev-type inequalities (including log-Sobolev
inequalities) imply Talagrand's transport inequalities, as shown by
Otto and Villani \cite{otto2000generalization,villani2008optimal}
on Euclidean spaces and Riemannian manifolds and extended by Gigli
and Ledoux \cite{gigli2013log} to Polish metric measure spaces. A
nonlinear analogue of this result was given by F. Y. Wang \cite[Theorem 2.1]{wang2008generalized}
which is proven by extending Otto and Villani's idea. Specifically,
F. Y. Wang showed that if a weighted manifold $(\M,\nu)$ admits a
$\theta$-nonlinear log-Sobolev inequality, then the following nonlinear
transport inequality holds: 
\begin{equation}
\W\left(\mu,\nu\right)\le\Upsilon(0,D(\mu\|\nu)),\;\forall\mu\in\P_{2}^{\mathrm{ac}}(\M),\label{eq:-66-1}
\end{equation}
where 
\[
\Upsilon(s,\alpha):=\int_{s}^{\alpha}\frac{1}{\sqrt{\theta(r)}}\d r.
\]
In the present paper, we require a variant of F. Y. Wang's result.
Given $\mu_{W}$, we consider the conditional Wasserstein metric $\W\left(\mu_{X|W},\mu_{Y|W}|\mu_{W}\right)$
as a distance between conditional probabilities $\mu_{X|W}$ and $\mu_{Y|W}$. 
\begin{thm}[From Nonlinear Log-Sobolev to Nonlinear Transport]
\label{thm:Sobolev-Talagrand} Let $(\mathcal{X},d)$ be a Polish
(complete and separable) metric space and $\nu$ is a Borel probability
measure on $\mathcal{X}$ which admits a $\theta$-nonlinear log-Sobolev
inequality for some function $\theta$. Let $\mu_{X_{0}W}$ be a joint
probability measure such that $\mu_{X_{0}}=f_{0}\nu\in\P_{2}^{\mathrm{ac}}(\mathcal{X})$
and $D(\mu_{X_{0}|W}\|\nu|\mu_{W})<\infty$. Let $\mu_{X_{t}}=f_{t}\nu$
be the heat flow with initial data $\mu_{X_{0}}$ and denote the induced
transition probability $\mu_{X_{t}|X_{0}}$. Let $\mu_{X_{t}X_{0}W}=\mu_{X_{0}W}\mu_{X_{t}|X_{0}}$.
Then, the curve $t\mapsto\mu_{X_{t}|W}$ is absolutely continuous
on $[0,\infty)$ with respect to the conditional Wasserstein metric
$\W$ and for its metric speed $|\dot{\mu}_{X_{t}|W}|$, it holds
that $|\dot{\mu}_{X_{t}|W}|^{2}\le I(\mu_{X_{t}|W}\|\nu|\mu_{W})$
for a.e. $t$. Moreover, the map $t\mapsto D(\mu_{X_{t}|W}\|\nu|\mu_{W})$
is absolutely continuous and non-increasing on $[0,\infty)$, and
\begin{equation}
\W\left(\mu_{X_{0}|W},\mu_{X_{t}|W}|\mu_{W}\right)\le\Upsilon(D(\mu_{X_{t}|W}\|\nu|\mu_{W}),D(\mu_{X_{0}|W}\|\nu|\mu_{W})).\label{eq:-33}
\end{equation}
\end{thm}

\begin{rem}
If $W$ is constant, then the unconditional version is obtained. This
theorem can be generalized to the $\Phi$-entropy version by incorporating
some ideas from \cite{villani2008optimal} to the proof. 
\end{rem}

\begin{proof}
Since $\mu_{X_{0}}\in\P_{2}^{\mathrm{ac}}(\mathcal{X})$, we have
$\mu_{X_{0}|W=w}\in\P_{2}^{\mathrm{ac}}(\mathcal{X})$ for $\mu_{W}$-a.e.
$w$. Observe that for $0\le s<t\le1$, 
\begin{align}
\W^{2}\left(\mu_{X_{s}|W},\mu_{X_{t}|W}|\mu_{W}\right) & =\mathbb{E}_{\mu_{W}}[\W^{2}\left(\mu_{X_{s}|W},\mu_{X_{t}|W}\right)]\nonumber \\
 & \le(t-s)\mathbb{E}_{\mu_{W}}[\int_{s}^{t}I(\mu_{X_{r}|W}\|\nu)\d r]\nonumber \\
 & =(t-s)\int_{s}^{t}I(\mu_{X_{r}|W}\|\nu|\mu_{W})\d r,\label{eq:-7}
\end{align}
where the inequality follows since $\W\left(\mu_{X_{s}|W=w},\mu_{X_{t}|W=w}\right)\le(t-s)\int_{s}^{t}I(\mu_{X_{r}|W=w}\|\nu)\d r$
for $\mu_{W}$-a.e. $w$; see \cite[p. 359]{ambrosio2014calculus}.
Hence, for $0\le s<t\le1$, 
\[
\W\left(\mu_{X_{s}|W},\mu_{X_{t}|W}|\mu_{W}\right)\le\frac{1}{2}(t-s+\int_{s}^{t}I(\mu_{X_{r}|W}\|\nu|\mu_{W})\d r).
\]
By the triangle inequality for $\W,$ this inequality extends to the
case of $0\le s<t<\infty$. The integral here is finite for any finite
$t$ since 
\begin{align}
\int_{0}^{t}I(\mu_{X_{r}|W}\|\nu|\mu_{W})\d r & =\mathbb{E}_{\mu_{W}}[\int_{0}^{t}I(\mu_{X_{r}|W}\|\nu)\d r]\nonumber \\
 & =\mathbb{E}_{\mu_{W}}[D(\mu_{X_{0}|W}\|\nu)-D(\mu_{X_{t}|W}\|\nu)]\label{eq:-3}\\
 & =D(\mu_{X_{0}|W}\|\nu|\mu_{W})-D(\mu_{X_{t}|W}\|\nu|\mu_{W}),\label{eq:-31}
\end{align}
and hence 
\begin{equation}
\int_{0}^{\infty}I(\mu_{X_{r}|W}\|\nu|\mu_{W})\d r\le D(\mu_{X_{0}|W}\|\nu|\mu_{W})<\infty\label{eq:-4}
\end{equation}
where \eqref{eq:-3} follows by \cite[Proposition 4.22]{ambrosio2014calculus}.
Hence, $t\mapsto\mu_{X_{t}|W}$ is absolutely continuous on $[0,\infty)$.
From \eqref{eq:-7}, we estimate $|\dot{\mu}_{X_{t}|W}|^{2}\le I(\mu_{X_{t}|W}\|\nu|\mu_{W})$
for a.e. $t$.

Equations \eqref{eq:-31} and \eqref{eq:-4} imply that $t\mapsto D(\mu_{X_{t}|W}\|\nu|\mu_{W})$
is absolutely continuous and nonincreasing on $[0,\infty)$. In fact,
\eqref{eq:-31} further implies that $\frac{\d}{\d t}D(\mu_{X_{t}|W}\|\nu|\mu_{W})=-I(\mu_{X_{t}|W}\|\nu|\mu_{W})$
for a.e. $t$.

We next prove \eqref{eq:-33} by following Otto and Villani's idea
in \cite{otto2000generalization} and Gigli and Ledoux's extension
\cite{gigli2013log}. For $\mu_{W}$-a.e. $w$, $\theta(D(\mu_{X|W=w}\|\nu))\leq I(\mu_{X|W=w}\|\nu).$
Taking expectation and by Jensen's inequality, 
\[
\theta(D(\mu_{X|W}\|\nu|\mu_{W}))\leq\mathbb{E}_{\mu_{W}}[\theta(D(\mu_{X|W}\|\nu))]\leq I(\mu_{X|W}\|\nu|\mu_{W}).
\]

Define 
\[
f(t)=\int_{0}^{t}\sqrt{I(\mu_{X_{r}|W}\|\nu|\mu_{W})}\d r+\Upsilon(0,D(\mu_{X_{t}|W}\|\nu|\mu_{W})).
\]
Then, it holds that for a.e. $t$, 
\begin{align*}
\frac{\d^{+}}{\d t}f(t) & \leq\sqrt{I(\mu_{X_{t}|W}\|\nu|\mu_{W})}+\frac{\frac{\d}{\d t}D(\mu_{X_{t}|W}\|\nu|\mu_{W})}{\sqrt{\theta(D(\mu_{X_{t}|W}\|\nu|\mu_{W}))}}\\
 & =\sqrt{I(\mu_{X_{t}|W}\|\nu|\mu_{W})}-\frac{I(\mu_{X_{t}|W}\|\nu|\mu_{W})}{\sqrt{\theta(D(\mu_{X_{t}|W}\|\nu|\mu_{W}))}}\;\le0,
\end{align*}
where $\frac{\d^{+}}{\d t}$ is the right derivative.

So, $f$ is non-increasing as a function of $t$, which implies $f(t)\le f(0)=\Upsilon(0,D(\mu_{X_{0}|W}\|\nu|\mu_{W}))$
for any $t\ge0$. Thus, 
\[
\int_{0}^{t}\sqrt{I(\mu_{X_{r}|W}\|\nu|\mu_{W})}\d r\le\Upsilon(D(\mu_{X_{t}|W}\|\nu|\mu_{W}),D(\mu_{X_{0}|W}\|\nu|\mu_{W})).
\]
Combining it with $|\dot{\mu}_{X_{t}|W}|^{2}\le I(\mu_{X_{t}|W}\|\nu|\mu_{W})$
for a.e. $t$. yields \eqref{eq:-33}. 
\end{proof}
We immediately obtain a corollary for Gaussian measure. 
\begin{cor}[From Nonlinear Log-Sobolev to Nonlinear Transport for Gaussians]
\label{cor:OU} Let $\gamma$ be the standard Gaussian measure on
$\mathbb{R}^{k}$. Let $X_{t}=e^{-t}X+\sqrt{1-e^{-2t}}Z$ be an Ornstein--Uhlenbeck
process on $\mathbb{R}^{k}$ with distribution $\mu_{t}$, where $X\sim\mu\in\P_{2}^{\mathrm{ac}}(\mathbb{R}^{k})$
and $Z\sim\gamma$ are independent. Then, for any $t\ge0$, 
\[
\sqrt{2D(\mu_{t}\|\gamma)}+\W(\mu_{t},\mu)\le\sqrt{2D(\mu\|\gamma)}.
\]
\end{cor}

The importance of Theorem \ref{thm:Sobolev-Talagrand} is as follows.
Given $\tau\ge0,\mu,\nu$, consider the following optimization problem:
\begin{equation}
\omega(\beta):=\inf_{\pi:\W(\pi,\mu)\le\beta}D(\pi\|\nu),\label{eq:-18}
\end{equation}
which characterizes the best tradeoff between $D(\pi\|\nu)$ and $\W(\pi,\mu)$
when we vary $\beta\ge0$. In other words, within the Wasserstein
ball of center $\mu$ and radius $\beta$, we aim at finding a distribution
$\pi$ which minimizes $D(\pi\|\nu)$. By Sanov's theorem in large
deviations theory \cite{Dembo}, $\omega(\beta)$ is the convergence
rate of the probability that the empirical measure of i.i.d. random
variables $X_{1},...,X_{n}$ with each $\sim\nu$ belongs to the Wasserstein
ball of center $\mu$ and radius $\beta$. This probability naturally
appears in concentration of measure and isoperimetric inequalities.

The Wasserstein gradient flow of the relative entropy starting from
$\mu$ forms a feasible solution to \eqref{eq:-18}. Specifically,
the Wasserstein gradient flow is a curve that evolves to decrease
$D(\pi\|\nu)$ most steeply with respect to the Wasserstein metric.
If it stops until touching the Wasserstein sphere of center $\mu$
and radius $\beta$, we then obtain a feasible solution to \eqref{eq:-18}
that is locally greedy. In other words, the Wasserstein gradient flow
$\mu_{t}$ yields an upper bound on $\omega(\beta)$, although in
general, $\mu_{t}$ is not the globally optimal. That is, 
\begin{align*}
\omega(\beta) & \le D(\mu_{T(\beta)}\|\nu),
\end{align*}
where $T(\beta)$ is the time $t$ such that $\W(\mu_{t},\mu)=\beta$,
corresponding to the escape time from the Wasserstein ball of center
$\mu$ and radius $\beta$. Optimization problems like the one in
\eqref{eq:-18} will be investigated in details in the next section.

\section{From Transport to Variant Isoperimetry }\label{sec:From-Talagrand-}

\subsection{Problem Formulation and Preliminaries}

We now consider a variant of the isoperimetric problem in which the
perimeter of a set is identified by the size of an enlargement shell
of the set. We then provide a solution to this variant isoperimetric
problem by using the nonlinear transport inequality obtained in the
last section.

Let $(\mathcal{X},d,\nu)$ be a Polish metric measure space with $\nu$
being a Borel probability measure. Denote $(\mathcal{X}^{n},d_{n},\nu_{n})$
as the $n$-fold product space of $\mathcal{X}$ equipped with the
product metric $d_{n}$ and product measure $\nu_{n}=\nu^{\otimes n}$.
Define the variant isoperimetric profile as for $a\in[0,1],t\ge0$,
\begin{equation}
\Gamma_{n}(a,t):=\inf_{A:\nu^{\otimes n}(A)\ge a}\nu^{\otimes n}(A^{t}).\label{eq:Gamma-2}
\end{equation}
We set $a=e^{-n\alpha}$ and $t=\sqrt{n\tau}$. Define the isoperimetric
exponent as for $\alpha,\tau\ge0$, 
\begin{align}
E_{n}(\alpha,\tau) & :=-\frac{1}{n}\ln\Gamma_{n}(e^{-n\alpha},\sqrt{n\tau}).\label{eq:-5}
\end{align}

Recall that the Wasserstein metric $\W(\pi_{X},\pi_{Y})$ (of order
$2$) between two probability measures $\pi_{X}$ and $\pi_{Y}$ is
defined by 
\begin{equation}
\W^{2}(\pi_{X},\pi_{Y}):=\min_{\pi\in\Pi(\pi_{X},\pi_{Y})}\int d^{2}(x,y)\d\pi(x,y),\label{eq:OT-2-1}
\end{equation}
where $\Pi(\pi_{X},\pi_{Y})$ is the set of joint probability measures
on $\mathcal{X}\times\mathcal{Y}$ with marginals equal to $\pi_{X},\pi_{Y}$
respectively. The existence of the minimizers are well-known; see,
e.g., \cite[Theorem 1.3]{villani2003topics}. For a probability measure
$\pi_{W}$ and two transition probabilities $\pi_{X|W}$ and $\nu_{X|W}$,
the conditional Wasserstein metric between $\pi_{X|W}$ and $\pi_{Y|W}$
given $\pi_{W}$ is defined by 
\[
\W^{2}(\pi_{X|W},\pi_{Y|W}|\pi_{W}):=\int\W^{2}(\pi_{X|W=w},\pi_{Y|W=w})\d\pi_{W}(w).
\]
The conditional relative entropy of $\pi_{X|W}$ w.r.t. $\nu_{X|W}$
given $\pi_{W}$ is defined as 
\[
D(\pi_{X|W}\|\nu_{X|W}|\pi_{W}):=\int D(\pi_{X|W=w}\|\nu_{X|W=w})\d\pi_{W}(w).
\]

Let $\nu_{X}$ and $\nu_{Y}$ be two probability measures on the space
$(\mathcal{X},d)$. Define $E_{n}(\alpha,\tau|\nu_{X},\nu_{Y})$ similarly
as the exponent in \eqref{eq:-5} but with $\Gamma_{n}$ replaced
by 
\[
\Gamma_{n}(a,t|\nu_{X},\nu_{Y}):=\inf_{A:\nu_{X}^{\otimes n}(A)\ge a}\nu_{Y}^{\otimes n}(A^{t}).
\]
Define 
\begin{align}
\psi_{\kappa}(\alpha,\tau|\nu_{X},\nu_{Y}) & :=\sup_{\substack{\pi_{XW}\in\P(\mathcal{X}\times\{0,1\}):\\
D(\pi_{X|W}\|\nu_{X}|\pi_{W})\le\alpha
}
}\eta_{\tau,\kappa}(\pi_{XW}|\nu_{Y}),\label{eq:-38}
\end{align}
where 
\begin{align}
\eta_{\tau,\kappa}(\pi_{XW}|\nu_{Y}) & :=\inf_{\substack{\pi_{Y|XW}:\mathbb{E}[d^{2}(X,Y)]\le\tau,\\
\mathbb{E}[\mathrm{Var}(d^{2}(X,Y)|X,W)]\le\kappa
}
}D(\pi_{Y|W}\|\nu_{Y}|\pi_{W}).\label{eq:-39}
\end{align}
Denote $\psi(\alpha,\tau|\nu_{X},\nu_{Y})=\psi_{\infty}(\alpha,\tau|\nu_{X},\nu_{Y})$,
equivalently in which case, the variance constraint in \eqref{eq:-39}
is removed. Without changing the optimal value, the range of $W$
in \eqref{eq:-38} can be changed to any larger finite set. This point
can be proven by following proof steps similar to that of \cite[Theorem 20]{yu2024exact}.

Based on $\psi$ and $\psi_{\kappa}$, the asymptotic expression of
$E_{n}$ for compactly supported $\nu_{X}$ is characterized in our
previous work, which is a generalization of Ahlswede and Zhang's work
\cite{ahlswede1999asymptotical}. 
\begin{thm}[{{Asymptotic Exponent for Compactly-Supported Measures \cite[Theorem 12]{yu2024exact}}}]
\label{thm:LD} Let $(\mathcal{X},d)$ be a Polish metric space equipped
with two probability measures $\nu_{X}$ and $\nu_{Y}$. The following
hold. 
\begin{enumerate}
\item Suppose $\nu_{X}$ is compactly supported. Then, for any $\alpha\ge0,\tau\ge0$,
\begin{equation}
\limsup_{n\to\infty}E_{n}(\alpha,\tau|\nu_{X},\nu_{Y})\le\lim_{\kappa\uparrow\infty}\psi_{\kappa}(\alpha,\tau|\nu_{X},\nu_{Y}).\label{eq:-38-8-1}
\end{equation}
In fact, a stronger result holds: Given $\delta,\kappa,\tau_{0}>0$,
for sufficiently large $n$ and any set $A\subseteq\mathcal{X}^{n}$
such that $\nu_{X}^{\otimes n}(A)\ge e^{-n\alpha}$, there is a probability
measure $\pi_{XW}\in\P(\mathcal{X}\times\mathbb{N})$ induced by $A$
such that $D(\pi_{X|W}\|\nu_{X}|\pi_{W})\le\alpha+\delta$ and for
any $\tau\in(\delta,\tau_{0}]$, 
\[
-\frac{1}{n}\ln\nu_{Y}^{\otimes n}(A^{\sqrt{n\tau}})\le(1+\delta)\eta_{\tau-\delta,\kappa}(\pi_{XW}|\nu_{Y})+\delta.
\]
\item For any $(\alpha,\tau)$ in the interior of $\mathrm{dom}\psi:=\left\{ (\alpha,\tau):\psi(\alpha,\tau|\nu_{X},\nu_{Y})<\infty\right\} $,
it holds that 
\begin{equation}
\liminf_{n\to\infty}E_{n}(\alpha,\tau|\nu_{X},\nu_{Y})\ge\psi(\alpha,\tau|\nu_{X},\nu_{Y}).\label{eq:-40}
\end{equation}
In particular, if further assume $\nu_{X}=\nu_{Y}$, then \eqref{eq:-40}
holds for any $\alpha\ge0,\tau\ge0$. 
\end{enumerate}
\end{thm}

\begin{rem}
Theorem \ref{thm:LD} still holds if the constraint $\pi_{XW}\in\P(\mathcal{X}\times\{0,1\})$
in \eqref{eq:-38} is replaced by $\pi_{XW}\in P_{\mathrm{c}}(\mathcal{X}\times\{0,1\})$. 
\end{rem}

Why is the asymptotic exponent determined by the relative entropy
and the Wasserstein metric? The intuition is as follows. For a distribution
$\pi$, denote $B_{\epsilon]}(\pi)$ as the closed ball of center
$\nu$ and radius $\epsilon$ under the Lévy--Prokhorov metric. Let
$\L_{n}:\mathbf{x}\in\mathcal{X}^{n}\mapsto\L_{\mathbf{x}}\in\P(\mathcal{X})$
be the empirical measure map, where $\L_{\mathbf{x}}:=\frac{1}{n}\sum_{i=1}^{n}\delta_{x_{i}}$
is the empirical measure of $x$, and $\P(\mathcal{X})$ is the set
of probability measures on $\mathcal{X}$. For a Polish space $\mathcal{X}$
and $\epsilon>0$, the empirically $\epsilon$-typical set of $\pi$
\cite{mitran2015on} is defined as $\mathcal{T}_{\epsilon}^{(n)}(\pi):=\L_{n}^{-1}(B_{\epsilon]}(\pi)).$

For an empirically $\epsilon$-typical set $A=\L_{n}^{-1}(B_{\epsilon}(\pi_{X}))$
(i.e., $A=\{\mathbf{x}\in\mathcal{X}^{n}:\L_{\mathbf{x}}\approx\pi_{X}\}$),
roughly speaking, by Sanov's theorem it holds that $\nu_{X}^{\otimes n}(A)\approx e^{-nD(\pi_{X}\|\nu_{X})}$
for $\epsilon$ chosen sufficiently small, which is no smaller than
$e^{-n\alpha}$ if $D(\pi_{X}\|\nu_{X})\le\alpha$. The $\sqrt{n\tau}$-enlargement
of this set can be written as 
\begin{align}
A^{\sqrt{n\tau}} & =\{\mathbf{y}:\frac{1}{n}d_{n}^{2}(\mathbf{x},\mathbf{y})\leq\tau,\exists\mathbf{x}\in A\}\nonumber \\
 & =\{\mathbf{y}:\mathbb{E}_{\L_{\mathbf{x},\mathbf{y}}}[d^{2}(X,Y)]\leq\tau,\L_{\mathbf{x}}\approx\pi_{X},\exists\mathbf{x}\}\nonumber \\
 & =\{\mathbf{y}:\W(\L_{\mathbf{x}},\L_{\mathbf{y}})\leq\sqrt{\tau},\L_{\mathbf{x}}\approx\pi_{X}\}\label{eq:-59}\\
 & \approx\L^{-1}\{\pi_{Y}:\W(\pi_{X},\pi_{Y})\leq\sqrt{\tau}\},\nonumber 
\end{align}
where \eqref{eq:-59} is due to Birkhoff's theorem, $\min_{\sigma}\mathbb{E}_{\L_{\sigma(\mathbf{x}),\mathbf{y}}}[d^{2}(X,Y)]=\W^{2}(\L_{\mathbf{x}},\L_{\mathbf{y}})$,
with $\sigma$ denoting a permutation operator \cite[Page 5]{villani2003topics}.
By Sanov's theorem again, 
\[
\nu_{Y}^{\otimes n}(A^{\sqrt{n\tau}})\approx\exp\{-n\inf_{\pi_{Y}:\W(\pi_{X},\pi_{Y})\leq\sqrt{\tau}}D(\pi_{Y}\|\nu_{Y})\}.
\]
Optimizing over $\pi_{X}$ such that $D(\pi_{X}\|\nu_{X})\le\alpha$,
we obtain the isoperimetric exponent for $A$, 
\[
\sup_{\pi_{X}:D(\pi_{X}\|\nu_{X})\le\alpha}\inf_{\pi_{Y}:\W(\pi_{X},\pi_{Y})\leq\sqrt{\tau}}D(\pi_{Y}\|\nu_{Y}).
\]
By product construction of such $A$'s (forming a conditional empirically
typical set), we obtain the exponent $\psi(\alpha,\tau|\nu_{X},\nu_{Y})$.

The intuition above provides a proof idea for Statement 2 of Theorem
\ref{thm:LD}. However, Statement 1 is highly nontrivial and its proof
relies on the so-called the inherently typical subset lemma \cite{ahlswede1997identification}.
This lemma enables us to tailor a set into a new one that is ``regular''
in some sense, so that we can analyze the asymptotic expression of
the boundary size of the set.

\subsection{Riemannian Energy Measure Spaces Satisfying $\text{BE}(K,\infty)$
and $\operatorname{RCD}(K,\infty)$-spaces}\label{subsec:Riemannian-Energy-Measure}

Our goal in the whole section is to generalize Statement 1 of Theorem
\ref{thm:LD} to probability measures satisfying Condition \ref{cond:integrability}.
However, since the class of Polish metric measure spaces is too broad
to allow for straightforward derivation of bounds, we restrict our
analysis to a subclass of Polish spaces with favorable properties---namely,
Riemannian energy measure spaces satisfying the $\text{BE}(K,\infty)$
condition, or equivalently, $\text{RCD}(0,\infty)$-spaces. This choice
enables us to use the well-defined heat flow to establish bounds for
variant isoperimetry. To this end, we now introduce these two concepts;
for further details, we refer the reader to \cite{AmbrosioGigliSavare14,Ambrosio2015BakryEmery}.

We consider an\emph{ Energy measure space} $(\X,d,\nu)$, where $(\X,d)$
is a Polish metric space, $\nu$ is a Borel probability measure with
full support, and the underlying $\sigma$-algebra is the $\nu$-completion
of the Borel $\sigma$-algebra. The structure is governed by the Cheeger
energy $\text{Ch}:L^{2}(\X,\nu)\to[0,\infty]$ given in \eqref{eq:Ch},
defined as the $L^{2}$-relaxation of the local Lipschitz constant.
The space is assumed to be \emph{infinitesimally Hilbertian} in the
following sense. 
\begin{itemize}
\item $\mathcal{E}=2\text{Ch}$ is a quadratic form that constitutes a strongly
local, symmetric Dirichlet form $(\mathcal{E},\mathbb{V})$ with domain
$\mathbb{V}:=D(\mathcal{E})=\bigl\{ f\in L^{2}(\X,\nu):\mathcal{E}(f)<\infty\bigr\}$. 
\end{itemize}
Define the set $\mathbb{G}\subseteq\mathbb{V}$ such that for $f\in\mathbb{G}$,
the linear form 
\[
\boldsymbol{\Gamma}[f;\varphi]:=\mathcal{E}(f,f\varphi)-\tfrac{1}{2}\mathcal{E}(f^{2},\varphi),\;\varphi\in\mathbb{V}_{\infty}:=\mathbb{V}\cap L^{\infty}(\X,\nu)
\]
is represented by the $L_{+}^{1}$ density $\Gamma(f)=|\nabla f|_{*}^{2}$,
satisfying $\mathcal{E}(f,f)=\int\Gamma(f)\,\d\nu$. Define $\Gamma(f,g)=\frac{1}{4}\Gamma(f+g)-\frac{1}{4}\Gamma(f-g)$
which is known as the carré du champ operator. By the uniform continuity
property, we extend $\boldsymbol{\Gamma}$ to a real multilinear map
defined in $\mathbb{V}\times\mathbb{V}\times\mathbb{V}_{\infty}$,
which satisfies 
\[
\mathbf{\Gamma}[f,g;\varphi]=\int\Gamma(f,g)\varphi\,\d\nu\quad\text{if }f,g\in\mathbb{G}.
\]

Define $\mathcal{L}:=\{\psi\in\mathbb{G}:\Gamma(\psi)\leq1\quad\nu\text{-a.e. in }\X\}$
and $\mathcal{L}_{C}:=\mathcal{L}\cap C(\X)$. The compatibility between
the metric $d$ and the energy $\mathcal{E}$ is enforced by the following
structural assumptions: 
\begin{itemize}
\item Identification: $\mathcal{L}=\mathcal{L}_{C}$. That is, every function
with a gradient density bounded by one $\nu$-a.e. admits a continuous
representative. 
\item Lipschitz Consistency: Every function in $\mathcal{L}_{C}$ is $1$-Lipschitz
with respect to $d$. 
\end{itemize}
The above three assumptions turns the space $(\X,d,\nu)$ described
above into a \emph{Riemannian Energy measure space} in which the intrinsic
distance $d_{\mathcal{E}}$ induced by the Dirichlet form coincides
with the distance $d$, generalizing smooth Riemannian manifolds to
the possibly non-smooth setting.

The Dirichlet form $\mathcal{E}$ induces a densely defined self-adjoint
operator $\Delta_{\mathcal{E}}:D(\Delta_{\mathcal{E}})\subset\mathbb{V}\to L^{2}(\X,\nu)$
defined by the integration by parts formula $\mathcal{E}(f,g)=-\int g\Delta_{\mathcal{E}}f\,\d\nu$
for all $g\in\mathbb{V}$. This self-adjoint operator generates the
heat flow $(P_{t})_{t\ge0}$, an analytic Markov semigroup characterized
by $\frac{\mathrm{d}}{\mathrm{d}t}P_{t}f=\Delta_{\mathcal{E}}P_{t}f$,
a specialization of the heat equation in \eqref{eq:gradCh} in the
Riemannian Energy measure space.

Bakry--Émery conditions are formulated via the iterated gradient
form $\boldsymbol{\Gamma}_{2}$. For $(f,\varphi)\in D(\boldsymbol{\Gamma}_{2})$,
define: 
\[
\boldsymbol{\Gamma}_{2}[f;\varphi]:=\tfrac{1}{2}\boldsymbol{\Gamma}[f;\Delta_{\mathcal{E}}\varphi]-\boldsymbol{\Gamma}[f,\Delta_{\mathcal{E}}f;\varphi],
\]
where 
\[
D(\Gamma_{2}):=\left\{ (f,\varphi)\in D(\Delta_{\mathcal{E}})\times D(\Delta_{\mathcal{E}}):\Delta_{\mathcal{E}}f\in\mathbb{V},\,\varphi,\Delta_{\mathcal{E}}\varphi\in L^{\infty}(\X,\nu)\right\} .
\]
We can write 
\[
\mathbf{\Gamma}_{2}[f;\varphi]=\int\left(\frac{1}{2}\Gamma(f)\Delta_{\mathcal{E}}\varphi-\Gamma(f,\Delta_{\mathcal{E}}f)\varphi\right)\d\nu\quad\text{if }(f,\varphi)\in D(\mathbf{\Gamma}_{2}),\ f,\Delta_{\mathcal{E}}f\in\mathbb{G}.
\]
The space satisfies the \emph{Bakry--Émery condition} $\text{BE}(K,\infty)$
if for every $\varphi\in D(\Delta_{\mathcal{E}})$ with $\varphi\geq0$:
\[
\boldsymbol{\Gamma}_{2}[f;\varphi]\geq K\boldsymbol{\Gamma}[f;\varphi].
\]
In the setting of smooth weighted Riemannian manifolds, this condition
can be written as $\frac{1}{2}\Gamma(f)\Delta_{\mathcal{E}}-\Gamma(f,\Delta_{\mathcal{E}}f)\ge K\Gamma(f)$,
or equivalently, $\mathrm{Ric}_{V}\ge Kg$ for probability measure
$\nu=e^{-V}\mathrm{vol}$. Under $\text{BE}(K,\infty)$, it is known
that $\mathbb{G}=\mathbb{V}$.

The Riemannian Energy measure space $(\X,d,\nu)$ (with $d=d_{\mathcal{E}}$)
satisfies $\text{BE}(K,\infty)$ if and only if it is an $\text{RCD}(K,\infty)$
space \cite[Theorem 4.17]{Ambrosio2015BakryEmery}. From the perspective
of the geometry of optimal transport, this equivalence implies that
the relative entropy functional $D(\cdot\|\nu)$ is not merely $K$-geodesically
convex on the Wasserstein space $(\mathscr{P}_{2}(\X),\W)$, but that
this convexity is robust: it extends to hold along all interpolated
measures induced by weighted optimal geodesic plans (or test plans).

A fundamental consequence of the $\operatorname{RCD}(K,\infty)$ property
(equivalently, the Riemannian Energy measure space with $\text{BE}(K,\infty)$
property) is the identification of the metric $\W$-gradient flow
$\mathcal{H}_{t}$ with the analytic $L^{2}$-heat flow $P_{t}$.
For any $f\in L^{2}(\X,\nu)$ such that $f\nu\in\mathscr{P}_{2}(\X)$,
the identification $\mathcal{H}_{t}(f\nu)=(P_{t}f)\nu$ holds. Furthermore,
the heat flow admits the representation: 
\[
(P_{t}f)\nu=\int f(x)\mathcal{H}_{t}(\delta_{x})\,\d\nu(x),
\]
characterizing the evolved measure as a weighted superposition of
the evolved Dirac masses.

\subsection{Variant Isoperimetry for Noncompactly-Supported Measures}\label{subsec:Variant-Isoperimetry-for}

We next extend Statement 1 of Theorem \ref{thm:LD} to the Riemannian
Energy measure space satisfying $\text{BE}(K,\infty)$ and Condition
\ref{cond:integrability} (equivalently, a $\mathrm{RCD}(0,\infty)$-space
satisfying Condition \ref{cond:integrability}). To this end, we need
the following conditions on the covering bound, volume ratio bound,
and curvature-dimension bound.

\begin{condition}[Covering Bound]\label{cond:covering} Let $x_{0}\in\mathcal{X}$
be a fixed point. Then, for any $n\ge1$, $\lambda\ge1$, and $r>0$,
there is cover of the large ball $B_{\sqrt{n}\lambda r}(\mathbf{x}_{0})$
with $\mathbf{x}_{0}=(x_{0},...,x_{0})\in\mathcal{X}^{n}$ that consists
of small balls $B_{\sqrt{n}r}(\mathbf{x}_{i})$ with $\mathbf{x}_{i}\in\mathcal{X}^{n}$
and $1\le i\le N$, i.e., $B_{\sqrt{n}\lambda r}(\mathbf{x}_{0})\subseteq\bigcup_{1\le i\le N}B_{\sqrt{n}r}(\mathbf{x}_{i})$,
where $N=N_{n,\lambda,r}$ satisfies 
\[
\lim_{r\to\infty}\limsup_{n\to\infty}\frac{1}{nr^{2}}\ln N_{n,\lambda,r}=0,\;\forall\lambda\ge1.
\]
\end{condition}

\begin{condition}[Volume Ratio Bound]\label{cond:volume} There
is a nonnegative (not necessarily finite or probability) measure $\mathfrak{m}$
on $\mathcal{X}^{n}$ such that for all $\lambda\ge1$, 
\begin{equation}
\lim_{r\to\infty}\limsup_{n\to\infty}\frac{1}{nr^{2}}\ln N_{n}(\sqrt{n}\lambda r,\sqrt{n}r)=0,\label{eq:-44}
\end{equation}
where 
\begin{equation}
N_{n}(R,r)=\sup_{\mathbf{x}\in\mathcal{X}^{n}}\frac{\mathfrak{m}(B_{R}(\mathbf{x}))}{\mathfrak{m}(B_{r}(\mathbf{x}))}.\label{eq:-45}
\end{equation}
\end{condition}

\begin{condition}[Curvature-Dimension Bound]\label{cond:CD}\textbf{
}Suppose $\mathcal{X}$ is a subset of a Riemannian manifold $(\M,g)$
or the interior of $\mathcal{X}$ is a Riemannian manifold $(\M,g)$.
There is a nonnegative (not necessarily finite or probability) measure
$\mathfrak{m}=e^{-U}\mathrm{vol}$ on manifold $\M$ with $U\in C^{2}(\M)$
satisfying $\mathrm{CD}(K,N)$ for some $N\in(k,\infty)$ and $K\in\mathbb{R}$
with $k=\mathrm{dim}(\M)$: 
\[
\mathrm{Ric}_{U}^{N}:=\mathrm{Ric}_{g}+\mathrm{Hess}_{g}U-\frac{\nabla U\otimes\nabla U}{N-k}\geq Kg.
\]
\end{condition}

The conditions specified above and Condition \ref{cond:Ricci} satisfy
the following implication relation. 
\begin{prop}
\label{prop:Condition} It holds that Condition \ref{cond:volume}
$\Longrightarrow$ Condition \ref{cond:covering}. Moreover, if $\mathcal{X}$
is a subset of a Riemannian manifold $(\M,g)$ or the interior of
$\mathcal{X}$ is a Riemannian manifold $(\M,g)$, then Condition
\ref{cond:Ricci} $\Longrightarrow$ Condition \ref{cond:CD} $\Longrightarrow$
Condition \ref{cond:volume} $\Longrightarrow$ Condition \ref{cond:covering}. 
\end{prop}

Define 
\[
\Upsilon_{\alpha}(s)=\int_{s}^{\alpha}\frac{1}{\sqrt{\theta(r)}}\d r.
\]
Define 
\begin{align}
\bar{\psi}(\alpha,\tau) & =\begin{cases}
\Upsilon_{\alpha}^{-1}(\sqrt{\tau}), & \Upsilon_{\alpha}(0)\ge\sqrt{\tau}\\
0, & \Upsilon_{\alpha}(0)<\sqrt{\tau}
\end{cases},\label{eq:-63-2-1}
\end{align}
where for each $\alpha$, denote $\Upsilon_{\alpha}^{-1}$ as the
inverse of $\Upsilon_{\alpha}$. We are ready to extend Statement
1 of Theorem \ref{thm:LD} to all probability measures satisfying
Condition \ref{cond:integrability} or equivalently admitting a standard
log-Sobolev inequality. 
\begin{thm}[Asymptotic Exponent for Noncompactly-Supported Measures]
\label{thm:LD2} Let $(\mathcal{X},d,\nu)$ be a Riemannian Energy
measure space satisfying $\text{BE}(K,\infty)$ (equivalently, it
is a $\mathrm{RCD}(0,\infty)$-space). Suppose that this space satisfies
Condition \ref{cond:integrability} and it admits a $\theta$-nonlinear
log-Sobolev inequality. If the condition in \eqref{eq:Gaussian-2}
holds, we further assume Condition \ref{cond:covering} (or Conditions
\ref{cond:Ricci}, \ref{cond:CD}, or \ref{cond:volume}) holds. Then,
for any $\alpha\ge0,\tau\ge0$, 
\[
\limsup_{n\to\infty}E_{n}(\alpha,\tau)\le\bar{\psi}(\alpha,\tau).
\]
\end{thm}

\begin{rem}
\label{rem:Under}Under $\mathrm{RCD}(0,\infty)$ (equivalently, the
Riemannian Energy setting with $\mathrm{BE}(0,\infty)$) for a probability
measure $\nu\in\P_{2}(\mathcal{X})$, Condition \ref{cond:integrability}
holds if and only if it admits a standard log-Sobolev inequality \cite[Theorem 3.2]{bakry2004functional}. 
\end{rem}

\begin{rem}
Define 
\begin{align}
\psi(\alpha,\tau) & :=\psi(\alpha,\tau|\nu,\nu)\nonumber \\
 & =\sup_{\pi_{XW}\in\P(\mathcal{X}\times\{0,1\}):D(\pi_{X|W}\|\nu|\pi_{W})\le\alpha}\nonumber \\
 & \qquad\inf_{\pi_{Y|W}:\W(\pi_{X|W},\pi_{Y|W}|\pi_{W})\le\sqrt{\tau}}D(\pi_{Y|W}\|\nu|\pi_{W}),\label{eq:-41}
\end{align}
and 
\[
\tilde{\psi}(\alpha,\tau)=\frac{1}{2}\left[\sqrt{2\alpha}-\sqrt{K_{LS}^{+}\tau}\right]_{+}^{2},
\]
where recall the definition of $K_{LS}^{+}$ in \eqref{eq:-11}. By
checking our proof, the following bound can be obtained: 
\[
\limsup_{n\to\infty}E_{n}(\alpha,\tau)\le\sup_{\alpha_{0}+\alpha_{1}\le\alpha}\inf_{\tau_{0}+\tau_{1}\le\tau}\psi(\alpha_{0},\tau_{0})+\tilde{\psi}(\alpha_{1},\tau_{1}).
\]
This bound is more consistent with the lower bound in Statement 2
of Theorem \ref{thm:LD}, and it coincides with the latter if $\psi\le\tilde{\psi}$. 
\end{rem}

\subsection{Proof of Proposition \ref{prop:Condition} }

By choosing $\mathfrak{m}$ as the Riemannian volume measure, Condition
\ref{cond:Ricci} $\Longrightarrow$ Condition \ref{cond:CD} is obvious.
We next show that Condition \ref{cond:volume} $\Longrightarrow$
Condition \ref{cond:covering}.

For a subset $A\subseteq\M^{n}$, denote by $N(A,r)$ the minimal
covering number (i.e., the smallest number of balls of radius $r$
needed to cover $A$) and by $M(A,r)$ the maximal packing number
(i.e., the largest number of disjoint balls of radius $r$ whose centers
lie in $A$). By the standard lemma for metric spaces, $N(A,r)\le M(A,r/2).$
Setting $A=B_{R}(\mathbf{x})$ yields that $N(B_{R}(\mathbf{x}),r)\le M(B_{R}(\mathbf{x}),r/2).$

For any ball $B(\mathbf{x}_{i},r/2)$ in the packing for $M(B_{R}(\mathbf{x}),r/2)$,
since $p_{i}\in B_{R}(\mathbf{x})$, all such balls are contained
in $B_{R+r/2}(\mathbf{x})$. Let $i^{*}$ minimize $\mathfrak{m}^{\otimes n}(B_{r/2}(\mathbf{x}_{i}))$
over all $i$. Combining this with the fact that $B_{R+r/2}(\mathbf{x})\subseteq B_{2R+r/2}(\mathbf{x}_{i^{*}})$,
we obtain: 
\[
M(B_{R}(\mathbf{x}),r/2)\cdot\mathfrak{m}^{\otimes n}(B_{r/2}(\mathbf{x}_{i^{*}}))\le\mathfrak{m}^{\otimes n}(B_{R+r/2}(\mathbf{x}))\le\mathfrak{m}^{\otimes n}(B_{2R+r/2}(\mathbf{x}_{i^{*}})).
\]

Therefore, for any $\mathbf{x}$ and any $R>r>0$, 
\[
N(B_{R}(\mathbf{x}),r)\le M(B_{R}(\mathbf{x}),r/2)\le\frac{\mathfrak{m}^{\otimes n}(B_{2R+r/2}(\mathbf{x}_{i^{*}}))}{\mathfrak{m}^{\otimes n}(B_{r/2}(\mathbf{x}_{i^{*}}))}\le N_{n}(2R+r/2,r/2).
\]

Substituting $(R,r)\leftarrow(\sqrt{n}\lambda r,\sqrt{n}r)$ into
the inequality above and by the condition \eqref{eq:-44}, we obtain
that 
\[
\lim_{r\to\infty}\limsup_{n\to\infty}\frac{1}{nr^{2}}\ln N(B_{\sqrt{n}\lambda r}(\mathbf{x}_{0}),\sqrt{n}r)=0,\;\forall\lambda\ge1.
\]
Therefore, Condition \ref{cond:covering} holds.

We now prove that Condition \ref{cond:CD} $\Longrightarrow$ Condition
\ref{cond:volume}. Let $\mathfrak{m}=e^{-U}\mathrm{vol}$ be the
nonnegative measure on manifold $\M$ given in Condition \ref{cond:CD}.
Then, $\mathfrak{m}^{\otimes n}=e^{-U^{\oplus n}}\mathrm{vol}_{n}$
is a nonnegative measure on manifold $\M^{n}$ with $U^{\oplus n}\in C^{2}(\M^{n})$
such that $\mathrm{Ric}_{U^{\oplus n}}^{nN}\ge nK$.

For a point $p$ in $\M$ and $\xi\in T_{p}\M$ such that $|\xi|=1$,
let $\gamma_{\xi}(t)=\exp_{p}(t\xi)$ be the geodesic starting at
$p$ with direction $\xi$, and let $\lambda_{p}(t,\xi)=\lambda_{p}^{U}(t,\xi)$
be the density of the weighted measure $\mathfrak{m}$ at the point
$\gamma_{\xi}(t)$. The density $\lambda_{p}(t,\xi)$ is defined to
be zero whenever $t$ exceeds the distance to the cut point in the
direction $\xi$. This density incorporates both the Riemannian volume
distortion (Jacobian of the exponential map) and the weight function
$e^{-U}$. Let $\lambda^{(K)}(t)$ be the Riemannian volume density
in the $(\mathrm{dim}(\M)+N)$-dimensional model space $\M^{(K)}$
with constant Ricci curvature $K$. Following \cite[Theorem 7]{qian1997estimates},
under the condition $\mathrm{Ric}_{U}^{N}\ge Kg$, it holds that for
any $\xi$, $\frac{\lambda_{p}(t,\xi)}{\lambda^{(K)}(t)}$ is non-increasing
in $t$ for $\gamma_{\xi}(t)\notin\mathrm{cut}(p)$, where $\mathrm{cut}(p)$
is the cut locus of the Riemannian manifold $M$ with respect to the
point $p$.

Let $A(p,t)$ be the $(\mathrm{dim}(\M)-1)$-dimensional $\mathfrak{m}$-weighted
area of a geodesic sphere of radius $t$ at center $p$ in $M$. Then,
for any $t>0$, 
\begin{align*}
A(p,t) & =\int_{S_{p}\M}\lambda_{p}(t,\xi)\,\d\mathrm{area}(\xi)
\end{align*}
where $S_{p}\M$ is the unit sphere in the tangent space $T_{p}\M$
at point $p$, and $\d\mathrm{area}(\xi)$ is the standard area element
of the unit $(n-1)$-sphere in the tangent space.

Let $A^{(K)}(t)$ be the $(\mathrm{dim}(\M)-1)$-dimensional Riemannian
area of a geodesic sphere of radius $t$ at center $p$ in the $\mathrm{dim}(\M)$-dimensional
model space $\M^{(K)}$ with constant Ricci curvature $K$. Then,
\begin{align*}
A^{(K)}(t) & =\lambda^{(K)}(t)\cdot|S_{p}\M^{(K)}|,
\end{align*}
where $S_{p}\M^{(K)}$ is the unit sphere in the tangent space $T_{p}\M^{(K)}$
at point $p$, and $|S_{p}\M^{(K)}|$ is the standard area of the
unit $(n-1)$-sphere in the tangent space. Note that the value of
$|S_{p}\M^{(K)}|$ is independent of $p$.

Observe that 
\[
\frac{A(p,t)}{A^{(K)}(t)}=\frac{1}{|S_{p}\M^{(K)}|}\int_{S_{p}\M}\frac{\lambda_{p}(t,\xi)}{\lambda^{(K)}(t)}\,\d\mathrm{area}(\xi),
\]
which is non-increasing in $t$. The volume of the ball with center
$\mathbf{p}=(p_{1},...,p_{n})$ and radius $R$ in the space $\M^{n}$
is 
\[
V_{n}(\mathbf{p},R)=\int_{\sum_{i=1}^{n}\rho_{i}^{2}\le R^{2}}\prod_{i=1}^{n}A(p_{i},\rho_{i})\d\rho_{i}.
\]
Similarly, the volume of the ball with radius $R$ (at any center)
in the space $(\M^{(K)})^{n}$ is 
\begin{align*}
V_{n}^{(K)}(R) & =\int_{\sum_{i=1}^{n}\rho_{i}^{2}\le R^{2}}\prod_{i=1}^{n}A^{(K)}(\rho_{i})\d\rho_{i}.
\end{align*}
We have the following key lemma. 
\begin{lem}
\label{lem:volumeratio}Given any $\mathbf{p}$, $f(R)=\frac{V_{n}(\mathbf{p},R)}{V_{n}^{(K)}(R)}$
is non-increasing. 
\end{lem}

\begin{proof}
To show that the ratio $f(R)=\frac{V_{n}(\mathbf{p},R)}{V_{n}^{(K)}(R)}$
is non-increasing, we utilize the properties of the Bishop-Gromov
comparison theorem and the behavior of ratios of integrals.

Step 1: As shown above, the ratio of the areas of geodesic spheres
$\frac{A(p,r)}{A^{(K)}(r)}$ is a non-increasing function of $r$.
Let $\rho=(\rho_{1},\dots,\rho_{n})$ be a vector in the integration
domain. For any scalar $r>0$ and a unit vector $\omega\in\ensuremath{S_{+}^{n-1}=\{\omega\in S^{n-1}:\omega_{i}\ge0\text{\ for\ all\ }i\}}$,
let $\rho=r\omega$. The product ratio is: $\frac{\prod_{i=1}^{n}A(p_{i},r\omega_{i})}{\prod_{i=1}^{n}A^{(K)}(r\omega_{i})}=\prod_{i=1}^{n}\frac{A(p_{i},r\omega_{i})}{A^{(K)}(r\omega_{i})}$.
Since each term $\frac{A(p_{i},r\omega_{i})}{A^{(K)}(r\omega_{i})}$
is non-increasing in $r$ (because $r\omega_{i}$ increases with $r$
for $\omega_{i}\ge0$), the product of these non-negative non-increasing
functions is also non-increasing in $r$.

Step 2: We express the volumes $V_{n}(R)$ and $V_{n}^{(K)}(R)$ using
polar coordinates in the $\rho$-space $\mathbb{R}^{n}$:$V_{n}(R)=V_{n}(\mathbf{p},R)=\int_{0}^{R}r^{n-1}\left(\int_{S_{+}^{n-1}}\prod A(p_{i},r\omega_{i})d\omega\right)dr$.
Let $G(r)=G(\mathbf{p},r)=r^{n-1}\int_{S_{+}^{n-1}}\prod A(p_{i},r\omega_{i})d\omega$
and $H(r)=r^{n-1}\int_{S_{+}^{n-1}}\prod A^{(K)}(r\omega_{i})d\omega$.
From Step 1, the ratio of the integrands $\frac{\prod A(p_{i},r\omega_{i})}{\prod A^{(K)}(r\omega_{i})}$
is non-increasing in $r$ for every $\omega$. Consequently, the ratio
of the integrals over the angular components, $\frac{G(r)}{H(r)}$,
is also non-increasing.

Step 3: A standard lemma in analysis states that if $G(r)$ and $H(r)$
are positive functions such that $\frac{G(r)}{H(r)}$ is non-increasing,
then the ratio of their cumulative integrals $F(R)=\frac{\int_{0}^{R}G(r)dr}{\int_{0}^{R}H(r)dr}$
is also non-increasing. To see this, calculate the derivative $F^{\prime}(R)$:
$F^{\prime}(R)=\frac{G(R)\int_{0}^{R}H(r)dr-H(R)\int_{0}^{R}G(r)dr}{(\int_{0}^{R}H(r)dr)^{2}}$.
Since $\frac{G(r)}{H(r)}\ge\frac{G(R)}{H(R)}$ for all $r\le R$,
we have $H(R)G(r)\ge G(R)H(r)$. Integrating both sides w.r.t $r$
from $0$ to $R$ yields $H(R)\int_{0}^{R}G(r)dr\ge G(R)\int_{0}^{R}H(r)dr$.
Thus, $F^{\prime}(R)\le0$. The function $f(R)=\frac{V_{n}(R)}{V_{n}^{(K)}(R)}$
is non-increasing for $R>0$. 
\end{proof}
In a model space $M^{(K)}$ of dimension $k=\mathrm{dim}(\M)+N$,
the area $A^{(K)}(\rho)$ of a geodesic sphere of radius $\rho$ is
given by 
\[
A^{(K)}(\rho)=S_{k-1}\cdot\mathrm{sn}_{\kappa}(\rho)^{k-1},
\]
where $S_{k-1}=\frac{2\pi^{k/2}}{\Gamma(k/2)}$ is the area of a unit
$(k-1)$-sphere (in Euclidean spaces), and the function $\mathrm{sn}_{\kappa}(\rho)$
is defined as 
\[
\mathrm{sn}_{\kappa}(\rho)=\begin{cases}
\frac{1}{\sqrt{\kappa}}\sin(\sqrt{\kappa}\rho), & K>0,\\
\rho, & K=0,\\
\frac{1}{\sqrt{|\kappa|}}\sinh(\sqrt{|\kappa|}\rho), & K<0,
\end{cases}
\]
with $\kappa=\frac{K}{k-1}$.

To prove the desired result, it suffices to only consider the weaker
case, $K<0$. Recalling the definition of $N_{n}$ in \eqref{eq:-45}
and by Lemma \ref{lem:volumeratio}, 
\[
N_{n}(\sqrt{n}R,\sqrt{n}r)=\sup_{\mathbf{p}}\frac{V_{n}(\mathbf{p},\sqrt{n}R)}{V_{n}(\mathbf{p},\sqrt{n}r)}\le\frac{V_{n}^{(K)}(\sqrt{n}R)}{V_{n}^{(K)}(\sqrt{n}r)}.
\]
By large deviations theory (or Laplace's Method), one can determine
the asymptotics of $V_{n}^{(K)}(\sqrt{n}R)$ as follows 
\[
\lim_{n\rightarrow\infty}\frac{1}{n}\ln V_{n}^{(K)}(\sqrt{n}R)=\sqrt{|K|(k-1)}R.
\]
Hence, substituting $R\leftarrow\lambda r$, we obtain that 
\[
\lim_{r\to\infty}\limsup_{n\to\infty}\frac{1}{nr^{2}}\ln N_{n}(\sqrt{n}\lambda r,\sqrt{n}r)=\lim_{r\to\infty}\frac{1}{r^{2}}\sqrt{|K|(k-1)}(\lambda-1)r=0.
\]

\subsection{Proof of Theorem \ref{thm:LD2}}

We now provide the proof of Theorem \ref{thm:LD} with the outline
given as follows. By Sanov's theorem, we may restrict most coordinates
to compact sets without altering the exponential order of the set's
size, though the remaining coordinates cannot be constrained. Fortunately,
in the case that the condition \eqref{eq:subGaussian} holds, the
measure $\nu^{\otimes n}$ exhibits stronger concentration than Gaussian
measures, which allows us to cover the majority of the mass of the
remaining coordinates at only a negligible enlargement cost (i.e.,
in terms of distance). In other words, the isoperimetric exponent
in this case is determined by the one for compact sets and the latter
was already characterized in Theorem \ref{thm:LD}. On the other hand,
in the case that the condition \eqref{eq:Gaussian-2} holds, under
the covering bound assumption (Condition \ref{cond:covering}), the
majority of the mass of the remaining coordinates can be covered by
subexponentially many small balls with negligible radius. Thus, there
is one small ball that dominates others in the sense that it has almost
the same mass (up to a subexponential factor) as the whole mass of
the remaining coordinates. Covering this small ball will cover the
majority of the mass of the remaining coordinates at only a negligible
enlargement cost (i.e., in terms of distance). In this case, the isoperimetric
exponent is also determined by the one for compact sets characterized
in Theorem \ref{thm:LD}. We next present the detailed proof.

\subsubsection{Step 1: Tailoring the Set}

Let $A_{n}$ be a sequence of optimal subsets asymptotically attaining
the optimal exponent $\limsup_{n\to\infty}E_{n}(\alpha,\tau)$. Specifically,
we let $A_{n}$ be a sequence of subsets such that $\nu^{\otimes n}(A_{n})\ge e^{-n\alpha}$
and $-\frac{1}{n}\ln\nu^{\otimes n}(A_{n}^{\sqrt{n\tau}})\ge E_{n}(\alpha,\tau)-\frac{1}{n}.$
When there is no ambiguity, for brevity we denote $A=A_{n}$.

Since $\mathcal{X}$ is Polish, any probability measure on it is tight.
So, for any $\eta\in(0,1)$, there is a compact set $F=F_{\eta}\subseteq\mathcal{X}$
such that $\nu(F^{c})\le\eta$. Let $\mathbf{X}=(X_{1},...,X_{n})\sim\nu^{\otimes n}$
and $S_{i}:=\bone_{F^{c}}(X_{i})$, $i\in[n]$. Then, $\mathbf{S}\sim\Bern(\nu(F^{c}))^{\otimes n}$.
By Sanov's theorem, for any $\epsilon\in(\eta,1)$, 
\[
\mathbb{P}\{|\mathbf{S}|>n\epsilon\}\le e^{-nD(\epsilon\|\nu(F^{c}))}\le e^{-nD(\epsilon\|\eta)},
\]
where $|\mathbf{S}|=\sum_{i=1}^{n}S_{i}$ denotes the number of ``1''
in $\mathbf{S}$, $D(\epsilon\|\eta)=\epsilon\ln\frac{\epsilon}{\eta}+(1-\epsilon)\ln\frac{1-\epsilon}{1-\eta}$
is the binary relative entropy function, and the second inequality
follows since $\eta\mapsto D(\epsilon\|\eta)$ is decreasing for $\eta<\epsilon$.
Since $\eta\mapsto D(\epsilon\|\eta)$ goes to infinity as $\eta\downarrow0$,
given any $\epsilon>0$, we can choose $\eta=\eta_{\epsilon}$ small
enough so that $D(\epsilon\|\eta)>\alpha$. For example, we can choose
$\eta=\epsilon e^{-\epsilon^{-2}}$ and choose $\epsilon$ small enough.
By basic probabilistic inequalities, we have that 
\begin{align*}
\mathbb{P}\{\mathbf{X}\in A,\,|\mathbf{S}|\le n\epsilon\} & \ge\mathbb{P}\{\mathbf{X}\in A\}-\mathbb{P}\{|\mathbf{S}|>n\epsilon\}\ge e^{-n\alpha}-e^{-nD(\epsilon\|\eta)},
\end{align*}
which implies that 
\begin{equation}
\mathbb{P}\{\mathbf{X}\in A,\,|\mathbf{S}|\le n\epsilon\}\ge e^{-n(\alpha+o(1))}\label{eq:-3-1}
\end{equation}
with $o(1)$ denoting a term vanishing as $n\to\infty$.

Under the condition $|\mathbf{S}|\le n\epsilon$, the random vector
$\mathbf{S}$ takes ${n \choose \le n\epsilon}:=\sum_{i=1}^{\left\lfloor n\epsilon\right\rfloor }{n \choose i}$
distinct vector values. By Sanov's theorem, ${n \choose \le n\epsilon}\le e^{nH(\epsilon)}$,
where $H(\epsilon)=-\epsilon\ln\epsilon-(1-\epsilon)\ln(1-\epsilon)$
is the binary entropy function. Combining this with \eqref{eq:-3-1}
yields that 
\begin{equation}
\max_{|\mathbf{s}|\le n\epsilon}\mathbb{P}\{\mathbf{X}\in A,\mathbf{S}=\mathbf{s}\}\ge e^{-n(\alpha+H(\epsilon)+o(1))}.\label{eq:-5-1}
\end{equation}
Let $\mathbf{s}^{*}$ be the optimal $\mathbf{s}$ attaining the maximum
in the above equation. By rearrangement of the coordinates of $\mathbf{s}^{*}$,
we assume $\mathbf{s}^{*}=(0,...,0,1,...,1)$ such that $|\mathbf{s}^{*}|\le n_{1}:=n\epsilon$.
(For brevity, $n\epsilon$ is assumed to be an integer, otherwise,
denote $n_{1}:=\left\lfloor n\epsilon\right\rfloor $ and our proof
still works.) Let $n_{0}=n-n_{1}=(1-\epsilon)n$. Then, \eqref{eq:-5-1}
implies that $A\cap(F^{n_{0}}\times\mathcal{X}^{n_{1}})$ has probability
at least $e^{-n(\alpha+H(\epsilon)+o(1))}$.

By Chernoff's inequality, for $\mathbf{X}_{1}\sim\nu^{\otimes n_{1}}$,
\[
\mathbb{P}\{d^{2}(\mathbf{X}_{1},\mathbf{v}_{0})>n\beta\}\le e^{-n_{1}\left(\lambda\beta/\epsilon-\Psi(\lambda)\right)}=e^{-n\left(\lambda\beta-\epsilon\Psi(\lambda)\right)},
\]
where $\mathbf{v}_{0}=(x_{0},...,x_{0})\in\mathcal{X}^{n_{1}}$, $\Psi(\lambda):=\ln\mathbb{E}_{\nu}[e^{\lambda d^{2}(X,x_{0})}]$
is the logarithmic moment generating function of $d^{2}(X,x_{0})$,
and $\lambda$ is chosen such that $\Psi(\lambda)<\infty$ (by Condition
\ref{cond:integrability}, such $\lambda$ exists).

In the case that the condition \eqref{eq:Gaussian-2} holds, we choose
$\beta$ large enough so that $\lambda\beta-\epsilon\Psi(\lambda)>\alpha+H(\epsilon)$
(e.g., $\beta=\beta_{\alpha}:=\frac{\alpha+1+\Psi(\lambda)}{\lambda}$).
In the case that the condition \eqref{eq:subGaussian} holds, for
any small $\beta>0$, we can choose $\lambda$ large enough and $\epsilon$
small enough so that $\lambda\beta-\epsilon\Psi(\lambda)>\alpha+H(\epsilon)$.
In both these two cases, $A\cap(F^{n_{0}}\times B_{\sqrt{n\beta}}(\mathbf{v}_{0}))$
has probability at least $e^{-n(\alpha+H(\epsilon)+o(1))}$.

\subsubsection{Step 2: Covering Argument}

We next quantize the ball $B_{\sqrt{n\beta}}(\mathbf{v}_{0})$ by
a covering argument. Specifically, we cover the ball $B_{\sqrt{n\beta}}(\mathbf{v}_{0})$
by using a collection of small balls of radius $\sqrt{n\delta}$ for
a given small $\delta>0$. In the case that the condition \eqref{eq:subGaussian}
holds, $\beta$ can be arbitrarily small. Thus, in this case, we can
set $\beta\le\delta$ and $B_{\sqrt{n\beta}}(\mathbf{v}_{0})$ is
covered by itself.

We next deal with the case that the condition \eqref{eq:Gaussian-2}
holds. In this case, we need the covering bound assumption (Condition
\ref{cond:covering}). By this assumption, for any $\delta>0$, there
is a cover of $B_{\sqrt{n\beta}}(\mathbf{v}_{0})$ consisting of balls
$B_{\sqrt{n\delta}}(\mathbf{v}_{i})$ with $\mathbf{v}_{i}\in\mathcal{X}^{n_{1}}$
and $1\le i\le N$, where $N=e^{n\epsilon'}$ with $\epsilon'=\epsilon'_{n,\epsilon,\beta,\delta}$
satisfying 
\[
\lim_{\epsilon\to0}\limsup_{n\to\infty}\epsilon'_{n,\epsilon,\beta,\delta}=0,\;\forall\beta>\delta>0.
\]
Thus, there is at least one small ball $B:=B_{\sqrt{n\delta}}(\mathbf{v}_{i^{*}})$
for some $i^{*}$ such that 
\[
\tilde{A}:=A\cap(F^{n_{0}}\times B)
\]
has probability at least $e^{-n(\alpha+H(\epsilon)+\epsilon'+o(1))}$.
Note that for fixed $\beta,\delta$, the exponent of this probability
approaches $\alpha$, as $n\to\infty$ first and $\epsilon\to0$ then.

Denote $\tilde{A}_{0}:=\{\mathbf{u}\in F^{n_{0}}:\exists\mathbf{v},(\mathbf{u},\mathbf{v})\in\tilde{A}\}$
as the the projection of $\tilde{A}$ on the first $n_{0}$ components.
Then, $\tilde{A}_{0}\times B$ is a good isoperimetric approximation
of $A$ in the sense that $\tilde{A}_{0}\times B$ has probability
at least $e^{-n(\alpha+H(\epsilon)+\epsilon'+o(1))}$ (since $\tilde{A}\subseteq\tilde{A}_{0}\times B$)
and 
\begin{align}
A^{\sqrt{n\tau}}\supseteq\tilde{A}^{\sqrt{n\tau}} & =\bigcup_{\mathbf{x}\in\tilde{A}}B_{\sqrt{n\tau}}(\mathbf{x})\nonumber \\
 & \supseteq\bigcup_{(\mathbf{u},\mathbf{v})\in\tilde{A}}B_{\sqrt{n\tau_{0}}}(\mathbf{u})\times B_{\sqrt{n\tau_{1}}+2\sqrt{n\delta}}(\mathbf{v})\nonumber \\
 & \supseteq\bigcup_{(\mathbf{u},\mathbf{v})\in\tilde{A}}B_{\sqrt{n\tau_{0}}}(\mathbf{u})\times B_{\sqrt{n\tau_{1}}+\sqrt{n\delta}}(\mathbf{v}_{i^{*}})\nonumber \\
 & =\tilde{A}_{0}^{\sqrt{n\tau_{0}}}\times B^{\sqrt{n\tau_{1}}},\label{eq:-1-1}
\end{align}
where $\tau_{0},\tau_{1}$ are arbitrary nonnegative numbers such
that $\tau_{0}+\left(\sqrt{\tau_{1}}+2\sqrt{\delta}\right)^{2}\le\tau$.
Let $\alpha_{0,n}=-\frac{1}{n}\ln\nu^{\otimes n_{0}}(\tilde{A}_{0})$
and $\alpha_{1,n}=-\frac{1}{n}\ln\nu^{\otimes n_{1}}(B)$. By passing
to subsequence, we can assume that $\alpha_{0,n}\to\alpha_{0}$ and
$\alpha_{1,n}\to\alpha_{1}$ as $n\to\infty$ for some nonnegative
numbers $\alpha_{0},\alpha_{1}$ such that $\alpha_{0}+\alpha_{1}\le\alpha+H(\epsilon)+\epsilon'$.

\subsubsection{Step 3: Information-Theoretic Bound }

We next use information-theoretic techniques to derive bounds. Denote
$\nu_{F}=\nu(\cdot|F)$. Applying Theorem \ref{thm:LD} with substitution
$(n,\alpha,\tau,\nu_{X},\nu_{Y},A)\leftarrow(n_{0},\alpha_{0},\tau_{0},\nu_{F},\nu,\tilde{A}_{0})$
yields that given any $\delta>0$, there is a probability measure
$\pi_{XW}\in\P_{\mathrm{c}}(\mathcal{X}\times\{0,1\})$ induced by
the set $\tilde{A}_{0}$ such that $D(\pi_{X|W}\|\nu_{F}|\pi_{W})\le\frac{\alpha_{0}}{1-\epsilon}+\delta$
and for sufficiently large $n$, 
\begin{equation}
-\frac{1}{n_{0}}\ln\nu^{\otimes n_{0}}(\tilde{A}_{0}^{\sqrt{n\tau_{0}}})\le\inf_{\substack{\pi_{Y|XW}:\mathbb{E}[d^{2}(X,Y)]\le\tau_{0}-\delta,\\
\mathbb{E}[\mathrm{Var}(d^{2}(X,Y)|X,W)]\le\kappa
}
}(1+\delta)D(\pi_{Y|W}\|\nu|\pi_{W})+\delta.\label{eq:-35}
\end{equation}
We choose $\delta$ here as the same $\delta$ in the radius of the
small ball described above.

We next derive a multiletter bound for exponent of probability of
$B^{\sqrt{n\tau_{1}}}$. Let $\mu_{\mathbf{X}}=\nu^{\otimes n_{1}}(\cdot|B)$
be the conditional probability measure of $\nu^{\otimes n_{1}}$ given
$B$. Then, 
\begin{equation}
D(\mu_{\mathbf{X}}\|\nu^{\otimes n_{1}})=-\log\nu^{\otimes n_{1}}(B).\label{eq:-4-1}
\end{equation}

Let $\mu_{\mathbf{Y}|\mathbf{X}}$ be a transition probability defined
from $\mathcal{X}^{n_{1}}$ to $\mathcal{Y}^{n_{1}}$ such that given
each $\mathbf{x}$, $\mu_{\mathbf{Y}|\mathbf{X}=\mathbf{x}}$ is concentrated
on the cost ball $B_{\sqrt{n\tau_{1}}}(\mathbf{x}):=\{\mathbf{y}:d(\mathbf{x},\mathbf{y})\le\sqrt{n\tau_{1}}\}$.
Then, we have that $\mu_{\mathbf{Y}}:=\mu_{\mathbf{X}}\circ\mu_{\mathbf{Y}|\mathbf{X}}$
is concentrated on $B^{\sqrt{n\tau_{1}}}$, which implies that $-\frac{1}{n_{1}}\ln\nu^{\otimes n_{1}}(B^{\sqrt{n\tau_{1}}})\le\frac{1}{n_{1}}D_{0}(\mu_{\mathbf{Y}}\|\nu^{\otimes n_{1}})\leq\frac{1}{n_{1}}D(\mu_{\mathbf{Y}}\|\nu^{\otimes n_{1}})$.
Here $D_{0}(\mu\|\nu):=-\ln\nu\{\frac{\mathrm{d}\mu}{\mathrm{d}\nu}>0\}$
is the Rényi divergence of order $0$, which is no greater than the
relative entropy $D(\mu\|\nu)$ \cite{Erven}. Since $\mu_{\mathbf{Y}|\mathbf{X}}$
is arbitrary, we have that 
\begin{equation}
-\log\nu^{\otimes n_{1}}(B^{\sqrt{n\tau_{1}}})\le\inf_{\mu_{\mathbf{Y}|\mathbf{X}}:d(\mathbf{X},\mathbf{Y})\le t_{1}\;\mu_{\mathbf{XY}}\textrm{-a.e.}}D(\mu_{\mathbf{Y}}\|\nu^{\otimes n_{1}}).\label{eq:-25-1}
\end{equation}

We now construct a feasible solution to the infimization above. Given
$M,\tau_{1},\delta$, and $\mu_{\mathbf{X}}$, let $\mu_{Y|X}$ be
a transition probability such that the induced joint distribution
$\mu_{\mathbf{X}\mathbf{Y}}=\mu_{\mathbf{X}}\mu_{Y|X}^{\otimes n_{1}}$
satisfies 
\begin{equation}
\theta(\mathbf{x}):=\sum_{i=1}^{n_{1}}\mathbb{E}_{\mu}[d^{2}(x_{i},Y_{i})|x_{i}]\le n(\tau_{1}-\delta)\;\mu_{\mathbf{X}}\textrm{-a.e.}\textrm{ and }\;\sum_{i=1}^{n_{1}}\mathrm{Var}_{\mu}(d^{2}(x_{i},Y_{i})|x_{i})\le nM\;\mu_{\mathbf{X}}\textrm{-a.e.}\label{eq:-30-1}
\end{equation}

In the following, we modify $\mu_{\mathbf{Y}|\mathbf{X}}$ by truncating
it into some high-probability set, in order to make it satisfy the
constraint $d(\mathbf{X},\mathbf{Y})\le\sqrt{n\tau_{1}}$ $\mu_{\mathbf{XY}}\textrm{-a.e.}$

Define 
\[
\theta(x)=\mathbb{E}_{\mu}\left[d^{2}(x,Y)|x\right].
\]
Then, $\theta(\mathbf{x})=\sum_{i=1}^{n_{1}}\theta(x_{i}).$

By Chebyshev's inequality, it holds that 
\begin{align*}
\mu\{\mathbf{Y}\notin B_{\sqrt{n\tau_{1}}}(\mathbf{x})|\mathbf{x}\} & =\mu\{d(\mathbf{x},\mathbf{Y})>\sqrt{n\tau_{1}}|\mathbf{x}\}\le\frac{\mathbb{E}_{\mu}\left[(d^{2}(\mathbf{x},\mathbf{Y})-\theta(x_{i}))^{2}|\mathbf{x}\right]}{(n\tau_{1}-\theta(\mathbf{x}))^{2}}\\
 & =\frac{\sum_{i=1}^{n_{1}}\mathbb{E}_{\mu}\left[(d^{2}(x_{i},Y_{i})-\theta(x_{i}))^{2}|x_{i}\right]}{(n\tau_{1}-\theta(\mathbf{x}))^{2}}\\
 & =\frac{\sum_{i=1}^{n_{1}}\mathrm{Var}_{\mu}(d^{2}(x_{i},Y_{i})|x_{i})}{(n\tau_{1}-\theta(\mathbf{x}))^{2}}\le\frac{M}{n\delta^{2}}=:\epsilon_{n}.
\end{align*}
Here given $M,\delta$, the sequence $\epsilon_{n}$ vanishes if we
let $n\to\infty$.

Denote $\hat{\mu}_{\mathbf{Y}|\mathbf{X}}$ as a transition probability
given by 
\begin{align*}
 & \hat{\mu}_{\mathbf{Y}|\mathbf{X}=\mathbf{x}}=\mu_{\mathbf{Y}|\mathbf{X}=\mathbf{x}}(\cdot|B_{\sqrt{n\tau_{1}}}(\mathbf{x})),\;\forall\mathbf{x}.
\end{align*}
That is, $\hat{\mu}_{\mathbf{Y}|\mathbf{X}}(A|\mathbf{x})=\frac{\mu_{\mathbf{Y}|\mathbf{X}}(A\cap B_{\sqrt{n\tau_{1}}}(\mathbf{x})|\mathbf{x})}{\mu_{\mathbf{Y}|\mathbf{X}}(B_{\sqrt{n\tau_{1}}}(\mathbf{x})|\mathbf{x})}$
for all measurable $A$. (Given measurable $A$, the measurability
$\mathbf{x}\mapsto\hat{\mu}_{\mathbf{Y}|\mathbf{X}}(A|\mathbf{x})$
is well known; see e.g., \cite[Theorem 18.3]{billingsley2008probability}.)
We can split $\mu_{\mathbf{Y}|\mathbf{X}}$ as a mixture: 
\begin{align*}
\mu_{\mathbf{Y}|\mathbf{X}} & =(1-\epsilon_{n})\hat{\mu}_{\mathbf{Y}|\mathbf{X}=\mathbf{x}}+\epsilon_{n}\tilde{\mu}_{\mathbf{Y}|\mathbf{X}=\mathbf{x}}
\end{align*}
for some transition probability $\tilde{\mu}_{\mathbf{Y}|\mathbf{X}}$.
For the same input distribution $\mu_{\mathbf{X}}$, the output distributions
of channels $\mu_{\mathbf{Y}|\mathbf{X}}$, $\hat{\mu}_{\mathbf{Y}|\mathbf{X}}$,
and $\tilde{\mu}_{\mathbf{Y}|\mathbf{X}}$ are respectively denoted
as $\mu_{\mathbf{Y}}$, $\hat{\mu}_{\mathbf{Y}}$, and $\tilde{\mu}_{\mathbf{Y}}$,
which satisfy 
\[
\mu_{\mathbf{Y}}=(1-\epsilon_{n})\hat{\mu}_{\mathbf{Y}}+\epsilon_{n}\tilde{\mu}_{\mathbf{Y}}.
\]

Denote $J\sim\mu_{J}:=\mathrm{Bern}(\epsilon_{n})$, and $\mu_{\mathbf{Y}|J=1}=\hat{\mu}_{\mathbf{Y}},\mu_{\mathbf{Y}|J=0}=\tilde{\mu}_{\mathbf{Y}}$.
Then, 
\[
\mu_{\mathbf{Y}}=\mu_{J}(1)\mu_{\mathbf{Y}|J=1}+\mu_{J}(0)\mu_{\mathbf{Y}|J=0}.
\]

Observe that 
\begin{align*}
D(\mu_{\mathbf{Y}|J}\|\nu^{\otimes n_{1}}|\mu_{J}) & =(1-\epsilon_{n})D(\hat{\mu}_{\mathbf{Y}}\|\nu^{\otimes n_{1}})+\epsilon_{n}D(\tilde{\mu}_{\mathbf{Y}}\|\nu^{\otimes n_{1}})\\
 & \ge(1-\epsilon_{n})D(\hat{\mu}_{\mathbf{Y}}\|\nu^{\otimes n_{1}}).
\end{align*}
On the other hand, 
\begin{align*}
D(\mu_{\mathbf{Y}|J}\|\nu^{\otimes n_{1}}|\mu_{J}) & =D(\mu_{\mathbf{Y}}\|\nu^{\otimes n_{1}})+D(\mu_{\mathbf{Y}|J}\|\mu_{\mathbf{Y}}|\mu_{J})\\
 & \le D(\mu_{\mathbf{Y}}\|\nu^{\otimes n_{1}})+\ln2,
\end{align*}
where the inequality follows since $D(\mu_{\mathbf{Y}|J}\|\mu_{\mathbf{Y}}|\mu_{J})=I_{\mu}(J;\mathbf{Y})\le H_{\mu}(J)\le\ln2.$
Hence, 
\begin{equation}
D(\hat{\mu}_{\mathbf{Y}}\|\nu^{\otimes n_{1}})\le\frac{D(\mu_{\mathbf{Y}}\|\nu^{\otimes n_{1}})+\ln2}{1-\epsilon_{n}}.\label{eq:-65}
\end{equation}

By choosing $Q_{\mathbf{Y}|\mathbf{X}}$ in \eqref{eq:-25-1} as the
feasible solution $\hat{Q}_{\mathbf{Y}|\mathbf{X}}$, we then have
that 
\begin{equation}
-\log\nu^{\otimes n_{1}}(B^{\sqrt{n\tau_{1}}})\le\inf\frac{D(\mu_{\mathbf{Y}}\|\nu^{\otimes n_{1}})+\log2}{1-\epsilon_{n}},\label{eq:-7-1}
\end{equation}
where the infimum is taken over all transition probabilities $\mu_{Y|X}$
satisfying \eqref{eq:-30-1}.

Combining \eqref{eq:-35} and \eqref{eq:-7-1} yields that given $\delta>0$,
for sufficiently large $n$ and sufficiently small $\epsilon$, it
holds that 
\begin{align}
 & -\log\nu^{\otimes n}(\tilde{A}_{0}^{\sqrt{n\tau_{0}}}\times B^{\sqrt{n\tau_{1}}})\nonumber \\
 & \le\inf(1+\delta)\left(n_{0}D(\pi_{Y|W}\|\nu|\pi_{W})+D(\mu_{\mathbf{Y}}\|\nu^{\otimes n_{1}})\right)+n\delta\label{eq:-27-1}
\end{align}
for some $(\pi_{XW},\mu_{\mathbf{X}})$ such that 
\[
n_{0}D(\pi_{X|W}\|\nu_{F}|\pi_{W})+D(\mu_{\mathbf{X}}\|\nu^{\otimes n_{1}})\le n\alpha',\quad\alpha':=\alpha+2\delta,
\]
where the infimum is taken over all distributions $(\pi_{Y|XW},\mu_{Y|X})$
satisfying 
\begin{align*}
\mathbb{E}_{\pi}[d^{2}(X,Y)] & \le\tau_{0}-\delta,\\
\mathbb{E}_{\pi}[\mathrm{Var}_{\pi}(d^{2}(X,Y)|X,W)] & \le\kappa,
\end{align*}
and \eqref{eq:-30-1}. Since $\tau_{0},\tau_{1}$ are arbitrary nonnegative
numbers such that $\tau_{0}+\left(\sqrt{\tau_{1}}+2\sqrt{\delta}\right)^{2}\le\tau$,
optimizing the bound over all such $\tau_{0},\tau_{1}$ yields that
the infimum can be taken over all distributions $(\pi_{Y|XW},\mu_{Y|X})$
satisfying 
\begin{align}
n_{0}\mathbb{E}_{\pi}[d^{2}(X,Y)]+\sum_{i=1}^{n_{1}}\mathbb{E}_{\mu}[d^{2}(x_{i},Y)|x_{i}] & \le n\tau'\;\mu_{\mathbf{X}}\textrm{-a.e.}\label{eq:}\\
\mathbb{E}_{\pi}[\mathrm{Var}_{\pi}(d^{2}(X,Y)|X,W)] & \le\kappa,\label{eq:-1}\\
\sum_{i=1}^{n_{1}}\mathrm{Var}_{\mu}(d^{2}(x_{i},Y)|x_{i}) & \le nM\;\mu_{\mathbf{X}}\textrm{-a.e.}
\end{align}
where $\tau'=\tau-5\delta-4\sqrt{\delta\tau}$.

\subsubsection{Step 4: Heat Flow}

We next construct a feasible solution to \eqref{eq:-27-1} by using
heat flow. Let $\mu_{X_{t}|X=x}$ be the heat flow starting from the
point $x$, which forms a transition probability. Let 
\[
\pi_{WXX_{t}}=\pi_{WX}\mu_{X_{t}|X}
\]
and let $\pi_{X_{t}|W}$ the conditional probability induced by this
joint probability. Let 
\begin{equation}
\pi_{WXX_{t}}^{*}=\pi_{WX}\pi_{X_{t}|XW}^{*}\label{eq:pi-star}
\end{equation}
where $\pi_{X_{t}|XW}^{*}$ is the optimal transport map attaining
$\W(\pi_{X|W},\pi_{X_{t}|W}|\pi_{W})$. Let $\mu_{\mathbf{X}\mathbf{X}_{t}}=\mu_{\mathbf{X}}\mu_{X_{t}|X}^{\otimes n_{1}}$.

We choose $(\pi_{Y|XW},\mu_{Y|X})=(\pi_{X_{t}|XW}^{*},\mu_{X_{t}|X})$
and choose the time $t$ properly to satisfy \eqref{eq:}. Let $\omega$
be a function given by 
\[
\W^{2}(t):=n_{0}\W^{2}(\pi_{X|W},\pi_{X_{t}|W}|\pi_{W})+\W^{2}(\mu_{\mathbf{X}},\mu_{\mathbf{X}_{t}}),
\]
which is continuous in $t\ge0$. By Remark \ref{rem:Under}, $\nu$
satisfies a log-Sobolev inequality with a constant $K>0$. Let $t_{0}=\frac{1}{2K}\ln\frac{\alpha'}{\delta}$.
Let $\tau''$ be a fixed number $\le\tau$ which will be specified
later. We choose the time $t=t_{0}$ if $\W(t_{0})\le\sqrt{n\tau''}$,
and choose $t$ as the smallest time $t_{1}$ such that $\W(t_{1})=\sqrt{n\tau''}$
if $\W(t_{0})>\sqrt{n\tau''}$. Hence, $\W(t)\le\sqrt{n\tau''}$.

We next verify the two requirements---the expectation requirement
and variance requirement in \eqref{eq:} and \eqref{eq:-1} for $(\mu_{X_{t}|X}\circ\mu_{X|W},\mu_{X_{t}|X})$,
where $\mu_{X_{t}|X}\circ\mu_{X|W}$ denotes the resulting transition
probability by cascading $\mu_{X|W}$ and $\mu_{X_{t}|X}$. We first
focus on the expectation requirement. By the triangle inequality,
we obtain for any $\mathbf{x},\mathbf{x}'\in B$, 
\begin{align}
\W\left(\mu_{\mathbf{X}_{t}|\mathbf{X}=\mathbf{x}},\delta_{\mathbf{x}}\right) & \le\W\left(\mu_{\mathbf{X}_{t}|\mathbf{X}=\mathbf{x}},\mu_{\mathbf{X}_{t}|\mathbf{X}=\mathbf{x}'}\right)+\W\left(\mu_{\mathbf{X}_{t}|\mathbf{X}=\mathbf{x}'},\delta_{\mathbf{x}}\right)\nonumber \\
 & \le\sqrt{n\delta}+\W\left(\mu_{\mathbf{X}_{t}|\mathbf{X}=\mathbf{x}'},\delta_{\mathbf{x}}\right),\label{eq:-2-1}
\end{align}
where $\W\left(\mu_{\mathbf{X}_{t}|\mathbf{X}=\mathbf{x}},\mu_{\mathbf{X}_{t}|\mathbf{X}=\mathbf{x}'}\right)\le\W\left(\delta_{\mathbf{x}},\delta_{\mathbf{x}'}\right)\le\sqrt{n\delta}$
follows by contraction of Wasserstein metric under $\mathrm{RCD}(0,\infty)$
given in \cite[(6.1)]{AmbrosioGigliSavare14} or \cite[Corollary 3.18]{Ambrosio2015BakryEmery}.
Taking expectation $\mathbb{E}_{\mathbf{X}'\sim\mu_{\mathbf{X}}}$
and by Jensen's inequality, 
\begin{align*}
\W\left(\mu_{\mathbf{X}_{t}|\mathbf{X}=\mathbf{x}},\delta_{\mathbf{x}}\right) & \le\mathbb{E}_{\mathbf{X}'\sim\mu}[\W\left(\mu_{\mathbf{X}_{t}|\mathbf{X}=\mathbf{X}'},\delta_{\mathbf{x}}\right)]+\sqrt{n\delta}\\
 & \le\sqrt{\mathbb{E}_{\mathbf{X}'\sim\mu}[\W^{2}\left(\mu_{\mathbf{X}_{t}|\mathbf{X}=\mathbf{X}'},\delta_{\mathbf{x}}\right)]}+\sqrt{n\delta}\\
 & =\sqrt{\mathbb{E}_{\mathbf{X}\sim\mu}\sum_{i=1}^{n_{1}}\mathbb{E}_{\mu}[d^{2}(x_{i},X_{i,t})|X_{i}]}+\sqrt{n\delta}\\
 & =\sqrt{\sum_{i=1}^{n_{1}}\mathbb{E}_{\mu}[d^{2}(x_{i},X_{i,t})]}+\sqrt{n\delta}\\
 & =\W\left(\delta_{\mathbf{x}},\mu_{\mathbf{X}_{t}}\right)+\sqrt{n\delta}.
\end{align*}
Therefore, 
\begin{align*}
\sqrt{\sum_{i=1}^{n_{1}}\mathbb{E}_{\mu}[d^{2}(x_{i},X_{i,t})|x_{i}]} & =\W\left(\mu_{\mathbf{X}_{t}|\mathbf{X}=\mathbf{x}},\delta_{\mathbf{x}}\right)\le\W\left(\delta_{\mathbf{x}},\mu_{\mathbf{X}_{t}}\right)+\sqrt{n\delta}\\
 & \le\W\left(\mu_{\mathbf{X}},\mu_{\mathbf{X}_{t}}\right)+2\sqrt{n\delta},
\end{align*}
where the last inequality follows by the triangle inequality again.
Moreover, $\mathbb{E}_{\pi^{*}}[d^{2}(X,Y)]=\W^{2}(\pi_{X|W},\pi_{X_{t}|W}|\pi_{W})$.
Therefore, combining these with $\W(t)\le\sqrt{n\tau''}$ yields 
\[
n_{0}\mathbb{E}_{\pi}[d^{2}(X,Y)]+\sum_{i=1}^{n_{1}}\mathbb{E}_{\mu}[d^{2}(x_{i},Y)|x_{i}]\le n(\tau''+4\delta+4\sqrt{\delta\tau})\;\mu_{\mathbf{X}}\textrm{-a.e.}
\]
Choosing $\tau''=\tau-9\delta-8\sqrt{\delta\tau}$ yields the expectation
requirement in \eqref{eq:}.

We next focus on the variance requirement. Under $\mathrm{CD}(0,\infty)$,
the gradient flow (heat flow) $\rho_{t}$ (given in \eqref{eq:gradCh})
of Cheeger's energy on the Hilbert space $L^{2}(\mathcal{X},\nu)$
is identical to the gradient flow $\mu_{t}$ of relative entropy on
the Wasserstein metric space in the sense that $\mu_{t}=\rho_{t}\nu$
when $\mu_{0}=\rho_{0}\nu$ \cite[Theorem 9.3]{ambrosio2014calculus}.
In this case, denote $P_{t}$ as the Markov (heat) semigroup operator
given by $P_{t}:\rho_{0}\mapsto\rho_{t}$. If in addition, the space
satisfies $\mathrm{RCD}(0,\infty)$, then these two flows satisfy
the equation $\frac{\mathrm{d}}{\mathrm{d}t}\rho_{t}=\Delta_{\mathcal{E}}\rho_{t}$
for all $t>0$ where $\Delta_{\mathcal{E}}$ is the infinitesimal
generator of $P_{t}$ and $\rho_{t}=P_{t}\rho_{0}$ \cite[Theorem 2.13 and Section 6.1]{AmbrosioGigliSavare14},
and moreover, in this case, Bakry-Émery curvature-dimension condition
$\mathrm{BE}(0,\infty)$ is satisfied by the semigroup $P_{t}$ \cite[Theorem 4.17]{Ambrosio2015BakryEmery}.

By the local Poincaré inequalities in \cite[Corollary 2.3, Definition 2.4, Theorem 4.17]{Ambrosio2015BakryEmery}\cite[Theorem 4.7.2]{bakry2013analysis},
under $\mathrm{RCD}(0,\infty)$ (equivalently, $\mathrm{BE}(0,\infty)$),
it holds that for every $f\in W^{1,2}(\mathcal{X},d,\nu)$ with the
Sobolev space $W^{1,2}(\mathcal{X},d,\nu)$ defined around \eqref{eq:-32}
and for $t>0$, 
\[
P_{t}(f^{2})-(P_{t}f)^{2}\le2tP_{t}(|\nabla f|_{*}^{2}),\;\nu\textrm{-a.e. in }\mathcal{X},
\]
where $|\nabla f|_{*}$ is the minimal weak gradient defined below
\eqref{eq:-32} and $f\mapsto|\nabla f|_{*}^{2}$ forms the Carré
du champ $\Gamma$ induced by the Dirichlet form $\mathcal{E}=2\mathrm{Ch}$.
That is, for every $f\in W^{1,2}(\mathcal{X},d,\nu)$, 
\begin{equation}
\mathrm{Var}_{\mu}(f(X_{t})|X=x)\le2t\mathbb{E}_{\mu}[|\nabla f|_{*}^{2}(X_{t})|X=x],\;\nu\textrm{-a.e. }x\in\mathcal{X}.\label{eq:-34}
\end{equation}

Let $f(y)=d^{2}(x,y)$ and for $N>0$, let $f_{N}(y)=\min\{f(y),N^{2}\}=\left(\min\{d(x,y),N\}\right)^{2}$.
Then, $|\nabla f_{N}|\le2\min\{d(x,y),N\}$ and hence $f_{N}\in W^{1,2}(\mathcal{X},d,\nu)$.
Substituting this function into \eqref{eq:-34} and noting that $|\nabla f_{N}|_{*}\le|\nabla f_{N}|\le2d(x,y)$
$\nu$-a.e. $y$, we obtain that 
\[
\mathrm{Var}_{\mu}(f_{N}(X_{t})|X=x)\le8t\mathbb{E}_{\mu}[d^{2}(x,X_{t})|X=x],\;\nu\textrm{-a.e. }x\in\mathcal{X}.
\]
Letting $N\to\infty$ and by the monotone convergence theorem, for
$\nu\textrm{-a.e. }x\in\mathcal{X}$, 
\begin{equation}
\mathrm{Var}_{\mu}(d^{2}(x,X_{t})|X=x)\le8t\mathbb{E}_{\mu}[d^{2}(x,X_{t})|X=x].\label{eq:-6}
\end{equation}
Using the triangle inequality, for any $x_{0},x\in F$, 
\begin{align}
\sqrt{\mathbb{E}_{\mu}[d^{2}(x_{0},X_{t})|X=x]} & =\W(P_{t}\delta_{x},\delta_{x_{0}})\nonumber \\
 & \le\W(P_{t}\delta_{x},\nu)+\W(\nu,\delta_{x_{0}})\nonumber \\
 & \le\W(\delta_{x},\nu)+\W(\nu,\delta_{x_{0}})\nonumber \\
 & \le\W(\delta_{x},\delta_{x_{0}})+2\W(\delta_{x_{0}},\nu)\nonumber \\
 & \le\mathrm{diam}(F)+2\W(\delta_{x_{0}},\nu)=:C,\label{eq:-28}
\end{align}
where the second inequality follows by $\W(P_{t}\delta_{x},\nu)\le\W(\delta_{x},\nu)$,
contraction of Wasserstein metric in \cite[(6.1)]{AmbrosioGigliSavare14}.
Noting that $C<\infty$, $\mathbb{E}_{\mu}[d^{2}(x_{0},X_{t})|X=x]$
is upper bounded by the finite number $C$ uniformly for all $x\in F$.
We then obtain from \eqref{eq:-6} that 
\begin{equation}
\mathbb{E}_{\mu}[d^{4}(x_{0},X_{t})|X=x]\le C^{4}+8tC^{2}.\label{eq:-6-1-1}
\end{equation}

Moreover, by triangle inequality and Minkowski inequality, for any
$x_{0}\in F$, 
\begin{align*}
\mathbb{E}_{\pi^{*}}[d^{4}(X,X_{t})]^{1/4} & \le\mathbb{E}_{\pi^{*}}[d^{4}(x_{0},X_{t})]^{1/4}+\mathbb{E}_{\pi^{*}}[d^{4}(X,x_{0})]^{1/4}\\
 & \le\mathbb{E}_{\pi}[d^{4}(x_{0},X_{t})]^{1/4}+\mathrm{diam}(F),
\end{align*}
where the last line follows since $\pi_{X_{t}}^{*}=\pi_{X_{t}}$,
according to \eqref{eq:pi-star}, and $d(x,x_{0})\le\mathrm{diam}(F)$
for all $x\in F$. By \eqref{eq:-6-1-1}, $\mathbb{E}_{\pi}[d^{4}(x_{0},X_{t})]\le C^{4}+8tC^{2}.$
Thus, 
\[
\mathbb{E}_{\pi^{*}}[d^{4}(X,X_{t})]^{1/4}\le(C^{4}+8tC^{2})^{1/4}+\mathrm{diam}(F).
\]
Note that we have truncated the time $t$ up to the finite value $t_{0}$,
implying $t\le t_{0}$. Hence, the first variance requirement is satisfied.
This gives the reason why we truncate $t$.

We next consider the second variance requirement. Substituting $x=x_{i}$
into \eqref{eq:-6} and taking summation over $i$, we obtain that
for $\mu_{\mathbf{X}}$-a.e. $\mathbf{x}$, 
\[
\sum_{i=1}^{n_{1}}\mathrm{Var}_{\mu}(d^{2}(x_{i},X_{i,t})|X_{i}=x_{i})\le8t\sum_{i=1}^{n_{1}}\mathbb{E}_{\mu}[d^{2}(x_{i},X_{i,t})|X_{i}=x_{i}]\le8tn(\tau-\delta),
\]
where the last inequality above follows by the expectation condition
in \eqref{eq:}. Since $t\le t_{0}$, the second variance requirement
is satisfied.

Since all requirements are satisfied by the transition probability
induced by heat flow, we obtain a feasible solution to \eqref{eq:-27-1}.

\subsubsection{Step 5: Final Bound}

We next analyze the relative entropy induced by the heat flow. Let
\[
D(t):=n_{0}D(\pi_{X_{t}|W}\|\nu|\pi_{W})+D(\mu_{\mathbf{X}_{t}}\|\nu^{\otimes n_{1}}).
\]

We first show that $D(t)\le n\delta$ if $t=t_{0}$. Since $\nu$
admits a log-Sobolev inequality with constant $K$, by exponential
decay in entropy \cite[Theorem 5.2.1]{bakry2013analysis}, we have
that in this case, 
\begin{equation}
D(t_{0})\le e^{-2Kt_{0}}D(0)\le e^{-2Kt_{0}}n\alpha=n\delta,\label{eq:-36}
\end{equation}
where the last equality is due to the choice of $t_{0}$.

We next upper bound the relative entropy for $t=t_{1}$. By the chain
rule for relative entropy and Fisher information and the fact that
conditioning increases relative entropy, the following kind of tensorization
property can be proven. 
\begin{thm}
\cite{polyanskiy2019improved} Assume that $\nu$ admits a $\theta$-nonlinear
log-Sobolev inequality for some convex function $\theta$. Then, $\nu^{\otimes n}$
admits a $\theta_{n}$-nonlinear log-Sobolev inequality where $\theta_{n}(t)=n\theta(t/n),\forall t\ge0$. 
\end{thm}

By this theorem and Theorem \ref{thm:Sobolev-Talagrand}, we obtain
that 
\begin{align}
\sqrt{n\tau''}=\W(t) & \le\int_{D(t)}^{D(0)}\frac{1}{\sqrt{n\theta(r/n)}}\d r=\sqrt{n}\int_{D(t)/n}^{D(0)/n}\frac{1}{\sqrt{\theta(t)}}\d t=\sqrt{n}\Upsilon\left(D(t)/n,\alpha\right).\label{eq:-37}
\end{align}

Substituting \eqref{eq:-36} and \eqref{eq:-37} into \eqref{eq:-27-1}
and letting $\delta\to0$ yield that $\limsup_{n\to\infty}E_{n}(\alpha,\tau)\le\bar{\psi}(\alpha,\tau).$

\subsection{Interpretation from Concentration of Measure}

The isoperimetric problem with thick boundary admits a natural interpretation
from the perspective of concentration of measure. Concentration of
measure principle states that for any set of not too small probability,
slightly enlarging this set will always have probability close to
one. Formally, let $r(a)$ be the minimum number such that the $r$-enlargement
of any set of measure $a$ will always have probability close to one.
Denote $a=e^{-n\alpha}$ and $r^{2}(a)=n\tau(\alpha)$. By our results,
for sufficiently large $n$, $\tau(\alpha)$ is the minimum $\tau$
such that $\psi(\alpha,\tau)=0$. Solving it yields $\tau\ge\W(\pi_{X|W},\nu|\pi_{W})$
for all $\pi_{XW}$ such that $D(\pi_{X|W}\|\nu|\pi_{W})\le\alpha$.
Hence, $\tau(\alpha)$ satisfies 
\begin{align*}
\tau(\alpha) & =\sup_{\pi_{XW}:D(\pi_{X|W}\|\nu|\pi_{W})\le\alpha}\W(\pi_{X|W},\nu|\pi_{W}),
\end{align*}
which is the upper concave envelope of $\tau_{0}$ given below 
\begin{align*}
\tau_{0}(\alpha): & =\sup_{\pi:D(\pi\|\nu)\le\alpha}\W(\pi,\nu).
\end{align*}
The curves $\tau$ and $\tau_{0}$ respectively characterize the best
convex and nonconvex tradeoffs between the relative entropy and the
Wasserstein metric. Thus, they determine the optimal nonlinear transport
inequalities.

Recall that $K_{C}$ is the optimal transport constant in the following
inequality: $\W(\pi,\nu)\le\sqrt{\frac{2}{K_{C}}D(\pi\|\nu)},\;\forall\pi.$
Hence, $\tau(\alpha)\le\frac{2\alpha}{K_{C}},\;\forall\alpha>0,$
and the constant $K_{C}$ is optimal. This implies that for sufficiently
large $n$, enlarging any set of measure $\exp(-\frac{K_{C}r^{2}}{2})$
by the distance $r$ will always make it have probability close to
one, and moreover, the constant $K_{C}$ is the optimal one satisfying
this property. In contrast, another (but indeed equivalent) interpretation
of $K_{C}$ in the concentration of measure is that for all $A$ such
that $\nu_{n}(A)=1/2$, it always holds that $\nu_{n}((A^{r})^{c})\le\exp(-\frac{K_{C}r^{2}}{2}),\;\forall r>0$.
Here, the constant $K_{C}$ cannot be improved either.

\section{From Variant Isoperimetry to Standard Isoperimetry}\label{sec:From-Variant-Isoperimetry}

We now turn back to the standard isoperimetric problem in the Riemannian
manifold, in which the boundary is the infinitely thin shell. We now
connect the isoperimetric problem with thick shell boundary considered
in the last section to this setting.

Recall that $\M$ is a $k$-dimensional Riemannian manifold equipped
with a reference probability measure $\nu=e^{-V}\mathrm{vol},V\in C^{2}(\M)$
and the induced geodesic distance $d$. Recall the boundary measure
defined in \eqref{eq:perimeter}. The boundary measure can be alternatively
expressed as 
\begin{align}
\nu_{n}^{+}(A) & =\liminf_{r\downarrow0}\frac{\nu_{n}(A^{r})-\nu_{n}(A)}{\ln\nu_{n}(A^{r})-\ln\nu_{n}(A)}\cdot\frac{\ln\nu_{n}(A^{r})-\ln\nu_{n}(A)}{r}\nonumber \\
 & =\nu_{n}(A)\liminf_{r\downarrow0}\frac{\ln[\nu_{n}(A^{r})/\nu_{n}(A)]}{r}.
\end{align}
However, to apply Theorem \ref{thm:LD2}, we need keep $r$ to be
large (in order of $\sqrt{n}$), and hence, we need remove $\liminf_{r\downarrow0}$
in the last line. To this end, we need some nice properties of isoperimetric
minimizers from geometric measure theory.

Here we adopt the notation used in \cite{milman2010isoperimetric}.
For a smooth hypersurface $S$, the second fundamental form is by
$\varPi_{S,p}^{\sigma}(u,v)=g(D_{u}v,v)$ for $u,v\in T_{p}S$, where
$D$ is the covariant derivative. According to this definition, the
second fundamental form of the sphere in Euclidean space with respect
to the outer normal $\sigma$ is positive definite. The total curvature
$H_{S}^{\sigma}(p)$ of $S$ at a point $p$ is the trace of the second
fundamental form $\varPi_{S,p}^{\sigma}$. (The mean curvature is
obtained if it is divided by the dimension of $S$.) 
\begin{defn}
The $\nu$-total curvature of a smooth hypersurface $S$ at a point
$p\in S$ with respect to the unit normal vector $\sigma$, denoted
$H_{S,\nu}^{\sigma}(p)$, is defined as 
\[
H_{S,\nu}^{\sigma}(p):=H_{S}^{\sigma}(p)-\langle\nabla V(p),\sigma(p)\rangle.
\]
\end{defn}

Bayle \cite{bayle2003proprietes} and Morgan \cite{morgan2005manifolds,morgan2016geometric}
showed that isoperimetric minimizers always exist and the isoperimetric
profile is concave. See an explicit explanation in \cite[p. 216]{milman2010isoperimetric}. 
\begin{thm}[Morgan \cite{morgan2003regularity,morgan2005manifolds,morgan2016geometric}
and Bayle \cite{bayle2003proprietes}]
\label{thm:DensityConcave} Under Condition \ref{cond:integrability},
the following hold. 
\begin{enumerate}
\item For any $a\in(0,1)$, there always exists an open isoperimetric minimizer
$A$ in $(\M^{n},\nu_{n})$ of given measure $a$ such that: (a) its
boundary $\partial A$ consists of disjoint regular part and singular
part; (b) the regular part is a $C^{2}$-smooth hypersurface (since
$V\in C^{2}(\M)$) which has a constant $\nu_{n}$-total curvature,
denoted by $H_{\nu_{n}}(A)$; (c) the singular part has dimension
strictly lower than that of the regular part, more precisely, lower
than or equal to the dimension of the regular part minus $7$. 
\item The isoperimetric profile $I_{n}$ is a concave function on $[0,1]$. 
\item For the isoperimetric minimizer given in Statement 1, the constant
$\nu_{n}$-total curvature $H_{\nu_{n}}(A)$ of the regular part of
its boundary satisfies 
\[
\lim_{\epsilon\downarrow0}\frac{I_{n}(a+\epsilon)-I_{n}(a)}{\epsilon}\le H_{\nu_{n}}(A)\le\lim_{\epsilon\downarrow0}\frac{I_{n}(a)-I_{n}(a-\epsilon)}{\epsilon}.
\]
In particular, if $I_{n}$ is differentiable at $a$, then $H_{\nu_{n}}(A)=I_{n}'(a)$. 
\end{enumerate}
\end{thm}

\begin{thm}[{{{{Generalized Heintze-Karcher due to Morgan \cite[Theorem 2, Remark 3]{morgan2005manifolds}}}}}]
\label{thm:DensityConcave-2} Assume that for $\kappa\in\mathbb{R}$,
\[
\mathrm{Ric}_{g}+\mathrm{Hess}_{g}V\ge\kappa g.
\]
Let $A\subseteq\M^{n}$ denote an isoperimetric minimizer in $(\M^{n},\nu_{n})$
of given measure $a\in(0,1)$ given in Theorem \ref{thm:DensityConcave}.
Let $S$ denote the regular part of $\partial A$, and let $H_{\nu_{n}}(A)$
denote the constant $\nu_{n}$-total curvature of $S$ with respect
to the outer unit normal vector field $\sigma$ on $S$. Then for
any $r>0$, 
\[
\nu_{n}(A^{r})-\nu_{n}(A)\le\nu_{n}^{+}(A)\int_{0}^{r}\exp\left(H_{\nu_{n}}(A)t-\frac{\kappa}{2}t^{2}\right)\d t.
\]
\end{thm}

Note that the sign of the total curvature $H_{S}^{\sigma}(p)$ here
is opposite to the one in \cite[Theorem 2, Remark 3]{morgan2005manifolds}.
Based on the two theorems above, we build a bridge between the standard
isoperimetric problem and the isoperimetric problem with thick boundary. 
\begin{thm}
\label{thm:isop-isop} Assume that Condition \ref{cond:convexity}
holds. Let $A\subseteq\M^{n}$ denote an isoperimetric minimizer in
$(\M^{n},\nu_{n})$ of given measure $a\in(0,1)$ given in Theorem
\ref{thm:DensityConcave}. Then, for any $r>0$, 
\[
\frac{I_{n}(a)}{a}=\frac{\nu_{n}^{+}(A)}{\nu_{n}(A)}\ge\frac{1}{r}\ln\frac{\nu_{n}(A^{r})}{\nu_{n}(A)}\ge\frac{1}{r}\ln\frac{\Gamma_{n}(a,r)}{a}.
\]
\end{thm}

\begin{proof}
Let $S$ denote the regular part of $\partial A$, and let $H_{\nu_{n}}(A)$
denote the constant $\nu_{n}$-total curvature of $S$ with respect
to the outer unit normal vector field $\sigma$ on $S$. Let 
\begin{align*}
P(r) & :=\nu_{n}^{+}(A)\exp\left(H_{\nu_{n}}(A)r\right),\\
V(r) & :=\nu_{n}(A)+\int_{0}^{r}P(t)\d t.
\end{align*}

Define $g(r):=V(r)P(0)-P(r)V(0),$ whose derivative satisfies 
\begin{equation}
g'(r)=P(r)P(0)-P'(r)V(0)=P(r)P(0)-H_{\nu_{n}}(A)P(r)V(0).\label{eq:-61}
\end{equation}
From Theorem \ref{thm:DensityConcave}, it is known that 
\[
H_{\nu_{n}}(A)\le\lim_{\epsilon\downarrow0}\frac{I_{n}(a)-I_{n}(a-\epsilon)}{\epsilon},
\]
where $a=\nu_{n}(A)$. Since $I_{n}$ is concave, 
\[
\lim_{\epsilon\downarrow0}\frac{I_{n}(a)-I_{n}(a-\epsilon)}{\epsilon}\le\frac{I_{n}(a)}{a}=\frac{P(0)}{V(0)}.
\]
Hence, $H_{\nu_{n}}(A)\le\frac{P(0)}{V(0)},$ which, combined with
\eqref{eq:-61}, implies $g'(r)\ge0.$

Observe that $g(0)=0$. So, for all $r\ge0$, it holds that $g(r)\ge0,$
i.e., $\frac{P(r)}{V(r)}\le\frac{P(0)}{V(0)}.$ Noting that $P(r)=V'(r)$,
we then have that $\ln\frac{V(r)}{V(0)}\le\frac{P(0)}{V(0)}r.$ From
Theorem \ref{thm:DensityConcave-2}, it is known that $V(r)\ge\nu_{n}(A^{r})$.
We hence conclude that 
\[
\frac{\nu^{+}(A)}{\nu(A)}=\frac{P(0)}{V(0)}\ge\frac{1}{r}\ln\frac{V(r)}{V(0)}\ge\frac{1}{r}\ln\frac{\nu_{n}(A^{r})}{\nu_{n}(A)}.
\]
\end{proof}

\section{Putting Everything Together: From Log-Sobolev to Standard Isoperimetry }\label{sec:Combining-All-Above}

In this section, we combine all the tools built in Sections \ref{sec:From-Log-Sobolev-to}--\ref{sec:From-Variant-Isoperimetry}
to prove the lower bound in our large deviations theorem, i.e., to
deduce standard isoperimery from nonlinear log-Sobolev inequalities.
As a consequence of Theorem \ref{thm:LD2} and Theorem \ref{thm:isop-isop},
we can lower bound the asymptotics of the isoperimetric profile $I_{n}$
in terms of $\breve{\Theta}$. Recall that $\breve{\Theta}$ defined
in \eqref{eq:Theta_LCE} is the function characterizing the optimal
tradeoff in the nonlinear log-Sobolev inequality. We obtain the lower
bound part of Theorem \ref{thm:equivalence}. Recall $\underline{\Lambda}$
defined in \eqref{eq:-82}. 
\begin{thm}[From Nonlinear Log-Sobolev to Standard Isoperimetry]
\label{thm:asymisoper} Under Conditions \ref{cond:convexity} and
\ref{cond:integrability}, for any $\alpha>0$, 
\begin{align}
\underline{\Lambda}(\alpha) & \ge\sqrt{\breve{\Theta}(\alpha)}.\label{eq:LB}
\end{align}
\end{thm}

\begin{proof}
Let $A_{n}$ be an isoperimetric minimizer of size $e^{-n\alpha}$
in the $n$-dimensional space. From Theorem \ref{thm:isop-isop} with
$r=\sqrt{n\tau}$, we see that 
\[
\frac{I_{n}(e^{-n\alpha})}{e^{-n\alpha}\sqrt{n}}\ge\frac{\alpha-E_{n}(\alpha,\tau)}{\sqrt{\tau}}.
\]
Letting $n\to\infty$, from Theorem \ref{thm:LD2} we obtain for any
$\tau>0$, $\underline{\Lambda}(\alpha)\ge\frac{\alpha-\bar{\psi}(\alpha,\tau)}{\sqrt{\tau}}.$
Letting $\tau\to0$, we obtain $\underline{\Lambda}(\alpha)\ge\limsup_{\tau\downarrow0}\frac{\alpha-\Upsilon_{\alpha}^{-1}(\sqrt{\tau})}{\sqrt{\tau}}$.

Denote $t=\sqrt{\tau}$ and $s(\alpha,t)=\Upsilon_{\alpha}^{-1}(t)$.
Then, 
\[
t=\int_{s(\alpha,t)}^{\alpha}\frac{1}{\sqrt{\breve{\Theta}(r)}}\d r.
\]
Note that $s(\alpha,0)=\alpha$. Taking derivative w.r.t. $t$ at
$t=0$ yields $-\partial_{t}s(\alpha,t)\big|_{t=0}=\sqrt{\breve{\Theta}(\alpha)}$.
Therefore, $\underline{\Lambda}(\alpha)\ge\sqrt{\breve{\Theta}(\alpha)}.$ 
\end{proof}

\section{The Other Direction: From Standard Isoperimetry to Log-Sobolev}\label{sec:From-Isoperimetry-to}

We now prove the other direction (i.e., the upper bound) in Theorem
\ref{thm:equivalence}. Recall $\bar{\Lambda}$ defined around \eqref{eq:-82}. 
\begin{thm}[From Isoperimetry to Nonlinear Log-Sobolev]
\label{thm:asymisoper-1} Under Condition \ref{cond:convexity},
for any $\alpha>0$, 
\begin{align}
\bar{\Lambda}(\alpha) & \le\sqrt{\breve{\Theta}(\alpha)}.\label{eq:LB-1}
\end{align}
\end{thm}

We provide two distinct proofs for this theorem. One is to use the
tools of the variant isoperimetry and the famous HWI inequality. The
other one is to apply the coarea formula and the large deviations
theory. 
\begin{proof}[First Proof of Theorem \eqref{thm:asymisoper-1}]
We first prove the first inequality in \eqref{eq:LB-1}. Let $A=A_{n}$
be a set approximately attaining $\Gamma_{n}(e^{-n\alpha},n\tau)$
(given in \eqref{eq:Gamma-2}) within a factor $1-\epsilon$ where
$\epsilon>0$, i.e., $A$ has probability $e^{-n\alpha}$ and $\nu_{n}(A^{\sqrt{n\tau}})\ge(1-\epsilon)\Gamma_{n}(e^{-n\alpha},n\tau)$.
Let $r:=\sqrt{n\tau}$. Since $f(t):=\ln\nu_{n}(A^{t})$ is monotone,
it is differentiable almost everywhere on $(0,r)$. As a basic property
of monotone functions, $f'$ is integrable over $[0,r]$ and $\int_{0}^{r}f'(t)\d t\le f(r)-f(0)$
(see e.g., \cite[Corollary 4 on p.113]{royden2010real}), which implies
that there is some $r_{0}\in(0,r)$ such that $\frac{f(r)-f(0)}{r}\ge f'(r_{0}),$i.e.,
\[
\frac{1}{r}\ln[\nu_{n}(A^{r})/\nu_{n}(A)]\ge\frac{\nu_{n}^{+}(A^{r_{0}})}{\nu_{n}(A^{r_{0}})}.
\]
Since $\nu_{n}^{+}(A^{r_{0}})\ge I_{n}(\nu_{n}(A^{r_{0}}))$ and $I_{n}(a)/a$
is non-increasing in $a$, it holds that 
\[
\frac{1}{r}\ln[\nu_{n}(A^{r})/\nu_{n}(A)]\ge\frac{I_{n}(\nu_{n}(A^{r_{0}}))}{\nu_{n}(A^{r_{0}})}\ge\frac{I_{n}(\nu_{n}(A^{r}))}{\nu_{n}(A^{r})}.
\]

Noting that $r=\sqrt{n\tau}$ and letting $n\to\infty$, by Statement
2 in Theorem \ref{thm:LD}, we obtain for any $\tau>0$, 
\begin{align}
\frac{\alpha-\psi(\alpha,\tau)}{\sqrt{\tau}} & \ge\lim_{n\to\infty}\frac{\ln[\nu_{n}(A_{n}^{\sqrt{n\tau}})/\nu_{n}(A_{n})]}{n\sqrt{\tau}}\ge\limsup_{n\to\infty}\frac{I_{n}(a_{n,\tau})}{a_{n,\tau}\sqrt{n}},\label{eq:-42}
\end{align}
where $a_{n,\tau}=\nu_{n}(A_{n}^{\sqrt{n\tau}})$ and $\psi$ is defined
in \eqref{eq:-41}.

Since $A$ approximately attains $\Gamma_{n}(e^{-n\alpha},n\tau)$
within a factor $1-\epsilon$, by Statement 2 in Theorem \ref{thm:LD}
again, it holds that $\liminf_{n\to\infty}-\frac{1}{n}\ln a_{n,\tau}\ge\psi(\alpha,\tau).$
Hence, \eqref{eq:-42} implies that for any $\delta>0$, $\bar{\Lambda}(\psi(\alpha,\tau)-\delta)\le\frac{\alpha-\psi(\alpha,\tau)}{\sqrt{\tau}}.$
Taking $\tau\downarrow0$ yields 
\begin{align}
\limsup_{\tau\downarrow0}\bar{\Lambda}(\psi(\alpha,\tau)-\delta)\le\xi(\alpha) & :=\limsup_{\tau\downarrow0}\frac{\alpha-\psi(\alpha,\tau)}{\sqrt{\tau}}.\label{eq:xi-2}
\end{align}

We can assume that $\liminf_{\tau\downarrow0}\psi(\alpha,\tau)=\alpha$,
since otherwise, $\xi(\alpha)=\infty$ and nothing is needed to prove.
Hence, given any $\delta>0$, we assume that $\psi(\alpha,\tau)>\alpha-\delta$
for all sufficiently small $\tau$, which, combined with the inequality
above, implies 
\begin{equation}
\bar{\Lambda}(\alpha-2\delta)\le\xi(\alpha).\label{eq:-43}
\end{equation}

We next prove $\xi(\alpha)\le\sqrt{\breve{\Theta}(\alpha)}$. Substituting
the expression of $\psi(\alpha,\tau)$ into the bound above yields
that 
\begin{align*}
\xi(\alpha) & =\limsup_{\tau\downarrow0}\frac{\alpha-\psi(\alpha,\tau)}{\sqrt{\tau}}\\
 & \le\lim_{\tau\downarrow0}\inf_{\pi_{XW}:D(\pi_{X|W}\|\nu|\pi_{W})\le\alpha}\\
 & \qquad\sup_{\pi_{Y|W}:\W(\pi_{X|W},\pi_{Y|W}|\pi_{W})\le\sqrt{\tau}}\frac{\alpha-D(\pi_{Y|W}\|\nu|\pi_{W})}{\W(\pi_{X|W},\pi_{Y|W}|\pi_{W})}\\
 & \le\inf_{\pi_{XW}:D(\pi_{X|W}\|\nu|\pi_{W})=\alpha}\lim_{\tau\downarrow0}\\
 & \qquad\sup_{\pi_{Y|W}:\W(\pi_{X|W},\pi_{Y|W}|\pi_{W})\le\sqrt{\tau}}\frac{D(\pi_{X|W}\|\nu|\pi_{W})-D(\pi_{Y|W}\|\nu|\pi_{W})}{\W(\pi_{X|W},\pi_{Y|W}|\pi_{W})},
\end{align*}
where in the last line, we swap the infimization and the limit, and
we restrict the inequality constraint in the infimization to be an
equality.

Denote $\alpha_{i}=D(\pi_{X|W=i}\|\nu),\beta_{i}=D(\pi_{Y|W=i}\|\nu)$,
and $\gamma_{i}=\W(\pi_{X|W=i},\pi_{Y|W=i})$, $i\in\{0,1\}$. Then,
by the Cauchy--Schwarz inequality, 
\begin{align*}
 & D(\pi_{X|W}\|\nu|\pi_{W})-D(\pi_{Y|W}\|\nu|\pi_{W})=\mathbb{E}_{\pi}[\alpha_{W}-\beta_{W}]\\
 & \qquad\le\sqrt{\mathbb{E}_{\pi}[\left(\frac{\alpha_{W}-\beta_{W}}{\gamma_{W}}\right)^{2}]\mathbb{E}_{\pi}[\gamma_{W}^{2}]}.
\end{align*}
So, 
\begin{align}
\xi(\alpha) & \le\inf_{\pi_{XW}:D(\pi_{X|W}\|\nu|\pi_{W})=\alpha}\lim_{\tau\downarrow0}\sup_{\pi_{Y|W}:\mathbb{E}_{\pi}[\gamma_{W}^{2}]\le\tau}\sqrt{\mathbb{E}_{\pi}[\left(\frac{\alpha_{W}-\beta_{W}}{\gamma_{W}}\right)^{2}]}.\label{eq:-70}
\end{align}
Since $\pi_{Y|W=i}=\pi_{X|W=i}$ is always a feasible solution to
the supremum above, we have $\beta_{i}\le\alpha_{i},i\in\{0,1\}$.

Define the descending slope of the relative entropy w.r.t. the Wasserstein
metric as 
\begin{align*}
|\partial^{-}D|(\mu\|\nu) & :=\limsup_{\pi\to\mu}\frac{[D(\mu\|\nu)-D(\pi\|\nu)]^{+}}{\W(\pi,\mu)}\\
 & =\lim_{\tau\downarrow0}\sup_{\pi:0<\W(\pi,\mu)\le\tau}\frac{[D(\mu\|\nu)-D(\pi\|\nu)]^{+}}{\W(\pi,\mu)}.
\end{align*}
Using this notation, we rewrite \eqref{eq:-70} as 
\begin{align*}
\xi(\alpha) & \le\inf_{\pi_{XW}:D(\pi_{X|W}\|\nu|\pi_{W})=\alpha}\sqrt{\mathbb{E}_{W\sim\pi_{W}}[\left(|\partial^{-}D|(\pi_{X|W}\|\nu)\right)^{2}]}.
\end{align*}

We now need the HWI inequality for the Riemannian manifold setting
\cite[Corollary 20.13]{villani2008optimal} \cite{otto2000generalization}. 
\begin{thm}[{{{{HWI Inequality \cite[Corollary 20.13]{villani2008optimal}}}}}]
Let $\M$ be a Riemannian manifold equipped with a reference probability
measure $\nu=e^{-V}\mathrm{vol}\in\P_{2}^{\mathrm{ac}}(\M),V\in C^{2}(\M)$.
Assume that Condition \ref{cond:convexity} holds. Then, for any two
probability measures $\mu_{0}=\rho_{0}\nu,\mu_{1}=\rho_{1}\nu$ in
$\P_{2}^{\mathrm{ac}}(\M)$ such that $D(\mu_{1}\|\nu)<\infty$ and
$\rho_{0}$ is Lipschitz, it holds that 
\[
D(\mu_{0}\|\nu)-D(\mu_{1}\|\nu)\le\W(\mu_{0},\mu_{1})\sqrt{I(\mu_{0}\|\nu)}.
\]
\end{thm}

This theorem implies $|\partial^{-}D|(\mu\|\nu)\le\sqrt{I(\mu\|\nu)}.$
In fact, this is an equality even for Polish metric probability measure
space satisfying $\mathrm{\mathrm{C}D}(K,\infty)$; see \cite[Theorem 9.3]{ambrosio2014calculus}.
We restrict $\pi_{X|W=i},i\in\{0,1\}$ to be in $\P_{2}^{\mathrm{ac}}(\M)$.
Therefore, 
\begin{align*}
\xi(\alpha) & \le\inf_{\substack{\pi_{XW}:\pi_{X|W=i}\in\P_{2}^{\mathrm{ac}}(\M),\forall i,\\
D(\pi_{X|W}\|\nu|\pi_{W})=\alpha
}
}\sqrt{I(\pi_{X|W}\|\nu|\pi_{W})}=\sqrt{\breve{\Theta}(\alpha)}.
\end{align*}
Combining this with \eqref{eq:-43} yields that $\bar{\Lambda}(\alpha-2\delta)\le\xi(\alpha)\le\sqrt{\breve{\Theta}(\alpha)}.$
That is, for $\alpha,\delta>0$, $\bar{\Lambda}(\alpha)\le\xi(\alpha+\delta)\le\sqrt{\breve{\Theta}(\alpha+\delta)}.$
Since $\breve{\Theta}(t)$ is continuous in $t>0$, letting $\delta\downarrow0$
yields $\bar{\Lambda}(\alpha)\le\sqrt{\breve{\Theta}(\alpha)}.$ 
\end{proof}
\begin{proof}[Second Proof of Theorem \eqref{thm:asymisoper-1}]
By an approximation argument, we may assume that the probability
measures $\mu$ in the nonlinear log-Sobolev inequality have smooth
density $f:=\d\mu/\d\nu$. Obviously, $\int f\,d\nu=1$. Let $f^{\otimes n}$
be the $n$-fold product of $f$ with itself, which satisfies $\d\mu_{n}/\d\nu_{n}=f^{\otimes n}$
with $\mu_{n}=\mu^{\otimes n}$. Then, the coarea formula (see e.g.,
\cite[Exercise III.12]{chavel2006riemannian}) states: 
\begin{equation}
\int|\nabla f^{\otimes n}|\,d\nu_{n}=\int_{0}^{\infty}\left(\int_{C_{s}}\varphi_{n}(\mathbf{x})\,\d\mathrm{\area}(\mathbf{x})\right)\d s\label{eq:-19}
\end{equation}
where $\varphi_{n}(\mathbf{x})=e^{-\sum_{i=1}^{n}V(x_{i})}$ be the
density function of $\nu_{n}$ w.r.t. the volume form on the manifold
$\M^{n}$, $C_{s}=\{\mathbf{x}\in\M^{n}:f^{\otimes n}(\mathbf{x})=s\}$
be the level set of the function $f^{\otimes n}$, $\d\mathrm{\area}$
be the induced volume form (Hausdorff measure of dimension equal to
the dimension of $\M^{n}$ minus $1$) on the level set $C_{s}$,
and $|\nabla f^{\otimes n}(\mathbf{x})|$ be the norm of the gradient
of $f^{\otimes n}$ at the point $\mathbf{x}\in\M^{n}$. For brevity,
we denote $D:=D(\mu\|\nu)$.

Observe that 
\begin{align}
 & \int_{0}^{\infty}\left(\int_{C_{s}}\varphi_{n}(\mathbf{x})\,\d\mathrm{\area}(\mathbf{x})\right)\d s\ge\int_{e^{n(D-\delta)}}^{\infty}\left(\int_{C_{s}}\varphi_{n}(\mathbf{x})\,\d\mathrm{\area}(\mathbf{x})\right)\d s\nonumber \\
 & \qquad=\int_{e^{n(D-\delta)}}^{\infty}\nu_{n}^{+}\{f^{\otimes n}\geq s\}\d s\ge\eta_{n}\int_{e^{n(D-\delta)}}^{\infty}\nu_{n}\{f^{\otimes n}\geq s\}\d s,\label{eq:-29}
\end{align}
where 
\[
\eta_{n}:=\inf_{\substack{s\ge e^{n(D-\delta)}:\\
\nu_{n}\{f^{\otimes n}\geq s\}>0
}
}\frac{\nu_{n}^{+}\{f^{\otimes n}\geq s\}}{\nu_{n}\{f^{\otimes n}\geq s\}}.
\]

Note that 
\begin{align}
 & \int_{e^{n(D-\delta)}}^{\infty}\nu_{n}\{f^{\otimes n}\geq s\}\d s=\int_{0}^{\infty}\int1\{f^{\otimes n}\geq s,\,s\ge e^{n(D-\delta)}\}\d\nu_{n}\d s\nonumber \\
 & \qquad=\int(f^{\otimes n}-e^{n(D-\delta)})\cdot1\{f^{\otimes n}\ge e^{n(D-\delta)}\}\d\nu_{n}=\gamma_{1,n}-\gamma_{2,n},\label{eq:-13}
\end{align}
where 
\begin{align*}
\gamma_{1,n} & :=\int f^{\otimes n}\cdot1\{f^{\otimes n}\ge e^{n(D-\delta)}\}\d\nu_{n}=\mu_{n}\{f^{\otimes n}\ge e^{n(D-\delta)}\},
\end{align*}
and 
\[
\gamma_{2,n}:=e^{n(D-\delta)}\nu_{n}\{f^{\otimes n}\ge e^{n(D-\delta)}\}.
\]

By the law of large numbers, $\gamma_{1,n}\to1$ as $n\to\infty$.
By the large deviations theory (or the asymptotic equipartition property
in information theory), 
\[
\limsup_{n\to\infty}-\frac{1}{n}\ln\gamma_{2,n}=\sup_{t>0}t\left(D_{1-t}(\mu\|\nu)-(D(\mu\|\nu)-\delta)\right)>0,
\]
where $D_{1-t}(\mu\|\nu)=-\frac{1}{t}\ln\int f^{1-t}d\nu,\;f=\d\mu/\d\nu$
is the $(1-t)$-Rényi divergence of $\mu$ w.r.t. $\nu$. Therefore,
substituting these into \eqref{eq:-13} yields $\int_{e^{n(D-\delta)}}^{\infty}\nu_{n}\{f^{\otimes n}\geq s\}\d s\to1,\textrm{ as }n\to\infty,$
which, combined with \eqref{eq:-29}, implies 
\begin{align*}
\int_{0}^{\infty}\left(\int_{C_{s}}\varphi_{n}(\mathbf{x})\,\d\mathrm{\area}(\mathbf{x})\right)\d s & \ge(1+o(1))\eta_{n}.
\end{align*}

We now lower bound $\eta_{n}$. Observe that 
\begin{align*}
\eta_{n} & \ge\inf_{\substack{s\ge e^{n(D-\delta)}:\\
\nu_{n}\{f^{\otimes n}\geq s\}>0
}
}\frac{I_{n}(\nu_{n}\{f^{\otimes n}\geq s\})}{\nu_{n}\{f^{\otimes n}\geq s\}}=\frac{I_{n}(a_{n})}{a_{n}},
\end{align*}
where $a_{n}:=\nu_{n}\{f^{\otimes n}\geq e^{n(D-\delta)}\}$, and
the equality follows since $I_{n}(a)/a$ is non-increasing in $a$.
Hence, 
\begin{align*}
\int_{0}^{\infty}\left(\int_{C_{s}}\varphi_{n}(\mathbf{x})\,\d\mathrm{\area}(\mathbf{x})\right)\d s & \ge(1+o(1))\frac{I_{n}(a_{n})}{a_{n}}.
\end{align*}

One can see that 
\begin{align*}
a_{n} & =\int_{\frac{\d\mu_{n}}{\d\nu_{n}}\ge e^{n(D-\delta)}}\d\nu_{n}\le\int_{\frac{\d\mu_{n}}{\d\nu_{n}}\ge e^{n(D-\delta)}}e^{-n(D-\delta)}\d\mu_{n}\le e^{-n(D-\delta)}=:b_{n}.
\end{align*}
By again the fact that $I_{n}(a)/a$ is non-increasing in $a$, we
have 
\[
\limsup_{n\to\infty}\frac{I_{n}(a_{n})}{a_{n}\sqrt{n}}\ge\limsup_{n\to\infty}\frac{I_{n}(b_{n})}{b_{n}\sqrt{n}}=\bar{\Lambda}(D-\delta).
\]

Therefore, for any $\delta>0$, 
\begin{align}
\limsup_{n\to\infty}\frac{1}{\sqrt{n}}\int_{0}^{\infty}\left(\int_{C_{s}}\varphi_{n}(\mathbf{x})\,\d\mathrm{\area}(\mathbf{x})\right)\d s & \ge\bar{\Lambda}(D(\mu\|\nu)-\delta).\label{eq:-21}
\end{align}

On the other hand, by the Cauchy--Schwarz inequality, 
\begin{equation}
\int|\nabla f^{\otimes n}|\,\d\nu_{n}\le\sqrt{\int f^{\otimes n}\,\d\nu_{n}\,\cdot\int\frac{|\nabla f^{\otimes n}|^{2}}{f^{\otimes n}}\,\d\nu_{n}}=\sqrt{nI(\mu\|\nu)}.\label{eq:-20}
\end{equation}

Combining \eqref{eq:-19}, \eqref{eq:-21}, and \eqref{eq:-20} yields
that $\sqrt{I(\mu\|\nu)}\ge\bar{\Lambda}(D(\mu\|\nu)-\delta),\,\forall\mu.$
Noting that $\bar{\Lambda}$ is non-decreasing, we have $\sqrt{\Theta(\alpha)}\ge\bar{\Lambda}(\alpha-\delta).$
That is, 
\begin{equation}
\bar{\Lambda}(\alpha)\le\sqrt{\Theta(\alpha+\delta)}.\label{eq:-72}
\end{equation}

By substitution $\nu\leftarrow\nu^{\otimes k}$, we define 
\[
\bar{\Lambda}_{k}(\alpha):=\limsup_{n\to\infty}\frac{I_{kn}(e^{-kn\alpha})}{e^{-kn\alpha}\sqrt{n}}.
\]
Noting that both $n\mapsto I_{n}(a)$ and $a\mapsto I_{n}(a)/a$ are
non-increasing, it holds that for $kn\le n'<k(n+1)$, 
\[
\frac{I_{k(n+1)}(e^{-kn\alpha})}{e^{-kn\alpha}\sqrt{k(n+1)}}\le\frac{I_{n'}(e^{-n'\alpha})}{e^{-n'\alpha}\sqrt{n'}}\le\frac{I_{kn}(e^{-k(n+1)\alpha})}{e^{-k(n+1)\alpha}\sqrt{kn}}.
\]
Taking $n\to\infty$ yields that for any $\delta>0$, $\frac{\bar{\Lambda}_{k}(\alpha-\delta)}{\sqrt{k}}\le\bar{\Lambda}(\alpha)\le\frac{\bar{\Lambda}_{k}(\alpha+\delta)}{\sqrt{k}}.$
Applying \eqref{eq:-72} to $(\M^{k},\nu^{\otimes k})$ yields that
for all $k$, 
\begin{equation}
\bar{\Lambda}(\alpha)\le\frac{\bar{\Lambda}_{k}(\alpha+\delta)}{\sqrt{k}}\le\sqrt{\frac{\Theta_{k}(k(\alpha+2\delta))}{k}},\label{eq:-41-1}
\end{equation}
where $\Theta_{k}$ is defined for $(\M^{k},\nu^{\otimes k})$ and
given by 
\[
\Theta_{k}(\alpha):=\inf_{\mu\in\P_{2}^{\mathrm{ac}}(\M):D(\mu\|\nu^{\otimes k})\ge\alpha}I(\mu\|\nu^{\otimes k}).
\]

We now claim that 
\begin{equation}
\lim_{k\to\infty}\frac{\Theta_{k}(k\beta)}{k}=\breve{\Theta}(\beta).\label{eq:-62}
\end{equation}
This is because, on one hand, by the tensorization property of nonlinear
log-Sobolev inequality, for all $k$ and $\beta\ge0$, $\frac{1}{k}\Theta_{k}(k\beta)\ge\breve{\Theta}(\beta);$
on the other hand, by setting the function $f$ in the $k$-dimensional
nonlinear log-Sobolev inequality to be the product form $f(x)=\prod_{i=1}^{j}f_{1}(x_{i})\prod_{i=j+1}^{k}f_{2}(x_{i})$
for some functions $f_{1}$ and $f_{2}$ defined on $\M$, the lower
bound $\breve{\Theta}$ can be asymptotically approached.

Combining \eqref{eq:-41-1} and \eqref{eq:-62} yields $\bar{\Lambda}(\alpha)\le\sqrt{\breve{\Theta}(\alpha+2\delta)}$.
Letting $\delta\downarrow0$ yields the desired result $\bar{\Lambda}(\alpha)\le\sqrt{\breve{\Theta}(\alpha)}$. 
\end{proof}
Combining this theorem with Theorem \ref{thm:asymisoper} yields that
$\underline{\Lambda}(\alpha)=\bar{\Lambda}(\alpha)=\sqrt{\breve{\Theta}(\alpha)},$
completing the proof of our main result, Theorem \ref{thm:equivalence}.

\section*{Appendix A: Proof of Statement 2 of Theorem \ref{thm:equivalence}}\label{sec:Appendix-A:-Proof}

Our proof for $\Lambda(\alpha)\le\sqrt{\breve{\Theta}(\alpha)}$ still
works. We now show that if Condition \ref{cond:integrability} does
not hold, then $\breve{\Theta}(\alpha)=0$ for any $\alpha>0$. We
prove this by contradiction.

Suppose that $\breve{\Theta}(\alpha_{0})>0$ for some $\alpha_{0}>0$.
Then, $\breve{\Theta}(\alpha)\ge C\alpha$ for $\alpha\ge\alpha_{0}$
where $C=\frac{\breve{\Theta}(\alpha_{0})}{\alpha_{0}}>0$. Thus,
\begin{equation}
C(D(\mu\|\nu)-\alpha_{0})\leq I(\mu\|\nu),\;\forall\mu\in\P_{2}^{\mathrm{ac}}(\M).\label{eq:-74-1-1}
\end{equation}
Equivalently, for any smooth function $f\in L^{2}(\M,\nu)$: 
\[
\int f^{2}\ln f^{2}\d\nu-\left(\int f^{2}\d\nu\right)\ln\left(\int f^{2}\d\nu\right)-\alpha_{0}\int f^{2}\d\nu\le\frac{4}{C}\int|\nabla f|^{2}\d\nu
\]

We now use Herbst's argument to derive exponential integrability
from this log-Sobolev inequality. Let $F(x)=\ensuremath{d(x,x_{0})}$
be the distance function. By Rademacher's theorem, $|\nabla F|\le1$
almost everywhere. Define the moment generating function for $F$:
$\psi(\lambda)=\int e^{\lambda F}\d\nu$. We assume $\psi(\lambda)<\infty$
for all $\lambda\ge0$.

Apply the LSI to the function $f^{2}=e^{\lambda F}$. Then $|\nabla f|^{2}=\frac{\lambda^{2}}{4}e^{\lambda F}|\nabla F|^{2}$.
Substituting this into the LSI: 
\[
\int\lambda Fe^{\lambda F}\d\nu-\psi(\lambda)\ln\psi(\lambda)-\alpha_{0}\psi(\lambda)\le\frac{4}{C}\int\frac{\lambda^{2}}{4}e^{\lambda F}|\nabla F|^{2}\d\nu.
\]
Since $|\nabla F|\le1$, we have: 
\[
\frac{\lambda\psi^{\prime}(\lambda)-\psi(\lambda)\ln\psi(\lambda)-\alpha_{0}\psi(\lambda)}{\lambda^{2}\psi(\lambda)}\le\frac{1}{C}.
\]

Notice that the left-hand side is the derivative of the function $\lambda\mapsto\frac{\ln\psi(\lambda)+\alpha_{0}}{\lambda}$,
i.e., $\frac{\d}{\d\lambda}\left(\frac{\ln\psi(\lambda)+\alpha_{0}}{\lambda}\right)\le\frac{1}{C}$.
Integrating from a constant $\lambda_{0}$ to $\lambda$ yields $\frac{\ln\psi(\lambda)+\alpha_{0}}{\lambda}-C_{1}\le\frac{\lambda}{C}$
where $C_{1}=\frac{\ln\psi(\lambda_{0})+\alpha_{0}}{\lambda_{0}}$.
Thus, 
\[
\psi(\lambda)\le e^{\lambda^{2}/C+C_{1}\lambda-\alpha_{0}}.
\]

Note that in the proof above, we assume $\psi(\lambda)<\infty$ for
all $\lambda\ge0$. This is in fact unnecessary since we can obtain
the same inequality by substituting $F\leftarrow F_{n}=\min\{F,n\}$
in the proof above and letting $n\to\infty$.

Using Markov's inequality, we obtain the tail estimate for any $r>0$:
\[
\nu(F\ge r)=\nu(e^{\lambda F}\ge e^{\lambda r})\le\frac{\mathbb{E}[e^{\lambda F}]}{e^{\lambda r}}\le e^{\lambda^{2}/C+C_{1}\lambda-\alpha_{0}-\lambda r}.
\]
To get the tightest bound, we minimize the exponent with respect to
$\lambda$. Setting $\lambda=C(r-C_{1})/2$ gives $\nu(F\ge r)\le e^{-\frac{C(r-C_{1})^{2}}{4}-\alpha_{0}}$.

We now use this tail estimate to evaluate $\mathbb{E}_{\nu}[e^{\alpha F^{2}}]$.
We use the layer-cake representation $\mathbb{E}_{\nu}[e^{\alpha F^{2}}]=\int_{0}^{\infty}\nu(e^{\alpha F^{2}}>t)\d t$.
Let $t=e^{\alpha r^{2}}$, which implies $r=\sqrt{\frac{\ln t}{\alpha}}$.
Changing variables $\mathbb{E}_{\nu}[e^{\alpha F^{2}}]=1+\int_{0}^{\infty}\nu(F>r)\cdot2\alpha re^{\alpha r^{2}}\d r$.
Substitute the tail bound into the integral: 
\[
\mathbb{E}_{\nu}[e^{\alpha F^{2}}]\le1+\int_{0}^{\infty}2e^{-\frac{C(r-C_{1})^{2}}{4}-\alpha_{0}}\cdot2\alpha re^{\alpha r^{2}}\d r=1+4\alpha\int_{0}^{\infty}re^{\alpha r^{2}-\frac{C(r-C_{1})^{2}}{4}-\alpha_{0}}\d r.
\]
Obviously, the integral converges to a finite value for any $\alpha<\frac{C}{4}$.
This contradicts with that Condition \ref{cond:integrability} does
not hold. Hence, $\breve{\Theta}(\alpha)=0$ for any $\alpha>0$.

\section*{Appendix B: Proof of Theorem \ref{thm:approximate}}\label{sec:Appendix-B:-Proof}

Observe that $\{\nu_{m}^{\otimes n}\}$ converges to $\nu^{\otimes n}$
in total-variation distance and in addition $\nu_{m}^{\otimes n}(B)\ge(1-1/m)^{n}\nu^{\otimes n}(B)$
for any Borel set $B$. E. Milman \cite[Theorem 6.10]{milman2009role}
showed that in this case, for all $a\in[0,1]$, $I_{\nu^{\otimes n}}(a)=\lim_{m\to\infty}I_{\nu_{m}^{\otimes n}}(a),$
and consequently, $I_{\nu^{\otimes n}}$ is concave. Let $A$ be an
isoperimetric minimizer for $\nu_{m}^{\otimes n}$ in the standard
sense. By Theorem \ref{thm:isop-isop}, for any $r>0$, 
\[
\frac{I_{\nu_{m}^{\otimes n}}(a)}{a}\ge\frac{1}{r}\ln\frac{\nu_{m}^{\otimes n}(A^{r})}{a}\ge\frac{1}{r}\ln\frac{(1-1/m)^{n}\nu^{\otimes n}(A^{r})}{a}\ge\frac{1}{r}\ln\frac{(1-1/m)^{n}\Gamma_{\nu^{\otimes n}}(a,r)}{a}.
\]
Taking the limit as $m\to\infty$, we obtain $\frac{I_{\nu^{\otimes n}}(a)}{a}\ge\frac{1}{r}\ln\frac{\Gamma_{\nu^{\otimes n}}(a,r)}{a}.$
Note that $\mathrm{RCD}(0,\infty)$ is satisfied by $\nu$, and thus,
if Condition \ref{cond:integrability} is additionally satisfied,
then Theorem \ref{thm:LD2} is valid for $\nu$. The remaining part
of the proof is exactly the same as Theorem \ref{thm:equivalence}.

As for the other direction, noting that $I_{\nu^{\otimes n}}$ is
concave for every $n$ and $(\M,d,\nu)$ is a $\mathrm{RCD}(0,\infty)$-space,
our first proof given in Section \ref{sec:From-Isoperimetry-to} still
works for this space, but with the HWI inequality replaced by \cite[Theorem 9.3]{ambrosio2014calculus}.
Hence, $\Lambda=\sqrt{\breve{\Theta}}$ holds.

If Condition \ref{cond:integrability} is not satisfied, then $\Lambda\le\sqrt{\breve{\Theta}}$
still holds, and thus, by the same proof in Appendix A, $\Lambda=\sqrt{\breve{\Theta}}=0$.

\section*{Appendix C: Proof of Theorem \ref{thm:comparison}}\label{sec:Proof-of-Theorem}

In this section, we now prove the remaining unproven parts of Theorem
\ref{thm:comparison}.

\subsection*{Appendix C-1. Inequality (b) }

We first prove Inequality (b) in Theorem \ref{thm:comparison}. Due
to Theorem \ref{thm:DensityConcave-2}, or more precisely, due to
(3.1) and (4.2) in \cite{milman2010isoperimetric} but with $K_{CD}\ge0$,
it holds that for any $n$, 
\[
\frac{1}{2}-a\le I_{n}(a)\int_{0}^{r(a)}\exp\left(\frac{I_{n}(a)}{a}t-\frac{K_{CD}}{2}t^{2}\right)\d t,
\]
where $r(a)=\sqrt{\frac{2}{K_{C}}\ln\frac{1}{a}}$. Taking limits
as $n\to\infty$, we have 
\begin{equation}
\frac{1}{2}-a\le I_{\inf}(a)\int_{0}^{r(a)}\exp\left(\frac{I_{\inf}(a)}{a}t-\frac{K_{CD}}{2}t^{2}\right)\d t.\label{eq:-25}
\end{equation}
By definition, $\sqrt{K_{IS}^{-}}=\liminf_{a\to0}\frac{I_{\inf}(a)}{a\sqrt{2\ln\frac{1}{a}}}$,
which implies that for any $\epsilon>0$, there is a decreasing sequence
$a_{k}$ converging to $0$ as $k\to\infty$ such that for all sufficiently
large $k$, 
\begin{equation}
I_{\inf}(a_{k})\le a_{k}\sqrt{2(K_{IS}^{-}+\epsilon)\ln\frac{1}{a_{k}}}.\label{eq:-26}
\end{equation}
Substituting $a_{k}$ and the upper bound in \eqref{eq:-26} into
\eqref{eq:-25} yields 
\[
\frac{1}{2a_{k}}-1\le\sqrt{2(K_{IS}^{-}+\epsilon)\ln\frac{1}{a_{k}}}\int_{0}^{r(a_{k})}\exp\left(\sqrt{2(K_{IS}^{-}+\epsilon)\ln\frac{1}{a_{k}}}t-\frac{K_{CD}}{2}t^{2}\right)\d t.
\]
Denote $\alpha_{k}=\ln\frac{1}{a_{k}}$, which diverges to $+\infty$.
Substituting it to the inequality above and performing change of variables,
we then have 
\begin{align*}
\frac{1}{2}e^{\alpha_{k}}-1 & \le\sqrt{2(K_{IS}^{-}+\epsilon)}\alpha_{k}\int_{0}^{\sqrt{\frac{2}{K_{C}}}}\exp\left(\sqrt{2(K_{IS}^{-}+\epsilon)}\alpha_{k}t-\frac{K_{CD}}{2}\alpha_{k}t^{2}\right)\d t.
\end{align*}
Note that $K_{C}\ge K_{IS}^{-}\ge K_{CD}$, which implies $K_{C}\left(K_{IS}^{-}+\epsilon\right)>K_{CD}$.
By Laplace's method, under this condition, the last line above is
\[
\sim\frac{\sqrt{\left(K_{IS}^{-}+\epsilon\right)K_{C}}}{\sqrt{\left(K_{IS}^{-}+\epsilon\right)K_{C}}-K_{CD}}\exp\left((2\sqrt{\frac{K_{IS}^{-}+\epsilon}{K_{C}}}-\frac{K_{CD}}{K_{C}})\alpha_{k}\right)
\]
as $k\to\infty$. It is asymptotically larger than $\frac{1}{2}e^{\alpha_{k}}-1$
only if $2\sqrt{\frac{K_{IS}^{-}+\epsilon}{K_{C}}}-\frac{K_{CD}}{K_{C}}\ge1$,
i.e., $K_{IS}^{-}\ge\left(\frac{1+K_{CD}/K_{C}}{2}\right)^{2}K_{C}$
with letting $\epsilon\to0$.

\subsection*{Appendix C-2. Inequality (k)}

Inequality (k) is exactly the following inequality: for any $n\ge1$,
\begin{equation}
\liminf_{a\to0}\frac{I_{n}(a)}{a\sqrt{2\ln\frac{1}{a}}}\ge\sqrt{K_{LS}^{+}}.\label{eq:-16}
\end{equation}

Given any $n$, to prove \eqref{eq:-16}, it suffices to show that
for any sequence $\{a_{i}\}_{i\in\mathbb{N}}$ of positive numbers
that converges to $0$, it holds that 
\begin{equation}
\liminf_{i\to0}\frac{I_{n}(a_{i})}{a_{i}\sqrt{2\ln\frac{1}{a_{i}}}}\ge\sqrt{K_{LS}^{+}}.\label{eq:-17}
\end{equation}

Let $\alpha>0$. For each $a_{i}$, one can find an integer $n_{i}$
such that $e^{-(n_{i}+1)\alpha}<a_{i}\le e^{-n_{i}\alpha}$. As $i\to\infty$,
it holds that $n_{i}\to\infty$. So, by the monotonicity of $t\mapsto\frac{I_{n}(t)}{t}$
\cite[Proposition 3.1]{milman2010isoperimetric} and $t\mapsto\sqrt{\ln\frac{1}{t}}$,
\[
\frac{I_{n}(a_{i})}{a_{i}\sqrt{2\ln\frac{1}{a_{i}}}}\ge\frac{I_{n}(e^{-n_{i}\alpha})}{e^{-n_{i}\alpha}\sqrt{2(n_{i}+1)\alpha}}.
\]
Taking $\liminf_{i\to0}$, we obtain 
\begin{align*}
\liminf_{i\to0}\frac{I_{n}(a_{i})}{a_{i}\sqrt{2\ln\frac{1}{a_{i}}}} & \ge\liminf_{i\to0}\frac{I_{n}(e^{-n_{i}\alpha})}{e^{-n_{i}\alpha}\sqrt{2(n_{i}+1)\alpha}}=\liminf_{i\to0}\frac{I_{n}(e^{-n_{i}\alpha})}{e^{-n_{i}\alpha}\sqrt{2n_{i}\alpha}}\\
 & \ge\liminf_{i\to0}\frac{I_{n_{i}}(e^{-n_{i}\alpha})}{e^{-n_{i}\alpha}\sqrt{2n_{i}\alpha}}\ge\sqrt{\frac{\breve{\Theta}(\alpha)}{2\alpha}},
\end{align*}
where the second inequality is due to that $I_{n}(a)\ge I_{m}(a)$
for all $m\ge n$ and any $a$. Since $\alpha>0$ is arbitrary, we
obtain \eqref{eq:-17}.

\section*{Appendix D: Proof of Theorem \ref{thm:I_inf} }\label{sec:Proof-of-Theorem-1}

We next compare the dimension-free isoperimetric profile $I_{\inf}$
and the infinite-dimensional isoperimetric profile $I_{\infty}$.
Recall that the dimension-free isoperimetric profile $I_{\inf}(a):=\lim_{n\to\infty}I_{n}(a)=\inf_{n\ge1}I_{n}(a),$
and $I_{\infty}$ is the isoperimetric profile for the space $(\M^{\infty},\nu^{\otimes\infty})$. 
\begin{thm}
\label{thm:In}Let $(\mathcal{X},d,\nu)$ be a locally compact Polish
probability metric space. For $n\in\mathbb{N}\cup\{\infty\}$, let
\begin{equation}
I_{n}(a):=\inf_{\mathrm{closed}\,A:\nu_{n}(A)=a}\nu_{n}^{+}(A),\;a\in[0,1]\label{eq:I}
\end{equation}
be the isoperimetric profile for the $n$-product space $(\mathcal{X}^{n},\nu^{\otimes n})$,
where the enlargement of a set is defined under $d_{n}(\mathbf{x},\mathbf{y}):=\sqrt{\sum_{i=1}^{n}d(x_{i},y_{i})^{2}}$.
If $I_{\inf}$ given above is continuous at $a$, then $I_{\inf}(a)=I_{\infty}(a).$ 
\end{thm}

\begin{rem}
We restrict the set $A$ in \eqref{eq:I} to be closed, in order to
ensure that $A^{r}\to A$ and hence $\nu_{n}(A^{r})\to\nu_{n}(A)$
as $r\downarrow0$. This restriction usually does not affect the isoperimetric
profile, e.g., for the smooth Riemannian manifold setting with an
absolutely continuous reference measure \cite{morgan2003regularity}. 
\end{rem}

In our Riemannian manifold setting, under Condition \ref{cond:convexity},
the isoperimetric profiles $I_{n},n\in\mathbb{N}$ are all concave,
which implies their pointwise minimum $I_{\inf}$ is concave as well
and hence also continuous on $(0,1)$. So, Theorem \ref{thm:In} implies
$I_{\inf}(a)=I_{\infty}(a)$ for all $a\in(0,1)$. Moreover, obviously
$I_{\inf}(a)=I_{\infty}(a)=0$ for $a\in\{0,1\}$. Thus, $I_{\inf}$
is continuous on $[0,1]$. Theorem \ref{thm:I_inf} is a consequence
of Theorem \ref{thm:In}. Hence it remains to prove Theorem \ref{thm:In}.

\subsection*{Appendix D-1. Proof of Theorem \ref{thm:In} }

In the following, we also use $\mathcal{X}^{\mathbb{N}}$ to denote
$\mathcal{X}^{\infty}$, the countably infinite product of $\mathcal{X}$
with itself. Recall that $\nu_{n}=\nu^{\otimes n}$ with $n\in\mathbb{N}\cup\{\infty\}$.

Observe that $I_{\infty}(a)\le I_{\inf}(a)$, since for any $n$ and
$B\subseteq\mathcal{X}^{n}$, $B\times\mathcal{X}^{\mathbb{N}}$ in
$\mathcal{X}^{\mathbb{N}}$ has the same measure and the same boundary
as those of $B$ in $\mathcal{X}^{n}$. We next prove that $I_{\infty}(a)\ge I_{\inf}(a)$.
We follow similar proof steps for Gaussian measures in \cite[Theorem 2.2.4]{gine2021mathematical}.
We first make the following claim. 
\begin{claim}
\label{claim:Let--be}Let $\pi_{n}\colon\mathcal{X}^{\mathbb{N}}\to\mathcal{X}^{n}$
be the projection $\pi_{n}(\mathbf{x})=\pi_{n}(x_{k}:k\in\mathbb{N})=(x_{1},\dots,x_{n})$.
Then it holds that: (a) $\nu_{n}=\nu_{\infty}\circ\pi_{n}^{-1}$,
(b) if $K\subseteq\mathcal{X}^{\mathbb{N}}$ is compact in the product
topology, then $K=\bigcap_{n=1}^{\infty}\pi_{n}^{-1}(\pi_{n}(K))$,
(c) $K^{r}:=\bigcup_{\mathbf{x}\in K}B_{r]}(\mathbf{x})$ is compact
in the product topology if $K$ is, where $B_{r]}(\mathbf{x})$ is
the closed ball with center $\mathbf{x}$ and radius $r$ under the
metric $d_{\infty}$. 
\end{claim}

\begin{proof}[Proof of Claim \ref{claim:Let--be}]
(a) is obvious.

(b) Inclusion $K\subseteq\bigcap_{n}\pi_{n}^{-1}(\pi_{n}(K))$: By
definition, $\pi_{n}(\mathbf{x})\in\pi_{n}(K)$ for any $\mathbf{x}\in K$,
so $\mathbf{x}\in\pi_{n}^{-1}(\pi_{n}(K))$ for all $n$. Thus, $\mathbf{x}\in\bigcap_{n=1}^{\infty}\pi_{n}^{-1}(\pi_{n}(K)).$

Inclusion $\bigcap_{n}\pi_{n}^{-1}(\pi_{n}(K))\subseteq K$: Suppose
$\mathbf{x}\in\bigcap_{n=1}^{\infty}\pi_{n}^{-1}(\pi_{n}(K))$. Then
for each $n$, $\pi_{n}(\mathbf{x})\in\pi_{n}(K)$, so there exists
$\mathbf{x}^{(n)}\in K$ such that $\pi_{n}(\mathbf{x}^{(n)})=\pi_{n}(\mathbf{x})$.
Since $K$ is compact in the product topology of $\mathcal{X}^{\mathbb{N}}$,
it is closed and sequentially compact. The sequence $(\mathbf{x}^{(n)})$
lies in $K$. The projections $\pi_{n}(\mathbf{x}^{(n)})=\pi_{n}(\mathbf{x})$
imply that the first $n$ coordinates of $\mathbf{x}^{(n)}$ match
those of $\mathbf{x}$. This implies $\mathbf{x}^{(n)}\to\mathbf{x}$
in the product topology. Since $K$ is closed, $\mathbf{x}\in K$.
Hence, $K=\bigcap_{n=1}^{\infty}\pi_{n}^{-1}(\pi_{n}(K)).$

(c) We first note that $d_{\infty}$ might be infinity between two
points, and it does not induce the product topology (topology of coordinate-wise
convergence) on $\mathcal{X}^{\mathbb{N}}$. In fact, the convergence
in $d_{\infty}$ implies the convergence in product topology.

Let $K\subseteq\prod_{i=1}^{\infty}X_{i}$ be compact, and define
\[
C_{i}:=\pi_{i}(K)\subseteq X_{i},
\]
which is compact in $X_{i}$ since $\pi_{i}$ is continuous. For any
$\mathbf{y}\in K^{r}$, there exists $\mathbf{x}\in K$ such that
$\sum_{i=1}^{\infty}d(x_{i},y_{i})^{2}\le r^{2}.$ Hence, for each
coordinate $i$, $d(y_{i},C_{i})\le d(y_{i},x_{i})\le r,$ because
$x_{i}\in C_{i}$. This implies 
\[
y_{i}\in C_{i}^{r}:=B_{r]}(C_{i}):=\{z\in X_{i}:d(z,C_{i})\le r\},
\]
which is compact since $\mathcal{X}$ is locally compact. Therefore,
$K^{r}\subseteq\prod_{i=1}^{\infty}C_{i}^{r},$ and by Tychonoff's
theorem, $\prod_{i=1}^{\infty}C_{i}^{r}$ is compact in the product
topology.

We now prove that $K^{r}$ is closed in the product topology. Fix
$\mathbf{x}\in K$. For each $N\in\mathbb{N}$, define 
\[
B_{r]}^{(N)}(\mathbf{x}):=\left\{ \mathbf{y}:\sum_{i=1}^{N}d(x_{i},y_{i})^{2}\le r^{2}\right\} \textrm{ and }K_{N}^{r}:=\bigcup_{\mathbf{x}\in K}B_{r]}^{(N)}(\mathbf{x}).
\]
Obviously, $K_{N}^{r}$ is non-increasing in $N$, and $K^{r}\subseteq K_{N}^{r}$.
So, $K^{r}\subseteq G:=\bigcap_{N=1}^{\infty}K_{N}^{r}$.

On the other hand, let $\mathbf{y}\in G$. It holds that for each
$N$, there is $\mathbf{x}^{(N)}\in K$ such that $\sum_{i=1}^{N}d(x_{i}^{(N)},y_{i})^{2}\le r^{2}$.
Since $K$ is compact in the product topology, there is a subsequence
$\mathbf{x}^{(N_{k})},k\in\mathbb{N}$ which converges in the product
topology to some point $\mathbf{z}\in K$. By the definition of product
topology, for each $i$, $d(x_{i}^{(N_{k})},y_{i})\to d(z_{i},y_{i})$
as $k\to\infty$. Given any $m\in\mathbb{N}$, for sufficiently large
$k$ such that $N_{k}\ge m$, by triangle inequality and Minkowski
inequality, it holds that 
\begin{align*}
 & \sum_{i=1}^{m}d(z_{i},y_{i})^{2}\le\sum_{i=1}^{m}(d(x_{i}^{(N_{k})},y_{i})+d(x_{i}^{(N_{k})},z_{i}))^{2}\\
 & \qquad\le\Big(\sqrt{\sum_{i=1}^{m}d(x_{i}^{(N_{k})},y_{i})^{2}}+\sqrt{\sum_{i=1}^{m}d(x_{i}^{(N_{k})},z_{i})^{2}}\Big)^{2}\le\Big(r+\sqrt{\sum_{i=1}^{m}d(x_{i}^{(N_{k})},z_{i})^{2}}\Big)^{2},
\end{align*}
Taking limits for both sides as $k\to\infty$, we obtain that $\sum_{i=1}^{m}d(z_{i},y_{i})^{2}\le r^{2}.$
Letting $m\to\infty$ yields $\sum_{i=1}^{\infty}d(z_{i},y_{i})^{2}\le r^{2}.$
That is, $\mathbf{y}\in K^{r}$. So, $G\subseteq K^{r}$. Hence, $K^{r}=G=\bigcap_{N=1}^{\infty}K_{N}^{r}$.

By a similar convergence argument, the set $K_{N}^{r}$ is closed
in the product topology. Note that $K^{r}=\bigcap_{N=1}^{\infty}K_{N}^{r}$
an intersection of closed sets, hence $K^{r}$ is closed in the product
topology. Since $K^{r}$ is closed in the compact set $\prod_{i=1}^{\infty}C_{i}^{r}$,
it follows that $K^{r}$ is compact. 
\end{proof}
We now turn back to prove $I_{\infty}(a)\ge I_{\inf}(a)$. Let $A\subseteq\mathcal{X}^{\mathbb{N}}$
be a compact set. Let $A_{n}=\pi_{n}^{-1}(\pi_{n}(A))$. Then, by
the claim above, $A=\lim_{n\to\infty}A_{n}$, and $A^{r}$ is compact
as well. By the claim above again, we have $A^{r}=\lim_{n\to\infty}B_{n}$,
where $B_{n}=\pi_{n}^{-1}(\pi_{n}(A^{r}))=\pi_{n}^{-1}(\pi_{n}(A)^{r})=A_{n}^{r}$.
So, by the continuity of probability measures, 
\begin{align*}
\frac{\nu_{\infty}(A^{r})-\nu_{\infty}(A)}{r} & =\lim_{n\to\infty}\frac{\nu_{\infty}(A_{n}^{r})-\nu_{\infty}(A_{n})}{r}=\lim_{n\to\infty}\frac{\nu_{n}(E_{n}^{r})-\nu_{n}(E_{n})}{r},
\end{align*}
where $E_{n}:=\pi_{n}(A)$.

Since $f(t):=\nu_{n}(E_{n}^{r})$ is monotone, it is differentiable
almost everywhere on $(0,r)$. Hence, by the mean value theorem, there
is some $t_{n}\in(0,r)$ such that $f$ is differentiable at $t_{n}$,
and moreover, $\frac{f(r)-f(0)}{r}\ge f'(t_{n}),$ i.e., $\frac{\nu_{n}(E_{n}^{r})-\nu_{n}(E_{n})}{r}\ge\nu_{n}^{+}(E_{n}^{t_{n}}).$
Hence, 
\begin{align*}
\frac{\nu_{\infty}(A^{r})-\nu_{\infty}(A)}{r} & \ge\limsup_{n\to\infty}\nu_{n}^{+}(E_{n}^{t_{n}})\ge\limsup_{n\to\infty}I_{n}(\nu_{n}(E_{n}^{t_{n}}))\ge\limsup_{n\to\infty}I_{\inf}(\nu_{n}(E_{n}^{t_{n}})).
\end{align*}

Observe that $\nu_{\infty}(A_{n})=\nu_{n}(E_{n})\le\nu_{n}(E_{n}^{t_{n}})\le\nu_{n}(E_{n}^{r})=\nu_{\infty}(B_{n}).$
Taking $n\to\infty$, $\nu_{\infty}(A)\le\limsup_{n\to\infty}\nu_{n}(E_{n}^{t_{n}})\le\nu_{\infty}(A^{r}).$
Therefore, 
\begin{align*}
\frac{\nu_{\infty}(A^{r})-\nu_{\infty}(A)}{r} & \ge\inf_{\nu_{\infty}(A)\le a\le\nu_{\infty}(A^{r})}I_{\inf}(a).
\end{align*}
Taking $r\downarrow0$ and by the continuity of $I_{\inf}$ at $a$,
$\nu_{\infty}^{+}(A)\ge I_{\inf}(a).$ So, Theorem \ref{thm:In} holds
for compact sets.

We now extend the theorem to closed sets. Let $A\subseteq\mathcal{X}^{\mathbb{N}}$
be a closed set. Since $\mathcal{X}^{\mathbb{N}}$ equipped with the
product topology is Polish, it follows that $\nu_{\infty}$ is tight.
That is, given any $\epsilon>0$, there is a compact set $K_{\epsilon}\subseteq\mathcal{X}^{\mathbb{N}}$
such that $\nu_{\infty}(K_{\epsilon}^{c})\le\epsilon$. Let $A_{\epsilon}=A\cap K_{\epsilon}$
which is compact. So, $\nu_{\infty}(A_{\epsilon})\le\nu_{\infty}(A)\le\nu_{\infty}(A_{\epsilon})+\epsilon$
and $\nu_{\infty}(A_{\epsilon}^{r})\le\nu_{\infty}(A^{r})\le\nu_{\infty}(A_{\epsilon}^{r})+\epsilon$.
So, 
\begin{align*}
 & \frac{\nu_{\infty}(A^{r})-\nu_{\infty}(A)}{r}\ge\frac{\nu_{\infty}(A_{\epsilon}^{r})-\nu_{\infty}(A_{\epsilon})-\epsilon}{r}\\
 & \qquad\ge-\frac{\epsilon}{r}+\inf_{\nu_{\infty}(A_{\epsilon})\le a\le\nu_{\infty}(A_{\epsilon}^{r})}I_{\inf}(a)\ge-\frac{\epsilon}{r}+\inf_{\nu_{\infty}(A)-\epsilon\le a\le\nu_{\infty}(A^{r})}I_{\inf}(a).
\end{align*}
Letting first $\epsilon\to0$ and then $r\to0$, by the continuity
of $I_{\inf}$ at $a$, we obtain $\nu_{\infty}^{+}(A)\ge I_{\inf}(a).$
Hence, $I_{\infty}(a)\ge I_{\inf}(a)$.

\section*{Appendix E: Proof of Theorem \ref{thm:clt}}\label{sec:Proof-of-Theorem-clt}

By an approximation argument of Otto and Villani \cite[Appendix A]{otto2000generalization},
without affecting the optimal constant in \eqref{eq:poincare}, it
suffices to consider $f\in W^{1,2}(\M)$ satisfying the following
smoothness and boundedness conditions: $f$ is bounded away from $0$
and $\infty$ on $\M$, and $f$ is smooth and $|\nabla f|^{2}$ is
bounded on $\M$. Without loss of generality, we can further assume
$\mathbb{E}_{\nu}\left[f\right]=0$ and $\mathbb{E}_{\nu}\left[f^{2}\right]=1$,
since otherwise we can normalize $f$ which neither affects the optimal
constant in \eqref{eq:poincare}. Among all this kind of functions,
in the following we let $f$ be the approximately optimal function
attaining the optimal constant in \eqref{eq:poincare}, i.e., $\mathbb{E}[|\nabla f|^{2}]\le(K_{P}+\delta)\cdot\mathrm{Var}_{\nu}(f)=K_{P}+\delta$
for $\delta>0$.

For such $f$, we consider the set $A_{t}=H_{t\sqrt{n}}(f)=\left\{ \mathbf{x}:g_{n}(\mathbf{x})\leq t\right\} $,
where $g_{n}(\mathbf{x})=\frac{1}{\sqrt{n}}\sum_{i=1}^{n}f(x_{i})$.
We next determine the asymptotics of the volume and the perimeter
of this set.

By the coarea formula, 
\begin{equation}
\int_{|g_{n}-t|\le\delta}|\nabla g_{n}|\,d\nu_{n}=\int_{t-\delta}^{t+\delta}\left(\int_{g_{n}=s}\varphi_{n}(\mathbf{x})\,\d\mathrm{\area}(\mathbf{x})\right)\d s\label{eq:-19-2}
\end{equation}
where $\varphi_{n}(\mathbf{x})=e^{-\sum_{i=1}^{n}V(x_{i})}$ is the
density of $\nu_{n}$, $\d\mathrm{\area}$ is the induced volume form
on the level set $\{g_{n}=s\}$, and $|\nabla g_{n}(\mathbf{x})|$
be the norm of the gradient of $g_{n}$ at the point $\mathbf{x}\in\M^{n}$.
Thus, by the mean value theorem, the RHS above is larger than or equal
to 
\begin{equation}
2\delta\cdot\int_{g_{n}=t_{n}}\varphi_{n}(\mathbf{x})\,\d\mathrm{\area}(\mathbf{x})\ge2\delta\cdot\nu_{n}^{+}(A_{t_{n}}),\label{eq:-8}
\end{equation}
for some $t_{n}\in[t-\delta,t+\delta]$. Denoting $\mathbf{X}\sim\nu_{n}$,
the LHS above is equal to 
\begin{equation}
\mathbb{P}\{|g_{n}(\mathbf{X})-t|\le\delta\}\times\mathbb{E}\left[|\nabla g_{n}(\mathbf{X})|\mid|g_{n}(\mathbf{X})-t|\le\delta\right].\label{eq:-9}
\end{equation}

By the central limit theorem, as $n\to\infty$, 
\begin{equation}
\mathbb{P}\{|g_{n}(\mathbf{X})-t|\le\delta\}\to\Phi(t+\delta)-\Phi(t-\delta).\label{eq:-10}
\end{equation}

We next determine the limit of $\mathbb{E}\left[|\nabla g_{n}(\mathbf{X})|\mid|g_{n}(\mathbf{X})-t|\le\delta\right]$.
Note that $|\nabla g_{n}(\mathbf{X})|^{2}=\frac{1}{n}\sum_{i=1}^{n}|\nabla f|^{2}(X_{i})$.
Let $\beta:=\mathbb{E}[|\nabla f|^{2}]\le K_{P}+\delta.$ By the weak
law of large number, $\lim_{n\to\infty}\mathbb{P}\{\big||\nabla g_{n}(\mathbf{X})|^{2}-\beta\big|\le\delta\}\to1.$
By basic probabilistic inequalities, for $\delta>0$, 
\begin{align*}
\mathbb{P}\{\big||\nabla g_{n}(\mathbf{X})|^{2}-\beta\big|\le\delta,|g_{n}(\mathbf{X})-t|\le\delta\} & \ge\mathbb{P}\{|g_{n}(\mathbf{X})-t|\le\delta\}-\mathbb{P}\{\big||\nabla g_{n}(\mathbf{X})|^{2}-\beta\big|>\delta\},
\end{align*}
which implies 
\begin{align*}
\mathbb{P}\{\big||\nabla g_{n}(\mathbf{X})|^{2}-\beta\big|\le\delta\big||g_{n}(\mathbf{X})-t|\le\delta\} & \ge1-\frac{\mathbb{P}\{\big||\nabla g_{n}(\mathbf{X})|^{2}-\beta\big|>\delta\}}{\mathbb{P}\{|g_{n}(\mathbf{X})-t|\le\delta\}}.
\end{align*}
By taking limit as $n\to\infty$, $\lim_{n\to\infty}\mathbb{P}\{\big||\nabla g_{n}(\mathbf{X})|^{2}-\beta\big|\le\delta\big||g_{n}(\mathbf{X})-t|\le\delta\}=1.$
Since $|\nabla f|^{2}$ is bounded by assumption (and thus, so is
$|\nabla g_{n}|^{2}$), we obtain 
\begin{align}
\lim_{n\to\infty}\mathbb{E}\left[|\nabla g_{n}(\mathbf{X})|\mid|g_{n}(\mathbf{X})-t|\le\delta\right] & \le\sqrt{\beta+\delta}\le\sqrt{K_{P}+\delta}+\delta.\label{eq:-12}
\end{align}

Combining \eqref{eq:-19-2}, \eqref{eq:-8}, \eqref{eq:-9}, \eqref{eq:-10},
and \eqref{eq:-12} yields 
\[
\limsup_{n\to\infty}\nu_{n}^{+}(A_{t_{n}})\le\frac{\Phi(t+\delta)-\Phi(t-\delta)}{2\delta}\left(\sqrt{K_{P}+\delta}+\delta\right).
\]
Letting $\delta\to0$ yields 
\begin{equation}
\limsup_{\delta\to0}\limsup_{n\to\infty}\nu_{n}^{+}(A_{t_{n}})\le\sqrt{K_{P}}\phi(t).\label{eq:-14}
\end{equation}

By the central limit theorem, $g_{n}(\mathbf{X})$ converges in distribution
to $Z\sim\mathcal{N}(0,1)$. Thus, the volume $\nu_{n}(A_{t_{n}})$
is asymptotically sandwiched between $\Phi(t-\delta)$ and $\Phi(t+\delta)$
as $n\to\infty$, which implies 
\begin{equation}
\lim_{\delta\to0}\lim_{n\to\infty}\nu_{n}(A_{t_{n}})=\Phi(t).\label{eq:-22}
\end{equation}

Equations \eqref{eq:-14} and \eqref{eq:-22} implies that there is
a sequence of sets $B_{n}\subseteq\M^{n}$ such that $\lim_{n\to\infty}\nu_{n}^{+}(B_{n})\le\sqrt{K_{P}}\phi(t)$
and $\lim_{n\to\infty}\nu_{n}(B_{n})=\Phi(t)$. By continuity of $I_{\inf}$,
it holds that $I_{\inf}(a)\le\sqrt{K_{P}}\cdot I_{\mathrm{G}}(a)$.

\section*{Acknowledgements}

This work was supported by the National Key Research and Development
Program of China under grant 2023YFA1009604 and the NSFC under grant
62101286.

 \bibliographystyle{abbrv}
\bibliography{ref}

\end{document}